\documentclass[oneside,11pt,reqno]{amsart}
\usepackage{etex}
\usepackage{stmaryrd}
\usepackage{bbm}
\usepackage{mathrsfs}
\usepackage{enumerate}
\usepackage{latexsym,amsxtra}
\usepackage[dvips]{graphicx}
\usepackage{dsfont}
\usepackage{slashed}
\usepackage[all,cmtip]{xy}
\usepackage{amscd,graphics}
\usepackage{amsmath,amsfonts,amsthm,amssymb}
\usepackage{latexsym,amsmath}
\usepackage{graphicx,psfrag}
\usepackage[usenames, dvipsnames]{color}
\usepackage{soul}
\usepackage{hyperref}
\usepackage[normalem]{ulem}
\usepackage{dcpic,pictex}
\usepackage{pstricks}
\usepackage{pst-all}
\usepackage{pstricks-add}

\newtheorem{thm}{Theorem}[section]
\newtheorem{lem}[thm]{Lemma}
\newtheorem{cor}[thm]{Corollary}
\newtheorem{pro}[thm]{Proposition}

\theoremstyle{definition}
\newtheorem{defi}[thm]{Definition}
\newtheorem{ex}[thm]{Example}
\newtheorem{rmk}[thm]{Remark}

\newcounter{choicegluing}

\DeclareMathOperator{\inv}{inv}
\DeclareMathOperator{\inti}{int}
\DeclareMathOperator{\im}{im}

\DeclareMathOperator{\morp}{mor}
\DeclareMathOperator{\ob}{ob}
\DeclareMathOperator{\grad}{grad}

\DeclareMathOperator{\id}{id}
\DeclareMathOperator{\dist}{dist}
\DeclareMathOperator{\pic}{Pic^{0}}

\def\Z{\mathbb{Z}}
\def\Q{\mathbb{Q}}
\def\R{\mathbb{R}}
\def\C{\mathbb{C}}

\newcommand{\swf}{\underline{\operatorname{swf}}}

\DeclareFontFamily{U}{mathx}{\hyphenchar\font45}
\DeclareFontShape{U}{mathx}{m}{n}{<-> mathx10}{}
\DeclareSymbolFont{mathx}{U}{mathx}{m}{n}
\DeclareMathAccent{\widebar}{0}{mathx}{"73}

\title[Unfolded Seiberg--Witten Floer spectra, II]
{Unfolded Seiberg--Witten Floer spectra, II: Relative invariants and the gluing theorem}

\author{Tirasan Khandhawit}              
\date{} 
\address{Department of Mathematics, Faculty of Science, Mahidol University, 272 Rama VI Road, Thung Phayathai, Ratchathewi, Bangkok, 10400, Thailand  }
\email{tirasan.kha@mahidol.ac.th}

\author{Jianfeng Lin}
\address{ Yau Mathematical Sciences Center, Jingzhai,  Tsinghua University, Haidian District, Beijing, 100084,  China}
\email{linjian5477@mail.tsinghua.edu.cn}

\author{Hirofumi Sasahira}
\address{Faculty of Mathematics, Kyushu University, 
744, Motooka, Nishi-ku, Fukuoka, 819-0395
Japan}
\email{hsasahira@math.kyushu-u.ac.jp}

\begin{document}

\vskip 0.3 truecm

\maketitle 

\begin{abstract}
We use the construction of unfolded Seiberg--Witten Floer spectra of general 3-manifolds defined in our previous paper to extend the notion of relative Bauer--Furuta invariants to general 4-manifolds with boundary.  One of the main purposes of this paper is to give a detailed proof of the gluing theorem for the relative invariants. 
\end{abstract}


\section {Introduction}   \label{section introduction}

The Bauer--Furuta invariant, which was introduced in \cite{Bauer-Furuta1}, can be regarded a stable homotopy refinement of the Seiberg--Witten invariants \cite{Witten} for closed 4-manifolds. The invariant takes value in equivariant stable cohomotopy group of spheres and can give interesting applications in 4-manifold theory, such as the 10/8-theorem \cite{Furuta108}. On the other hand, the Seiberg--Witten Floer spectrum, which was first introduced by Manolescu for rational homology 3-spheres \cite{Manolescu1}, can be regarded as a stable homotopy refinement of monopole Floer homology \cite{Kronheimer-Mrowka}.
Using this Seiberg--Witten Floer spectrum, Manolescu extended the notion of the Bauer--Furuta invariant to 4-manifolds whose boundary are rational homology spheres. This ``relative'' invariant takes value in a stable cohomotopy group of the Seiberg--Witten Floer spectrum of the boundary manifold. 

In the previous paper \cite{KLS1}, we have constructed the ``unfolded'' version of Seiberg--Witten Floer spectrum for general 3-manifolds.
It is then natural to extend Manolescu's construction of relative Bauer--Furuta invariant to arbitrary 4-manifolds with boundary. Recall that the unfolded spectrum comes with two variations: type-A and type-R. Consequently, the unfolded relative Bauer--Furuta invariant will also come with type-A and type-R variations. Here the letters ``A'' and ``R'' stand for the notions ``attractor'' and ``repeller'' in dynamical system. The type-A invariant is an object in a category $\mathfrak{S}$ of ind-spectra, and the type-R invariant is an object in a category $\mathfrak{G}^{*}$ of pro-spectra.

{\begin{defi}
Let $(Y_{\text{in}},\mathfrak{s}_{\text{in}})$ and $(Y_{\text{out}},\mathfrak{s}_{\text{out}})$ be two oriented (but not necessarily connected) spin$^{c}$ 3--manifolds. A spin$^{c}$ cobordism from $(Y_{\text{in}},\mathfrak{s}_{\text{in}})$ to $(Y_{\text{out}},\mathfrak{s}_{\text{out}})$ is a compact, oriented spin$^{c}$ 4--manifold $(X,\hat{\mathfrak{s}})$ together with a \emph{fixed diffeomorphism}
$$
(\partial X, \hat{\mathfrak{s}}|_{\partial X})\cong (-Y_{\text{in}},\mathfrak{s}_{\text{in}})\sqcup (Y_{\text{out}},\mathfrak{s}_{\text{out}})
$$ 
between spin$^{c}$ 3--manifolds. Here we use the canonical identification between the set of spin$^{c}$ structures on $Y_{\text{in}}$ and the corresponding set for $-Y_{\text{in}}$, the orientation reversal of $Y_{\text{in}}$. When $Y_{\text{in}}=\emptyset$, we also call $(X, \hat{\mathfrak{s}})$ a spin$^{c}$ manifold with boundary $(Y_{\text{out}},\mathfrak{s}_{\text{out}})$.
\end{defi}}

Let $(X,\hat{\mathfrak{s}})$ be a connected spin$^{c}$ cobordism from  $(Y_{\text{in}},\mathfrak{s}_{\text{in}})$ to $(Y_{\text{out}},\mathfrak{s}_{\text{out}})$. We equip $X$ with a Riemannian metric $\hat{g}$ and a spin$^c$ connection $\hat{A}_{0}$. Denote the restriction of $\hat{A}_{0}, \hat{g}$ to $Y_{\text{in}}$ (resp.  $Y_{\text{out}}$) by 
$ A_{\text{in}}\text{ and }g_\text{in}$ (resp. $A_{\text{out}}\text{ and }g_\text{out}$). Then the type-A unfolded relative Bauer--Furuta invariant of $X$ is constructed as a morphism in the stable category $\mathfrak{S} $

\begin{align}
 \underline{\textnormal{bf}}^{A}(X, \hat{\mathfrak{s}};S^{1}) \colon & 
\Sigma^{-(V^{+}_X\oplus  V_{\text{in}})}T  (X,  \hat{\mathfrak{s}};S^{1})\wedge \underline{\textnormal{swf}}_{}^{A}(Y_{\text{in}},\mathfrak{s}_\text{in},A_{\text{in}},
g_\text{in} ; S^1)    \nonumber \\
& \rightarrow \underline{\textnormal{swf}}_{}^{A}(Y_\text{out},\mathfrak{s}_\text{out},A_{\text{out}}, g_\text{out}
; S^1). \nonumber
\end{align}
Here $\underline{\textnormal{swf}}^{A}(Y, \mathfrak{s}, A, g;S^1)$ is the type-A unfolded Seiberg-Witten Floer spectrum  defined in \cite[Definition 5.7]{KLS1}. 

The type-R unfolded relative Bauer--Furuta invariant of $X$  is constructed analogously as a morphism in the stable category $\mathfrak{S}^* $
\begin{align}
\underline{\textnormal{bf}}^{R}(X, \hat{\mathfrak{s}};S^{1}) \colon &
\Sigma^{-(V^{+}_X\oplus  V_{\text{out}})}T  (X,  \hat{\mathfrak{s}};S^{1})\wedge \underline{\textnormal{swf}}_{}^{R}(Y_{\text{in}},\mathfrak{s}_\text{in},A_{\text{in}},
g_\text{in} ; S^1)   \nonumber \\
& \rightarrow \underline{\textnormal{swf}}_{}^{R}(Y_\text{out},\mathfrak{s}_\text{out},A_{\text{out}}, g_\text{out}
; S^1). \nonumber
\end{align}
Here $\underline{\textnormal{swf}}^{R}(Y, \mathfrak{s}, A, g; S^1)$ is the type-R unfolded Seiberg-Witten Floer spectrum defined in \cite[Definition 5.9]{KLS1}.
The object $T  (X,  \hat{\mathfrak{s}};S^{1})$ is the Thom spectrum of virtual index bundle associated to a family of Dirac operators. See Lemma~\ref{lem thombundle} below for the precise definition. We also denote by $V^+_X$ a maximal positive subspace of $\operatorname{Im} (H^2(X, \partial X;\R) \rightarrow H^2(X;\R))$ with respect
to the intersection form, $V_{\text{in}}$ the cokernel of $\iota^* \colon H^{1}(X;\mathbb{R})\rightarrow H^{1}(Y_{\text{in}};\mathbb{R})$
, and $V_{\text{out}}$ similarly.
We refer readers to  
 Section~\ref{sec bfconstuct} and Definition~\ref{def BFtypeA}-\ref{def BFtypeR} for more detailed descriptions of these objects. 
 
 In Section~\ref{sec INVofBF}, we show that these are invariants of the 4-manifold with boundary in the following sense.

\begin{thm}
As one varies $(\hat{g},\hat{A}_{0})$, both domain and target of $\underline{\textnormal{bf}}^{A}(X, \hat{\mathfrak{s}};S^{1})$ are changed by suspending or desuspending with the same number of copies of $\mathbb{C}$; the morphism $\underline{\textnormal{bf}}^{A}(X, \hat{\mathfrak{s}};S^{1})$ is invariant as a stable homotopy class. The same result holds for $\underline{\textnormal{bf}}^{R}(X, \hat{\mathfrak{s}};S^{1})$. Moreover, when $c_{1}(\mathfrak{s}|_{Y})$ is torsion, one can construct further normalizations:
\begin{align}
\underline{\textnormal{BF}}^{A}(X, \hat{\mathfrak{s}};S^{1})  &\colon
\Sigma^{-(V^{+}_X\oplus  V_{\text{in}})}T(X,  \hat{\mathfrak{s}};S^{1})\wedge \underline{\textnormal{SWF}}_{}^{A}(Y_{\text{in}},\mathfrak{s}_\text{in}; S^1)  \nonumber\\ 
& \rightarrow \underline{\textnormal{SWF}}_{}^{A}(Y_\text{out},\mathfrak{s}_\text{out}; S^1). \nonumber
\end{align}
\begin{align}
\underline{\textnormal{BF}}^{R}(X, \hat{\mathfrak{s}};S^{1}) \colon &  
\Sigma^{-(V^{+}_X\oplus  V_{\text{out}})}T(X,  \hat{\mathfrak{s}};S^{1})\wedge \underline{\textnormal{SWF}}_{}^{R}(Y_{\text{in}},\mathfrak{s}_\text{in}; S^1)  \nonumber \\
& \rightarrow \underline{\textnormal{SWF}}_{}^{R}(Y_\text{out},\mathfrak{s}_\text{out}; S^1), \nonumber
\end{align}
which are completely independent of metrics and base connections.
\end{thm}

See \cite[Definition 5.10]{KLS1} for the definitions of the normalized  unfolded Seiberg-Witten Floer spectra $\underline{\textnormal{SWF}}^{A}(Y,\mathfrak{s};S^1)$ and $\underline{\textnormal{SWF}}^{R}(Y,\mathfrak{s};S^1)$,  and the normalized unfolded relative Bauer-Furuta invariants $\underline{\textnormal{BF}}^{A}(X, \hat{\mathfrak{s}}; S^1)$ and $\underline{\textnormal{BF}}^{R}(X, \hat{\mathfrak{s}};S^1)$ will be defined in  Definition \ref{def normalized BFA} and Definition \ref{def BFtypeR} below. 

\begin{rmk} First, we emphasize that our unfolded relative invariant is defined over the \emph{relative} Picard torus 
\[ \operatorname{Pic}^{0}(X,Y) \cong \ker(H^{1}(X;\mathbb{R})\rightarrow
H^{1}(Y;\mathbb{R}))/\ker(H^{1}(X;\mathbb{Z})\rightarrow H^{1}(Y;\mathbb{Z})). \]
Secondly, the choice of labeling each boundary component corresponds to which side its unfolded spectrum will appear in the morphism.
Essentially, $\underline{\textnormal{swf}}_{}^{A}(Y ) $ is the Spanier--Whitehead dual of $\underline{\textnormal{swf}}_{}^{R}(-Y) $ 
and $\underline{\textnormal{bf}}^{A}(X) $ is the same as $\underline{\textnormal{bf}}^{R}(X^{\dagger}) $ where  $X^{\dagger} \colon -Y_{\text{out}}\rightarrow -Y_{\text{in}}$ is the adjoint cobordism of $X \colon Y_{\text{in}} \rightarrow Y_{\text{out}}$. Finally, both $\underline{\textnormal{BF}}^{A}$ and $\underline{\textnormal{BF}}^{R}$ agree with Manolescu's construction when $b_1 (Y) = 0$. In this case, we denote $\underline{\textnormal{BF}}^{A}=\underline{\textnormal{BF}}^{R}$ by $\textnormal{BF}$.
\end{rmk}

\begin{ex} \label{example: relative BF for handles}
Let us consider the case when the 4-manifold $X$ is $S^2 \times D^2 $ or $D^3 \times S^1 $ and  $\hat{\mathfrak{s}}_{0}$ is the unique torsion spin$^{c}$ structure. Its boundary is $S^2 \times S^1$  with torsion spin$^{c}$ structure, whose unfolded spectrum is $S^0 $ by calculation in \cite{KLS1}. In this case, the type-A invariant $\underline{\textnormal{bf}}^{A}(X,\hat{\mathfrak{s}}_{0})$
is a stable homotopy class in $[S^0 ,S^0 ]^{st}_{S^1} \cong \mathbb{Z}$. By classical Hodge theory and \cite[Lemma 3.8]{Bauer-Furuta1},
 we can conclude that $\underline{\textnormal{bf}}^{A}(X,\hat{\mathfrak{s}}_{0}) = 1$, the identity element.
 
On the other hand, $\underline{\textnormal{bf}}^{R}(S^2 \times D^2,\hat{\mathfrak{s}}_{0})$ is a stable homotopy class in $[S^{-1} ,S^0 ]^{st}_{S^1}$ as $V_{\text{out}} = \mathbb{R} $. The group $[S^{-1} ,S^0 ]^{st}_{S^1} $ is trivial, therefore $\underline{\textnormal{bf}}^{R}(S^2 \times D^2,\hat{\mathfrak{s}}_{0}) = 0$, the trivial element.
\end{ex}

One of the main goals of the paper is to prove the gluing theorem for unfolded relative Bauer--Furuta invariants.
When decomposing a 4-manifold $X$ to two pieces along a 3-manifold $Y$, the gluing theorem can express the (relative) Bauer--Furuta invariant of $X$ in term of a ``product'' of relative invariants of the two pieces. The case when $Y=S^3$ was first proved by Bauer \cite{Bauer}  using only invariants of closed 4-manifolds and the positive scalar curvature metric of $S^3$. 
The case when $Y$ is a homology 3-sphere was proved by Manolescu \cite{Manolescu2}. Our setup and argument closely follow and generalize those of Manolescu.

Generally, our gluing theorem works when $Y$ is any 3-manifold.
Some mild homological assumptions will be made. 
These conditions are not essential in the sense that they can be removed under more generalized notion of category and unfolded spectrum (see the upcoming remark for more explanation). We now state the gluing theorem which will reappear in Section~\ref{sec gluingsetup} with more details.

\begin{thm} Let 
\[ 
(X_{0},\hat{\mathfrak{s}}_{0}) \colon (Y_{0},\mathfrak{s}_{0})\rightarrow  (Y_{2},\mathfrak{s}_{2}), \ 
(X_{1},\hat{\mathfrak{s}}_{1}) \colon  (Y_{1},\mathfrak{s}_{1})\rightarrow  (-Y_{2},\mathfrak{s}_{2})
\]
be connected, spin$^{c}$ cobordisms and 
\[
(X,\hat{\mathfrak{s}}) \colon (Y_0,\mathfrak{s}_{0}) \sqcup (Y_1,\mathfrak{s}_{1}) \rightarrow \emptyset 
\]
be the glued cobordism along $Y_2$.
If the following conditions hold
\begin{enumerate}[(i)]
\item \label{item gluecon1} $Y_2$ is connected,
\item \label{item gluecon2} $b_1(Y_0) = b_1 (Y_1) = 0$,
\item \label{item gluecon3} $
\im(H^{1}(X_{0};\mathbb{R})\rightarrow H^{1}(Y_{2};\mathbb{R}))\subset \im(H^{1}(X_{1};\mathbb{R})\rightarrow H^{1}(Y_{2};\mathbb{R})),
$
\end{enumerate}
then, under the natural identification between domains and targets, one has
\begin{equation} \label{eq gluingintro}
\textnormal{BF}(X,\hat{\mathfrak{s}})|_{\pic(X,Y_{2})}= \pmb{\tilde{\epsilon}}(\underline{\textnormal{bf}}^{A}(X_{0},\hat{\mathfrak{s}}_{0}),\underline{\textnormal{bf}}^{R}(X_{1},\hat{\mathfrak{s}}_{1})),
\end{equation}
where $\pmb{\tilde{\epsilon}}(\cdot,\cdot)$ is the Spanier-Whitehead
duality operation defined in Section~\ref{section spanierwhitehead}
and the relative Picard torus $\pic(X,Y_{2}) $ is given by 
\[
\ker(H^{1}(X;\mathbb{R})\rightarrow
H^{1}(Y_{2};\mathbb{R}))/\ker(H^{1}(X;\mathbb{Z})\rightarrow H^{1}(Y_{2};\mathbb{Z})).
\]
\end{thm}

\begin{rmk} The main limitation of the unfolded construction is that one can recover only the partial Bauer--Furuta invariant of $X$ on $\pic(X, Y_{2}) $ rather than on the full Picard torus. Regarding the hypotheses of the theorem,
\begin{itemize} 
\item Condition (\ref{item gluecon2}) is to avoid dealing with type-A and type-R of $\underline{\textnormal{swf}}(Y_0)$, $\underline{\textnormal{swf}}(Y_1)$, and $\textnormal{BF}(X)$. If one tries to extend this direction, a category containing more general kinds of diagrams in $\mathfrak{C} $ will be required 
\item Condition (\ref{item gluecon3}) is to control the harmonic action of the relative gauge groups on $Y_2$.
Otherwise, a more generalized version of unfolded spectrum such as mixture of type-A and type-R will be needed.
\end{itemize}
\end{rmk}

{
Many classical operations (e.g. log transformation, Fintushel--Stern surgery, fiber sum) in 4-dimensional topology involve
gluing and pasting 4-manifolds along 3-manifolds with $b_1>0$ (especially the 3-torus). Our gluing theorem provides a tool
to study the Bauer--Furuta invariant under these operations. The long term goal is to use this idea to compute the Bauer--Furuta
invariant for interesting examples of irreducible 4--manifolds and draw new results beyond the reach of classical Seiberg--Witten invariant.
Here, we mention one of the consequences on classical surgery.

\begin{cor}\label{cor: surgery along loops} 
Let $(X,\hat{\mathfrak{s}})$ be a smooth, oriented, spin$^{c}$ 4-manifold with $b_{1}(\partial X)=0$ or $\partial X=\emptyset$. Let $\gamma$ be an embedded loop with $[\gamma]\neq 0\in H_{1}(X;\mathbb{R})$. If $X' $ is the 4-manifold obtained by surgery along $\gamma$ and $\hat{\mathfrak{s}}'$ is the unique spin$^c$ structure that is isomorphic to $\hat{\mathfrak{s}}$ on $X\setminus \gamma$, then we have
\[ \textnormal{BF}(X',\hat{\mathfrak{s}}') = \textnormal{BF}(X,\hat{\mathfrak{s}})|_{\pic(X,Y)},
\] 
where $Y = S^2 \times S^1$ and $H^{1}(X\mathbb{})\rightarrow
H^{1}(Y\mathbb{}) $ is induced by inclusion of $\gamma$. 
\end{cor}

\begin{proof} Let $N \cong D^3 \times S^1 $ be a tubular neighborhood of $\gamma $ and denote by $Y \cong S^2 \times S^1$ its boundary.
Since $\gamma $ is homologically essential, we can apply the gluing theorem for $X = N \cup_Y (X \setminus N) $ and $X'=(S^2 \times D^2) \cup_Y (X \setminus N) $ to obtain
\[\textnormal{BF}(X,\hat{\mathfrak{s}})|_{\pic(X,Y_{})}= \pmb{\tilde{\epsilon}}(\underline{\textnormal{bf}}^{A}(D^3 \times S^1,\hat{\mathfrak{s}}_{0}),\underline{\textnormal{bf}}^{R}(X \setminus N,\hat{\mathfrak{s}}|_{X \setminus N}))
\]
and
\[\textnormal{BF}(X',\hat{\mathfrak{s}})|_{\pic(X',Y_{})}= \pmb{\tilde{\epsilon}}(\underline{\textnormal{bf}}^{A}(S^2 \times D^2,\hat{\mathfrak{s}}'_{0}),\underline{\textnormal{bf}}^{R}(X
\setminus N,\hat{\mathfrak{s}}|_{X \setminus N})).
\]
Here $\hat{\mathfrak{s}}_{0}$ and $,\hat{\mathfrak{s}}_{0}$ are torsion spin$^{c}$ structures. 
By Example \ref{example: relative BF for handles}, we have $$\underline{\textnormal{bf}}^{A}(D^3 \times S^1,\hat{\mathfrak{s}}_{0}) = \underline{\textnormal{bf}}^{A}(S^2 \times D^2,\hat{\mathfrak{s}}'_{0}) = 1,$$ so $\textnormal{BF}(X,\hat{\mathfrak{s}})|_{\pic(X,Y_{})} = \textnormal{BF}(X',,\hat{\mathfrak{s}}'_{0})|_{\pic(X',Y_{})} $. 
In addition, one can check that $\pic(X',Y_{}) = \pic(X')$ and we recover the full Bauer--Furuta invariant on $(X',\mathfrak{s}')$. 
\end{proof}
\begin{rmk}
Recently, there has been significant amount of attentions \cite{Zemke1,MillerZemke,LevineZemke,Sarkar} on using Floer homology to study ribbon cobordance and ribbon homology cobordisms (homology cobordisms with only $1$-handles and $2$-handles). In particular, Daemi-Lidman-Vela-Vick-Wong \cite{daemi2019obstructions} proved that the Heegaard Floer homology $\widehat{HF}(Y_{1})$ is a summand of  $\widehat{HF}(Y_{2})$ if there is a ribbon homology cobordism from $Y_{1}$ to $Y_{2}$. A central tool in their proof is a surgery formula that describes cobordism maps under surgery along essential loops. We expect that our Corollary \ref{cor: surgery along loops} can be useful in proving a parallel result for Manolescu's Seiberg-Witten Floer spectrum.
\end{rmk}

As another consequence, we can give a vanishing result for the relative Bauer-Furuta invariant. The case of closed manifolds was proved by Fr\o yshov \cite[Theorem~1.1.1]{Froyshov}. See also \cite[Proposition~4.6.5]{Nicolaescu} for an analogous result for the Seiberg--Witten invariant.

\begin{cor}  Let $X$ be a smooth, oriented, 4-manifold containing an embedded 2-sphere $S$ which has trivial self-intersection and is homologically essential (i.e., $[S]\neq 0\in H_{2}(X;\mathbb{R})$). Suppose either $X$ is closed or $b_{1}(\partial X)=0$. Then its (relative) Bauer--Furuta invariant for any spin$^{c}$ structure is trivial.
\end{cor}

\begin{proof} We focus on the case $\partial X\neq \emptyset$. The closed case is similar but simpler. Since its self-intersection is trivial, the sphere $S$ has a neighborhood diffeomorphic to $S^2 \times D^2 $, denoted by $X_1$. We then have a decomposition $X = X_0 \cup_{Y_{0}} X_1 $ where $X_0$ is  $X \setminus X_1 $ and $Y_{0} \cong S^2 \times S^1 $.  

For a spin$^{c}$ structure $\hat{\mathfrak{s}}$ on $X$, there are two possibilities:

Suppose $\langle c_{1}(\hat{\mathfrak{s}}),[S]\rangle \neq 0$. Then the conclusion follows from standard neck-stretching argument: Let $g_{0}$ be a positive scalar curvature metric on $Y_{0}$. We equip $X$ with a Riemannian metric $\hat{g}$ whose restriction to a collar neighborhood of $Y_{0}$ equals the product metric $[-1,1]\times g_{0}$. By stretching the neck of $Y_{0}$, we get a family of Riemannian metrics $\{\hat{g}_{T}\mid T\geq 1\}$ on $X$. Since the 3-dimensional Seiberg-Witten equations on $(Y_{0},\hat{\mathfrak{s}}|_{Y_{0}})$ has no solution, there are no finite type $X$-trajectories for the metric $g_{T}$ when $T$ is large enough (see Definition \ref{defi: X-trajectory}). By the converging result Lemma \ref{convegence of approximated X-trajectories}, there are no $(n,\epsilon)$-approximated $X$-trajectory of length $L$ when $n,L$ are large and $\epsilon$ is small (see Definition \ref{the approximated X-trajectory}). Therefore, by construction of the relative Bauer-Furuta invariant \cite[Page 918]{Manolescu1} (see also (\ref{equation: definition of upsilon})), we see that it is vanishing in this case.

Suppose  $\langle c_{1}(\hat{\mathfrak{s}}),[S]\rangle =0$. Then the conclusion follows from our gluing theorem: Since $S$ is essential, the condition 
$$
\im(H^{1}(X_{0};\mathbb{R})\rightarrow H^{1}(Y_{};\mathbb{R}))\subset \im(H^{1}(X_{1};\mathbb{R})\rightarrow H^{1}(Y_{};\mathbb{R}))$$
 holds because both images are zero. From the Mayer--Vietoris sequence, we have $\pic(X,Y_{}) = \pic(X) $. We now apply the gluing theorem
\[ \textnormal{BF}(X;\hat{\mathfrak{s}})=\pmb{\tilde{\epsilon}} (\underline{\textnormal{bf}}^{A}(X_{0},\hat{\mathfrak{s}}|_{X_{0}}),\underline{\textnormal{bf}}^{R}(X_{1},\mathfrak{s},\hat{\mathfrak{s}}|_{X_{1}})). \]
Since the type-R relative Bauer--Furuta invariant $\underline{\textnormal{bf}}^{R}(S^2 \times D^2,\hat{\mathfrak{s}}|_{X_{0}}) =0$ by Example \ref{example: relative BF for handles},
we conclude that $\textnormal{BF}(X,\hat{\mathfrak{s}})$ is trivial.

\end{proof}

%

}

 \bigskip\noindent\textbf{Acknowledgement:}  The first author is supported by JSPS KAKENHI Grant Number 18K13419. The second author was partially supported by the NSF Grant DMS-1707857. The third author is supported by JSPS KAKENHI Grant Number JP16K17590.

 \section{Summary of constructions and proofs}
 \label{sec outline}
 
 Most of required background in Conley theory is contained in Section \ref{sec Conley}. Background for our stable categories
and Spanier--Whitehead duality is contained in Section \ref{section stablecategory}. We summarize the major constructions here.

 \subsection{Unfolded Seiberg--Witten Floer spectra} \label{subsec Unfolded}
Here we will recall the construction and definition of the unfolded Seiberg--Witten Floer spectrum \cite{KLS1}. Let $Y$ be a closed spin$^c$ 3-manifold (not necessarily connected) with a spinor bundle $S_Y$.
 
 We always work on a Coulomb slice $Coul(Y) = \{ (a, \phi) \in  i \Omega^1 (Y) \oplus \Gamma(S_Y) \mid d^* a =0 \} $ with Sobolev completion.
 With a basepoint chosen on each connected component, we identify the residual gauge group with the based harmonic
gauge group $\mathcal{G}^{h,o}_{Y}\cong H^{1}(Y; \mathbb{Z}) $ acting on $Coul(Y)$. We consider a strip
of balls in $Coul(Y)$ translated by this action 
\begin{align}\label{eq: str}
Str(R)=\{x\in Coul(Y) \mid \exists h\in \mathcal{G}^{h,o}_{Y} \text{ s.t. } \|h\cdot x\|_{L^{2}_{k}}\leq
R\}.
\end{align}
Recall from \cite[Definition 3.1]{KLS1} that a Seiberg-Witten trajectory is called ``finite type'' if it is contained in a bounded region of $Coul(Y)$ in the $L^{2}_{k}$-norm. The boundedness result for 3-manifolds \cite[Theorem~3.2]{KLS1} states that all finite-type Seiberg-Witten trajectories are contained in $Str(R) $ for $R$ sufficiently large.

The basic idea of unfolded construction is to consider increasing sequences of bounded regions in the Coulomb slice. To do this, we choose a basis for $H^{1}(Y;\mathbb{R})$ and use it to identify $i\Omega^{h}(Y)$, the space of imaginary valued harmonic $1$-forms, with $\mathbb{R}^{b_{1}(Y)}$. Under this isomorphism, we let 
$$
p_{\mathcal{H}}=(p_{\mathcal{H}, 1},\cdots, p_{\mathcal{H},b_{1}(Y)}):Coul(Y)\rightarrow \mathbb{R}^{b_{1}(Y)}
$$
be the $L^{2}$-orthogonal projection. Let $\bar{g}:\mathbb{R}\rightarrow \mathbb{R}$ be a certain ``step function'' with small derivative (see \cite[Figure 1]{KLS1}). We consider the function 
$$
g_{j,\pm}=\bar{g}\circ p_{\mathcal{H},j}\pm \mathcal{L}: Coul(Y)\rightarrow \mathbb{R} \text{ for }1\leq j\leq b_{1}(Y).
$$
Here $\mathcal{L}$ denotes the balanced-perturbed Chern-Simons-Dirac functional.
These functions $g_{j,\pm}$ are constructed in such a way that for any real number $\theta$ and any integer $m$, the region 
\begin{align*}
J^{\pm}_{m} := Str(\tilde{R}) \cap \bigcap_{ 1 \leq j \leq b_1} g_{j, \pm}^{-1} (-\infty, \theta + m]
\end{align*}
is bounded. We pick a sequence of finite-dimensional subspaces $V^{\mu_n}_{\lambda_n} $ coming from eigenspaces of the operator $(*d,\slashed{D})$ and define $ J^{n, \pm}_{m} := J^{\pm}_{m}  \cap V^{\mu_n}_{\lambda_n}  $.
 
The main point is that when we choose a generic $\theta$, the region
$J^{n, \pm}_{m}  $ becomes an isolating neighborhood with respect to the approximated Seiberg--Witten flow $\varphi^n $ on $V^{\mu_n}_{\lambda_n}$ when $n$ is large relative to $m$. This is essentially because the perturbations we add on $\pm\mathcal{L}$ have small derivatives. We can now define desuspended Conley indices
\begin{align*}
\begin{split}
& I^{n,+}_{m} = \Sigma^{-\bar{V}^0_{\lambda_n}}I(\inv(J^{n,+}_{m}) ,  \varphi_n), \\
& I^{n,-}_{m} = \Sigma^{-V^0_{\lambda_n}}I(\inv(J^{n,-}_{m}) , \varphi_n)
\end{split}
\end{align*}
as objects in the stable category $\mathfrak{C} $ (see Section~\ref{section stablecategory}). Here $\bar{V}^0_{\lambda_n} $ is the orthogonal complement of the space of harmonic 1-forms in ${V}^0_{\lambda_n}$.  Note that these objects do not depend on $n$ up to canonical isomorphism of the form
\begin{align*}
\tilde{\rho}_{m_{}}^{n,\pm} \colon I^{n,\pm}_{m_{}}(Y_{\text{}}) \rightarrow I^{n+1,\pm}_{m_{1}}(Y_{\text{}}).
\end{align*}
 
The unfolded Seiberg-Witten Floer spectra  are represented by direct and inverse systems in the stable category $\mathfrak{C} $ as follows 
\begin{align} \label{eq SWFdef}
   \begin{split}
       & \swf^{A}(Y) : I^+_{1} \xrightarrow{j_{1}} I^+_2 \xrightarrow{j_{2}} \cdots   \\
       & \swf^{R}(Y) : {I}^-_1 \xleftarrow{\bar{j}_1} {I}^-_2 \xleftarrow{\bar{j}_{2}} \cdots.
   \end{split}
\end{align}
Connecting morphisms in the diagram for $\swf^{A}(Y)  $ are induced by attractor relation while morphisms in $\swf^{R}(Y) $ are induced by repeller relation. More precisely, we have morphisms between desuspended Conley indices
\begin{align*} 
\tilde{i}_{m_{}}^{n,+} \colon I^{n,+}_{m_{}}(Y_{\text{}}) \rightarrow I^{n,+}_{m_{}+1}(Y_{\text{}}) \text{ and } \tilde{i}_{m_{}-1}^{n,-} \colon I^{n,-}_{m_{}}(Y_{\text{}}) \rightarrow I^{n,-}_{m_{}-1}(Y_{\text{}}). 
\end{align*}
Then, the morphisms $j_m, \bar{j}_m$ in (\ref{eq SWFdef}) are given by composition of $\tilde{\rho}_{m_{}}^{n,\pm}$'s and $\tilde{i}_{m_{}}^{n,\pm} $ appropriately.

\subsection{Unfolded Relative Bauer--Furuta invariants}
Let $X$ be a compact, connected, oriented, Riemannian 4--manifold with boundary $Y = -Y_{\text{in}} \sqcup Y_{\text{out}} $. 
To define the invariant, we pick  auxiliary homological data which corresponds to a choice of basis of $H^1 (X ; \mathbb{R}) $ and keeps track of both kernel and image of $ \iota^* \colon H^1 (X ; \mathbb{R}) \rightarrow H^1 (Y ; \mathbb{R})$ (see the list at beginning of Section 5.1). 

In this construction, we use the double Coulomb slice $Coul^{CC}(X)$ as a gauge fixing. The main idea is to find suitable finite-dimensional approximation for the Seiberg--Witten map together with the restriction map
\begin{align*}
(SW, r) \colon Coul^{CC}(X) \rightarrow
L^{2}_{k-1/2}(i\Omega^{2}_{+}(X)\oplus \Gamma(S_{X}^{-}))\oplus Coul(Y).
\end{align*}
Note that there is an action of $H^1 (X ; \mathbb{Z})$ on both sides with restriction on $Coul(Y)$. Compactness of solutions can only be achieved modulo this action.   However, the construction of the unfolded spectra does not behave well under the action of $H^1 (X ; \mathbb{Z}) $ on $Coul(Y) $. This is essentially the reason we can define the unfolded relative invariant only on the relative Picard torus induced from $\ker{\iota^*} $. As one can see in the basic boundedness result (Theorem~\ref{boundedness for X-trajectory}), we need a priori bound on the $\im{i^*} $-part quantified by the projection $\hat{p}_{\beta} $.

We will focus on type-A relative invariant $\underline{\textnormal{bf}}^{A}(X)$. Although it is formulated as a morphism from $\swf^{A}(Y_{\text{in}})$ to $\swf^{A}(Y_{\text{out}})$, the main part of the construction is to obtain maps of the form 
\begin{equation} \label{eq bfmapintro}
\begin{split}
&B(W_{n,\beta})/S(W_{n,\beta})   \\
&  \quad \rightarrow (B(U_{n})/S(U_{n}))\wedge
I(\inv(J^{n,-}_{m_{0}}(-Y_{\text{in}}))) \wedge I(\inv(J^{n,+}_{m_{1}}(Y_{\text{out}}))).
\end{split}
\end{equation}
The left hand side is the Thom space of a finite-dimensional subbundle $W_{n,\beta}$ of the Hilbert bundle
$$
\mathcal{W}_{X}=Coul^{CC}(X)/\ker(H^1(X;\mathbb{Z})\rightarrow H^{1}(Y;\mathbb{Z})),
$$ while $B(U_{n})/S(U_{n})$ is a sphere. We point out  that the right hand side is intuitively $\swf^{R}(-Y_{\text{in}}) \wedge \swf^{A}(Y_{\text{out}})$, which may be viewed as a `mixed'-type unfolded spectrum of $Y$. It is possible to formally consider this in a larger category containing both $\mathfrak{S} $ and $\mathfrak{S}^*$, but we will not pursue this in this paper. Another remark is that $W_{n,\beta}$ has extra constraint $\hat{p}_{\beta, \text{out} }=0$ to control the $\im{i^*} $-part mentioned earlier. The reason we only need the part on $Y_{\text{out}}$ is because we start with a fixed $m_0$ and then choose sufficiently large $m_1$. The order of dependency of parameters is established at the beginning of Section~\ref{sec bfconstuct}.

A notion of pre-index pair (see Section~\ref{section T-tame}) is also required to define the map (\ref{eq bfmapintro}). 
This part closely resembles original Manolescu's construction \cite{Manolescu1} in the case $b_1 (Y)=0 $.
The last step to apply Spanier--Whitehead duality (see Section~\ref{section dualswf}) between $\swf^{R}(-Y_{\text{in}})$ and $\swf^{A}(Y_{\text{in}})$ and define the relative invariant as a morphism in $\mathfrak{S} $.

\subsection{The Gluing theorem} Let  $X_{0} \colon Y_{0}\rightarrow Y_{2}$ and $X_{1} \colon Y_{1}\rightarrow -Y_{2}$ be connected, oriented
cobordisms. We consider the composite cobordism $X = X_0 \cup_{Y_2} X_1 $ glued along $Y_2$ from $Y_0 \sqcup Y_1$ to the empty manifold.

The main technical difficulty of the proof of the gluing theorem is that two different kinds of index pairs arise in the construction. On one hand, to define the relative invariant, we require an index pair $(N_1, N_2)$ to contain a certain pre-index pair $(K_1 , K_2)$. On the other hand, we need a manifold isolating block when dealing with duality morphisms.
In general, a canonical homotopy equivalence between index pairs can be given by flow maps (Theorem~\ref{thm flowmap}), but the formula can sometimes be inconvenient to work with and the common squeeze time $T$ can be arbitrary. 

This is the reason we introduce the concept of $T$-tameness, which is a quantitative refinement of
notions in Conley theory (see Section \ref{section T-tame} and \ref{sec Ttamemfd}).
The flow maps from $T$-tame index pairs can be simplified (Lemma \ref{flow map from tame index pair}).
Most boundedness results in this paper are stated for trajectories with finite length.
As a result, the time parameter $T$, which also corresponds to the length of a cylinder, has a uniform bound during the construction.

The proof of the gluing theorem can be divided to two major parts. The first part, contained in Section \ref{sec deform1st},
involves simplifying the flow maps and duality morphisms. We carefully set up all the parameters needed to explicitly write down  $\pmb{\tilde{\epsilon}} (\underline{\textnormal{bf}}^{A}(X_{0}),\underline{\textnormal{bf}}^{R}(X_{1}))$. 
For instance, we can represent Conley index part of the map as a composition of smash product of flow maps and Spanier--Whitehead duality map
$$\pmb{\tilde{\epsilon}}(\iota_{0},\iota_{1}) \colon K_{0}/S_{0}\wedge
K_{1}/S_{1}\rightarrow \tilde{N}_{0}/\tilde{N}^{+}_{0}\wedge \tilde{N}_{1}/\tilde{N}^{+}_{1}\wedge B^{+}(V^{2}_{n},\bar{\epsilon})$$
given by formula (\ref{gluing with long neck}). (Here $V^{2}_{n}$ is a finite dimensional subspace of $Coul(Y)$ coming from eigenspaces of $(d^{*},\slashed{D})$.)   After two steps, we deform the formula to the one given in Proposition~\ref{deformed pairing}. 

The second part of the proof of the gluing theorem, contained in Section \ref{sec deform2nd}, is to deform Seiberg--Witten
maps on $X_0$ and $X_1$ to the Seiberg--Witten map on $X$.
Many of the arguments here will be similar to Manolescu's proof \cite{Manolescu2} when $b_1(Y_2) = 0$ . The crucial part is to deform gauge fixing with boundary
conditions and harmonic gauge groups on $X_0$ and $X_1$ to those on $X$.
For clarity, we subdivide the deformation to seven steps.
A recurring technique is to move between maps and conditions on the domain (Lemma~\ref{moving map to  domain2}). 
Other ingredients such as stably c-homotopic pairs are contained in Section~\ref{subsec stablyc}.  


\section{Conley Index}
\label{sec Conley}
In this section, we recall basic facts regarding the Conley index theory and develop some further properties we need.  Without any modification, all the results and constructions of this section can be adapted to the $G$-equivariant theory, when  $G$ is a compact Lie group. See \cite{Conley}  and \cite {Salamon} for more details. 
\subsection{Conley theory: definition and basic properties}
Let $\Omega$ be a finite dimensional manifold and $\varphi$ be a smooth flow on
$\Omega$, i.e. a $C^{\infty}$-map $\varphi \colon \Omega \times \mathbb{R}\rightarrow \Omega$ such
that $\varphi(x,0)=x$ and $\varphi(x,s+t)=\varphi(\varphi(x,s),t)$ for any
$x\in \Omega$ and $s,t \in \mathbb{R}$. We often denote by $\varphi(x,I) := \{ \varphi(x ,t) \mid t \in I \}$ for a subset $I \subset \mathbb{R} $.

\begin{defi} Let $A$ be a compact subset of $\Omega$.
\begin{enumerate}[(1)]
        \item The \emph{maximal invariant subset} of $A$ is given by $\inv{(\varphi,A)} := \{ x \in A \mid \varphi(x ,
\mathbb{R}) \subset A  \}$. We simply write $\inv(A)$ when the flow is clear from the context.

        \item $A$ is called an \emph{isolating neighborhood}  if $\inv{(A)}$ is contained in the interior $\operatorname{int}{(A)}$.
        \item A compact subset $S$ of $\Omega$ is called an \emph{isolated invariant set} if there is an isolating neighborhood
$\tilde{A}$ such that $\inv{(\tilde{A})} = S$. In this situation, we also say that $\tilde{A}$ is
an isolating neighborhood of $S$. 
\end{enumerate}

\end{defi}

A central idea in Conley index theory is a notion of index pairs.

\begin{defi}
For an isolated invariant set $S$, a pair $(N,L)$ of compact
sets $L\subset N$ is called an \emph{index pair} of $S$ if the following conditions hold:

\begin{enumerate}[(i)]
\item  $\inv(N\setminus L)=S\subset \inti(N\setminus L)$;
\item  $L$ is an exit set for $N$, i.e. for any $x\in N$ and $t>0$ such
that $\varphi(x,t)\notin N$, there exists $\tau\in [0,t)$ with $\varphi(x,\tau)\in
L$;
\item  $L$ is positively invariant in $N$, i.e. if $x\in L$, $t>0$, and $\varphi(x,[0,t])\subset N$, then we have $\varphi(x,[0,t])\subset
L$.
\end{enumerate}
\end{defi}

We state two fundamental facts regarding index pairs:
\begin{itemize}
\item For an isolated invariant set $S$ with an isolating neighborhood $A$,
there always exists an index pair $(N,L)$ of $S$ such that $L\subset N\subset A$.
\item For any two index pairs $(N,L)$ and $(N',L')$ of $S$,
there is a natural homotopy equivalence $N/L\rightarrow N'/L'$.
\end{itemize}
These lead to definition of the Conley index.

\begin{defi} \label{def conleyindex} Given an isolated invariant set $S$ of a flow $\varphi$ with an index pair $(N,L) $, we denote
by $I(\varphi,S,N,L)$ the space $N/L$ with $[L]$ as the basepoint. The \emph{Conley
index} $I(\varphi,S) $ can be defined as a collection of pointed spaces $I(\varphi,S,N,L)$  together with natural homotopy equivalences between them. We sometimes write $I(S)$
when the flow is clear from the
context.
\end{defi}

Given two index pairs, a canonical homotopy equivalence between them was constructed
by Salamon \cite{Salamon}.
\begin{thm}[{\cite[Lemma~4.7]{Salamon}}] \label{thm flowmap}
 If $(N,L)$ and $(N',L')$
are two index pairs for the same isolated invariant set $S$, then there exists $\bar{T}>0$ such that
 \begin{itemize}
 \item $\varphi(x,[-\bar{T},\bar{T}])\subset N'\setminus L'$ implies $x\in N\setminus L$;
 \item $\varphi(x,[-\bar{T},\bar{T}])\subset N\setminus L$ implies $x\in N'\setminus L'$.
 \end{itemize}
  Moreover, for any $T\geq \bar{T}$, the map  $s_{T,(N,L),(N',L')}:N/L\rightarrow N'/L'$ given by
\[
\begin{split}
& s_{T,(N,L),(N',L')}([x]) := \\
&\left\{
  \begin{array}{l l}
   [\varphi(x,3T)] & \quad \text{if } \varphi(x,[0,2T])\subset N\setminus L \text{ and } \varphi(x,[T,3T])\subset N'\setminus
L' \\
     \mathop{[L']} & \quad \text{otherwise}
  \end{array} \right.
  \end{split}
\]
is well-defined and continuous.   The maps $s_{T,(N,L),(N',L')}$ are all homotopic to each other  for different $T \geq \bar{T}$ and they
give an isomorphism  
between $N/L$ and $N'/L'$ in the homotopy category of pointed spaces. These isomorphisms satisfy the following properties
\begin{itemize}
\item For any index pair $(N,L)$, the map $s_{T,(N,L),(N,L)}$ is homotopic to the identity map on $N/L$;
\item For any index pairs  $(N,L),\ (N',L')$ and $(N'',L'')$, the composition 
$$s_{T,(N',L'),(N'',L'')}\circ s_{T,(N,L),(N',L')}:N/L\rightarrow N''/L''$$ is homotopic to $s_{T,(N,L),(N'',L'')}$.
\end{itemize}
 We call $s_{T,(N,L),(N',L')}$ the flow map at time $T$. We sometimes also write $s_{T}$ when the index pairs are clear from the context.
\end{thm}

Through the rest of the paper, we will be always working in the homotopy category when talking about maps between Conley indices. Namely, all maps should be understood as homotopy equivalent classes of maps and all commutative diagrams only hold up to homotopy. More precisely, a map $
f:I(\varphi_{1},S_{1})\rightarrow I(\varphi_{2},S_{2})
$ 
between two Conley indices mean a collection of maps 
 (in the homotopy category)
$$\{f_{(N_{1},L_{1}),(N_{2},L_{2})}:I(\varphi_{1},S_{1},N_{1},L_{1})\rightarrow I(\varphi_{2},S_{2},N_{2},L_{2})\},$$
from \emph{any} representative of $I(\varphi_{1},S_{1})$ to \emph{any} representative of $I(\varphi_{2},S_{2})$, such that the following diagram commutes up to homotopy:
$$\xymatrix{
I(\varphi_{1},S_{1},N_{1},L_{1}) \ar[d]_{s_{T,(N_{1},L_{1}),(N'_{1},L'_{1})}} \ar[rrr]^{f_{(N_{1},L_{1}),(N_{2},L_{2})}} &&&I(\varphi_{2},S_{2},N_{2},L_{2})\ar[d]^{s_{T,(N_{2},L_{2}),(N'_{2},L'_{2})}}\\
I(\varphi_{1},S_{1},N'_{1},L'_{1}) \ar[rrr] ^{f_{(N'_{1},L'_{1}),(N'_{2},L'_{2})}}          &&&I(\varphi_{2},S_{2},N'_{2},L'_{2})}$$
In the language of \cite{Salamon}, this collection of maps gives a morphism between two connected simple systems $I(\varphi_{1},S_{1})$ and $I(\varphi_{2},S_{2})$.
Note that to define such a collection $\{ f_{*,*}\}$, we only need to specify a single map 
$$
f_{(N_{1},L_{1}),(N_{2},L_{2})}:I(\varphi_{1},S_{1},N_{1},L_{1})\rightarrow I(\varphi_{2},S_{2},N_{2},L_{2})
$$
for specific choices of $(N_{1},L_{1})$ and $(N_{2},L_{2})$.
This is because all the other maps can be obtained by composing it with flow maps. 



Next, we consider a situation when an isolated invariant set can be decomposed to smaller isolated invariant
sets.

\begin{defi} \ \begin{enumerate}[(1)]
\item  For a subset $A$, we define the $\alpha$-limit set and respectively $\omega$-limit set as following 
$$
\alpha(A)=\mathop{\cap}_{t<0}\overline{\varphi(A,(-\infty,t])} \quad \text{and } \quad
\omega(A)=\mathop{\cap}_{t>0}\overline{\varphi(A,[t,\infty))}.
$$
\item Let $S$ be an isolated invariant set. A compact subset $T\subset
S$ is called an
\emph{attractor} (resp. \emph{repeller}) if there exists a neighborhood $U$
of $T$ in $S$ such that $\omega(U)=T$ (resp. $\alpha(U)=T$).
\item When $T$
is an attractor in $S$, we define the set $T^{*}:=\{x\in S \mid \omega(x)\cap
T=\emptyset\}$, which is a repeller in $S$. We call $(T,T^{*})$ an \emph{attractor-repeller
pair} in $S$.
\end{enumerate}

\end{defi}

Note
that an attractor and a repeller are isolated invariant sets.
We state an
important result relating Conley indices of an attractor-repeller pair.
\begin{pro}[{\cite[Theorem~5.7]{Salamon}}]\label{Attractor-repeller-exact sequence}Let
$S$ be an isolated invariant set with an isolating neighborhood $A$ and $(T,T^{*})$
be an attractor-repeller pair in $S$. Then there exist compact sets $\tilde{N}_{3}\subset
\tilde{N}_{2}\subset \tilde{N}_{1}\subset A$ such that the pairs $(\tilde{N}_{2},\tilde{N}_{3}),
(\tilde{N}_{1},\tilde{N}_{3}),(\tilde{N}_{1},\tilde{N}_{2})$ are index pairs
for $T,$ $S$ and $T^*$ respectively. The maps induced by inclusions give a natural coexact sequence of
Conley indices
$$
I(\varphi,T)\xrightarrow{i_{}} I(\varphi,S)\xrightarrow{r}
I(\varphi,T^{*})\rightarrow \Sigma I(\varphi,T)\rightarrow \Sigma I(\varphi,
S) \rightarrow \cdots.
$$
We call the triple $(\tilde{N}_{3},\tilde{N}_{2},\tilde{N}_{1})$
an index triple for the pair $(T,T^{*})$ and call the maps $i_{}$ and $r$ the attractor map and
the repeller map
respectively.
\end{pro}

%

\subsection{$T$-tame pre-index pair and $T$-tame index pair}
\label{section T-tame}
For a set $A$ and $I\subset \mathbb{R}$, let us denote
$$
A^{I}:=\{x\in \Omega \mid \varphi(x,I)\subset A\}.
$$
We also write $A^{[0,\infty)}$ and $A^{(-\infty,0]}$ as $A^{+}$ and $A^{-}$ respectively.
The following notion of pre-index pair was introduced by Manolescu \cite{Manolescu1}.

\begin{defi}\label{defi pre-index pair}
 A pair $(K_{1},K_{2})$ of compact subsets of an isolating neighborhood $A$ is called \emph{a pre-index pair} in $A$ if
\begin{enumerate}[(i)]
\item For any $x\in K_{1}\cap A^{+}$, we have $\varphi(x,[0,\infty))\subset \operatorname{int}(A)$;
\item $K_{2}\cap A^{+}=\emptyset$.
\end{enumerate}
\end{defi}

We have two basic results regarding pre-index pairs.

\begin{thm}[{\cite[Theorem~4]{Manolescu1}}]\label{from pre-index to index}
For any pre-index pair $(K_{1},K_{2})$ in an isolating neighborhood $A$, there exists an index pair $(N,L)$ satisfying
\begin{equation}\label{pre-index pair contained in index pair}
K_{1}\subset N\subset A,\ K_{2}\subset L.
\end{equation}

\end{thm}

\begin{thm}[{\cite[Proposition~A.5]{Khandhawit1}}]\label{pre-index map compatible}
Let $(K_{1},K_{2})$ be a pre-index pair  and $(N_{1},L_{1})$, $(N_2, L_2)$ be two index pairs containing $(K_{1},K_{2})$. Denote
by $\iota_{j} \colon K_{1}/K_{2}\rightarrow N_{j}/L_{j}$ the map induced by inclusion.
Let $s_{T} \colon N_{1}/L_{1}\rightarrow N_{2}/L_{2}$ be the flow map for some large $T $.  Then, the composition
 $s_{T}\circ \iota_{1}$ is homotopic to $ \iota_{2}$.
\end{thm}

Consequently, when $(K_{1},K_{2})$ is a pre-index pair in an isolating neighborhood $A$, 
we have a \emph{canonical map} to the Conley index 
\begin{equation}\label{map from pre-index to index}
\iota \colon K_{1}/K_{2}\rightarrow I(S),
\end{equation}
where $S = \inv(A) $ and the map is induced by inclusion.

Next, we discuss the quantitative refinement of Theorem \ref{from pre-index to index}, which will be useful in our formulation of relative Bauer--Furuta invariant and the gluing theorem. Let us consider the following definition.


\begin{defi}\label{defi tame pre-index pair}
Let  $A$ be an isolating neighborhood. For a positive real number $T$, a pair $(K_{1},K_{2})$ of compact subsets of $A$ is called a $\emph{$T$-tame
pre-index pair}$ in $A$ if it satisfies the following conditions:
\begin{enumerate}[(i)]

\item \label{item tamepreindex1} There exists a compact set $A' \subset  \operatorname{int}(A) $ containing $A^{[-T,T]} $ such that, if $x\in K_{1}\cap A^{[0,T']}$ for some $T'\geq T$, then $\varphi(x,[0,T'-T])\subset A'$.

\item \label{item tamepreindex2} $K_{2}\cap A^{[0,T]}=\emptyset$.
\end{enumerate}
\end{defi}

It is straightforward to see that a $T$-tame pre-index pair in $A$ is pre-index pair in $A$. The converse also holds.

\begin{lem}\label{always tame for large T}
Let $(K_{1},K_{2})$ be a pre-index pair in an isolating neighborhood $A$. Then, there exists $\bar{T}>0$ such that $(K_{1},K_{2})$ is a ${T}$-tame pre-index pair in $A$  for any ${T}\geq \bar{T}$.
\end{lem}
\begin{proof}
It is straightforward to see that $K_{2}\cap A^{[0,+\infty)}=\emptyset$ implies $K_{2}\cap A^{[0,T]}=\emptyset$ for a sufficiently large $T > 0$. We are  left with checking that condition (\ref{item tamepreindex1}) of Definition \ref{defi tame pre-index pair} holds for a sufficiently large $T>0$.

Suppose that the condition does not hold for $T_j > 0$. Then we can find sequences $\{x_{j,k}\}$, $\{T'_{j,k}\}$ and $\{T''_{j,k}\}$ where $x \in K_1 \cap A^{[0,T'_{j,k}]} $ and $0 \leq T''_{j,k} \leq T'_{j,k} -  T_j $ such that $\{\varphi(x_{j,k},T''_{j,k})\}$ converges to a point on $ \partial A $ as $k \rightarrow \infty $. Now assume that there is a sequence of such $\{T_j\} $ with $T_j \rightarrow \infty $. Passing to a subsequence, one can find a sequence $\{k_j\}$ such that $x_{j,k_j} \rightarrow x_\infty \in K_1 \cap A^{+} $ and $\varphi(x_{j,k_j},T''_{j,k_j}) \rightarrow y \in \partial A $. If $T''_{j,k_j} \rightarrow T'' $, we see that $\varphi(x_\infty , T'') = y $. This contradicts with definition of the pre-index pair $(K_1 , K_2) $. On the other hand, we observe that $\varphi(x_{j,k_j},T''_{j,k_j}) \in A^{[-T''_{j,k_j},T_j]} $. If $\{T''_{j,k_j} \}$ goes to infinity, we obtain that $y\in \inv(A)$.
This is a contradiction because $A$ is an isolating neighborhood, i.e. $\inv(A)\cap \partial A=\emptyset$.

\end{proof}

We next consider the $T$-tame version of index pairs.

\begin{defi}\label{def tame index pair}
For a positive real number $T$, an index pair $(N,L)$ contained in an isolating neighborhood $A$ is called a \emph{$T$-tame index pair}
in $A$ if it satisfies the following conditions:
\begin{enumerate}[(i)]
\item Both $N,L$ are positively invariant in $A$;
\item $A^{[-T,T]}\subset N$;
\item $A^{[0,T]}\cap L=\emptyset$.
\end{enumerate}

A subset $A$ is also called a \emph{$T$-tame isolating neighborhood} if $A^{[-T,T]} \subset \operatorname{int}(A) $.

\end{defi}

One important reason why we are interested in $T$-tame index pairs is that the definition of the flow maps can be simplified when one of the index pairs is $T$-tame. 

\begin{lem}\label{flow map from tame index pair}
Let $(N,L)$ and $(N',L')$ be two index pairs in an isolating neighborhood $A$. Let $T$ be a sufficiently large number so that the flow map $s_{T} \colon N/L\rightarrow
N'/L'$ is well-defined. If the index pair $(N,L)$
is $T$-tame, then flow map $s_{T}$ can be given by a formula
$$
s_{T}([x]) = \left\{
  \begin{array}{l l}
   [\varphi(x,3T)] & \quad \text{if } \varphi(x,[0,3T])\subset A \text{ and } \varphi(x,[T,3T])\subset N'\setminus L', \\
     \mathop{[L']} & \quad \text{otherwise.}
  \end{array} \right.
$$
\end{lem}
\begin{proof}
We only need to show that the following two conditions are equivalent for $x\in N$.
\begin{enumerate}[(1)]
\item $\varphi(x,[0,3T])\subset A \text{ and } \varphi(x,[T,3T])\subset N'\setminus L'$;
\item $\varphi(x,[0,2T])\subset N\setminus L \text{ and } \varphi(x,[T,3T])\subset N'\setminus L'$.
\end{enumerate}
It is straightforward to see that (2) implies (1) since $N \subset A $ and  $N'\subset A$. Let us suppose that $\varphi(x,[0,3T])\subset A$. Since $N$ is positively invariant in $A$, we have $\varphi(x,[0,3T])\subset N$. By property of $T$-tame index pair, we have $\varphi(x,[0,2T])\cap
L=\emptyset$ and we are done.

\end{proof}

We now show a quantitative refinement of \cite[Theorem 4]{Manolescu1}.

\begin{thm}\label{from pre-index to index refined}
For any $T>1$, let $A$ be a $(T-1)$-tame isolating neighborhood and $(K_{1},K_{2})$ be a $(T-1)$-tame pre-index pair in $A$. Then, there exists a $T$-tame index pair in $A$ which contains $(K_{1},K_{2})$.
\end{thm}

\begin{proof}
The proof is an adaption of the arguments in \cite{Manolescu1} to the $T$-tame setting. 
Denote by $\tilde{K}_{1}=K_{1}\cup A^{[-T,T]}$. We claim
that $(\tilde{K}_{1},K_{2})$ is a pre-index pair in $A$. Since $(K_1 , K_2)$ is already a pre-index pair in $A$, it suffices to check that $\varphi(y,[0,\infty))\subset
\inti(A)$ for any $y\in A^{[-T,T]} \cap A^{+} =A^{[-T,\infty]}$. This is straightforward since $A$ is $(T-1)$-tame. 

By Theorem~\ref{from pre-index to index}, there exists an index pair
containing $(\tilde{K}_{1},K_{2})$. 
 From the argument of the proof, one could pick a compact subset $C \subset A$ and as well as an open neighborhood $V$ of $C$  such that the following conditions hold:
\begin{enumerate}[(I)]
 \item
     $C$ is a compact neighborhood of $A^+ \cap \partial A$ in $A$;
     
      \item
     $C \cap A^{-} = \emptyset$;
     
     \item
     $C \cap P_A( \tilde{K}_1) = \emptyset$;

     \item
     ${V}$ is an open neighborhood of $A^{+}$ in $A$;
     
     \item
     $\overline{{V} \setminus C} \subset \inti(A)$;
     
     \item
     $K_2 \cap {V} = \emptyset$. 
\end{enumerate}
Recall that, for a subset $K$ of $A$, the set $P_{A}(K)$ is given by 
$$
P_{A}(K):=\{\varphi(x,t)\mid x\in K,\ t\geq 0 \text{ and } \varphi(x,[0,t])\subset A\}. $$
Let us say that a pair $(C,V)$ is good if it satisfies all of the above conditions. 
After specifying a good pair $(C,V) $, a compact subset $B$ can be chosen so that
$$(N,L) = (P_{A}(B)\cup P_{A}(A\setminus {V}),P_{A}(A\setminus {V}))$$
is an index pair containing $(\tilde{K}_1 ,K_2)$.

 Our strategy is to carefully choose a good pair $(C,V)$ so that the induced $(N,L)$ is  a $T$-tame index pair in $A$ which contains $(K_{1},K_{2})$.

Since $(K_1, K_2)$ is a $T$-tame pre-index pair (as $T > T-1$), we can take a compact set $A'$ in $A$ satisfying condition (\ref{item tamepreindex1}) of Definition~\ref{defi tame pre-index pair}. Fix a compact set $A''$ in $A$ such that
\[
         A' \subset \inti (A''), \ A'' \subset \inti(A)
\]
and pick a real number $T' \in ( T - 1 , T )$. Consider a pair
$$ (C_0 , V_0 ) = (( A \setminus \inti(A'')) \cap A^{[0, T']},  A^{[0, T']})  $$  
Note that $V_0$ is closed. We have the following observations:
\begin{itemize}
\item
   $A^+ \cap \partial A \subset C_0$:
         This is obvious as $A'' \subset \inti (A)$ and $A^+  \subset A^{[0, T']}$.

    \item
    The distance between $C_0$ and $A^{-}$ is positive:
    Observe that
    \[
      \begin{split}
    C_0 \cap A^{-} 
    &= ( A \setminus \inti(A'') ) \cap A^{(-\infty, T']}   \\
    &\subset   (A \setminus \inti(A'')) \cap A^{[-T+1, T-1]} \\
    & = \emptyset,
\end{split}
    \]      
     where we have used the fact that
     $
                   A^{[-T+1, T-1]} \subset A' \subset \inti (A'').
     $
     Since $C_0$ and $A^{-}$ are compact, the distance between them is positive. 
       
   \item
   The distance between $C_0$ and $P_{A}( \tilde{K}_1)$ is positive: 
  Suppose that this is not true. Since $C_0$ is compact,  there would be a sequence $\{ x_j \}$ of points in $\tilde{K}_1$ and a sequence of nonnegative number $\{t_j\}  $ such
that $\varphi(x_j, [0, t_j]) \subset A$ and $y_j = \varphi(x_j, t_j)$ converges to a point $y$ in $C_0$.  
  
  If $t_j \rightarrow \infty$, we would have $\varphi(y,
(-\infty, 0]) \subset A$, which means that $y \in A^{-}$.  This is a contradiction since $C_0 \cap A^{-} = \emptyset$.  
  
  After passing to a subsequence, we now assume that $(x_j ,t_j) \rightarrow (x,t)$ a point in $\tilde{K}_1 \times [0,\infty) $. If $x \in K_1$, then $x \in K_1 \cap A^{[0,  t+T']}$ because $\varphi(x, [0, t]) \subset A$ and $y = \varphi(x, t) \in C_0
\subset A^{[0, T']}$. By the property of $A'$, we have
 \[
      \varphi(x,   [0, t + T' - (T-1)]) \subset A', 
\]
which implies that $y \in A'$. This is a contradiction since $C_0 \cap A' = \emptyset$.  If $x \in A^{[-T, T]}$, then $y
\in A^{[-T-t, T']} \subset A^{[-T+1, T-1]}$. This is also a contradiction since $C_0 \cap A^{[-T+1, T-1]} = \emptyset$.
 
 \item
 $A^+ \subset {V}_0$: This is clear from the definition of ${V}_0$. 
 
 \item
 $\overline{{V}_0 \setminus C_0} \subset \inti(A)$: We will actually prove that $\overline{{V}_0 \setminus C_0} \subset A'' $.  Since $A''$ is closed, it is sufficient
to show that ${V}_0 \setminus C_0 \subset A''$. It is then straightforward to see that 
${V}_0 \setminus C_0  = A^{[0, T']} \cap \inti(A'')  \subset A''$.   
   \item
  The distance between  $K_2$ and ${V}_0$ is positive: 
   Since $(K_1, K_2)$ is $(T-1)$-tame,  we have $K_2 \cap A^{[0, T-1]} = \emptyset$, and consequently $K_2 \cap {V}_0 = \emptyset$. Since $K_2$ and ${V}_0$ are compact, the distance between them is positive. 
\end{itemize}

For a sufficiently small positive number $d$, we define
$$
            C := \{ x \in A | \operatorname{dist}(x, C_0) \leq d \}, \
            {V} := \{ x \in A | \operatorname{dist}(x, \tilde{V}_0) < d \}. 
$$
From the above observations, one can check that $(C,V)$ is a good pair.

We finally check that $(N,L) = (P_{A}(B)\cup P_{A}(A\setminus {V}),P_{A}(A\setminus {V}))$ is $T$-tame.
\begin{enumerate}[(i)]
\item Notice that $P_A(S)$ is positively invariant for any subset $S \subset A$ and that the union of two positively invariant sets in $A$ is again positively invariant in $A$. Thus, $N,L$ are positively invariant in $A$
\item From our construction, we have $A^{[-T,T]} \subset  \tilde{K}_1 \subset N $.
\item We are left to show that $A^{[0, T]} \cap L = \emptyset$. Suppose that there is an element  $x \in A^{[0, T]} \cap L$. From the definition, we obtain $y \in A \setminus {V}$ and $t \geq 0$ such that
$\varphi(y, [0, t]) \subset A$ and $x = \varphi(y, t)$. It follows that
$ y \in A^{[0, T + t]}$.
On the  other hand, we have
$ A^{[0, T+t]} \subset A^{[0, T']} = {V}_0 \subset {V}$.
This is a contradiction since $y \not\in \tilde{V}$.
\end{enumerate}
 
\end{proof}

\subsection{The attractor-repeller pair arising from a strong Morse decomposition}

In many situations, we obtain an attractor-repeller pair by decomposing an isolating neighborhood to two parts.
We introduce the following notion, which arises in many situations.
\begin{defi}\label{defi strong morses decom}
Let $(A_{1},A_{2})$ be a pair of compact subsets of an isolating neighborhood $A$. We say that $(A_{1},A_{2})$ is a \emph{strong Morse decomposition of $A$} if
\begin{itemize}
\item $A=A_{1}\cup A_{2};$
\item For any $x\in A_{1}\cap A_{2}$, there exists $\epsilon>0$ such that \begin{equation}\label{morse condition}
\varphi(x,(0,\epsilon))\cap A_{1}=\emptyset \text{ and } \varphi(x,(-\epsilon,0))\cap A_{2}=\emptyset.\end{equation}
\end{itemize}
\end{defi}

Simply speaking, the flow leaves $A_1$ immediately and enters $A_2$ immediately at any point on $A_1 \cap A_2 $ (see Figure~\ref{fig strong Morse}). A strong Morse decomposition naturally occurs when we split $A$ by a level set of some function transverse to the flow. Let us summarize some basic properties of a strong Morse decomposition in the following lemma.  The proofs are straightforward and we omit them.

\begin{figure}[hbtp]
    \centering
        \psset{unit=1.0cm}
\begin{pspicture}(0,0)(6, 2.6)
\psrotate(3,1.3){180}{
\pspolygon[fillstyle=hlines, hatchwidth=0.3pt, hatchsep=12pt](0,0)(2.8,0)(2.8,2.6)(0,2.6)
\pspolygon(0,0)(5.6,0)(5.6,2.6)(0,2.6)
\psline{->}(3.1,0.2)(2.5,0.1)
\rput(0,0.4){\psline{->}(3.1,0.2)(2.5,0.1)}
\rput(0,0.8){\psline{->}(3.1,0.2)(2.5,0.1)}
\rput(0,1.4){\psline{->}(3.1,0.2)(2.5,0.3)}
\rput(0,1.8){\psline{->}(3.1,0.2)(2.5,0.3)}
\rput(0,2.2){\psline{->}(3.1,0.2)(2.5,0.3)}
\psdot(1.4,1.3)
\psdot(4.2,1.3)
\psline{->}(4.2,1.3)(2.8,1.3)
\psline(2.8,1.3)(1.4,1.3)}
\rput(5.5,2.2){$A_2$}
\rput(1,2.2){$A_1$}
\end{pspicture}
      
    \caption{A strong Morse decomposition}
   \label{fig strong Morse}
\end{figure}
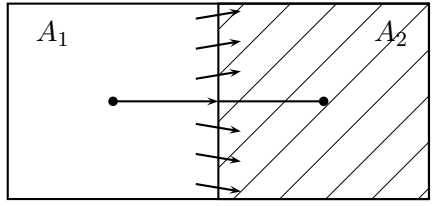

\begin{lem}\label{basic property of strong decomposition}
Let $(A_{1},A_{2})$ be a strong Morse decomposition of an isolating neighborhood $A$. Then, we have the following results.
\begin{enumerate}[(1)]
\item $A_{1}$ (resp. $A_{2}$) is negatively (resp. positively) invariant in $A$;
\item $A_{1}\cap A_{2}=\partial A_{1}\cap \partial A_{2}$ and $\partial A_{i} = (\partial A \cap A_i ) \cup (A_{1}\cap A_{2})$ for
$i=1,2$;
\item $A_{1}$ and $A_{2}$ are isolating neighborhoods;
\item $(\inv(A_{2}),\inv(A_{1}))$ is an attractor-repeller pair in $\inv(A)$.
\end{enumerate}
\end{lem}

When an attractor-repeller pair comes from a strong Morse decomposition, we have an extra property for index triples as follows.
\begin{lem}\label{special index triple}
Let $(A_{1},A_{2})$ be a strong Morse decomposition of $A$. Suppose that $(\tilde{N}_{3},\tilde{N}_{2},\tilde{N}_{1})$ is an index triple for $(\inv(A_{2}),\inv(A_{1}))$ and denote by $\tilde{N}'_{2}=\tilde{N}_{2}\cup (\tilde{N}_{1}\cap A_{2})$. Then, $(\tilde{N}_{3},\tilde{N}'_{2},\tilde{N}_{1})$
is again an index triple for $(\inv(A_{2}),\inv(A_{1}))$.  In particular, we can always pick an index triple  $(\tilde{N}_{3},\tilde{N}_{2},\tilde{N}_{1})$ of $(\inv(A_{2}),\inv(A_{1}))
$  satisfying  $\tilde{N}_{1}\cap A_{2} \subset \tilde{N}_{2}$.
\end{lem}
\begin{proof}
We simply check each condition of index pairs one by one. \begin{itemize}
 \item  $\tilde{N}'_{2}$ is positively invariant in $\tilde{N}_{1}$: Since $A_{2}$ is positively invariant in $A$, $A_{2}\cap
\tilde{N}_{1}$ is positively invariant in $\tilde{N}_{1}$. The set $\tilde{N}_{2}$ is also positively invariant in $\tilde{N}_{1}$ because $(\tilde{N}_{1},\tilde{N}_{2}) $ is an index pair. It is straightforward to see that the union
 of two positively invariant sets is a positively invariant set. 
 \item $\tilde{N}'_{2}$ is an exit set for $\tilde{N}_{1}$ because $\tilde{N}'_{2}$ contains $\tilde{N}_{2}$, which
is an exit set for $\tilde{N}_{1}$.
 \item $\inv(A_{1})= \inv(\tilde{N}_{1}\setminus \tilde{N}'_{2})\subset \inti(\tilde{N}_{1}\setminus \tilde{N}'_{2})$: Consider an element $x\in \inv(\tilde{N}_{1}\setminus
\tilde{N}_{2})=\inv (A_{1})$. Then, $\varphi(x,(-\infty,\infty))$ is contained in $(\tilde{N}_{1}\setminus \tilde{N}_{2})  \cap 
\inti (A_{1})$. Since $\inti (A_{1}) \cap A_2 = \emptyset$,  we see that  $\varphi(x,(-\infty,\infty))\subset
\tilde{N}_{1}\setminus (\tilde{N}_{2}\cup (\tilde{N}_{1}\cap A_{2}))$. Thus, $x\in \inv (\tilde{N}_{1}\setminus
\tilde{N}'_{2})$ and $\inv(\tilde{N}_{1}\setminus \tilde{N}_{2})\subset \inti(\tilde{N}_{1}\setminus \tilde{N}'_{2})$. Since $\tilde{N}_{1}\setminus \tilde{N}'_{2}\subset \tilde{N}_{1}\setminus \tilde{N}_{2} $, we have $\inv(\tilde{N}_{1}\setminus
\tilde{N}'_{2}) = \inv(\tilde{N}_{1}\setminus \tilde{N}_{2}) = \inv{(A_1)} $. Note that $\inv(A_1) \subset \inti(\tilde{N}_{1}\setminus \tilde{N}'_{2}) $ because $\inv(A_1) \subset \inti(\tilde{N}_{1}\setminus
\tilde{N}_{2}) $ and $\inv(A_1) \cap A_2 = \emptyset $. 

\item $\tilde{N}_{3}$ is positively invariant in $\tilde{N}'_{2}$ because $\tilde{N}_{3}$ is positively invariant
in $\tilde{N}_{1}$, which contains $\tilde{N}'_{2}$.
\item $\tilde{N}_{3}$ is an exit set for $\tilde{N}'_{2}$: We only have to check that $\tilde{N}_{3}$ is an exit set for $\tilde{N}_{1} \cap A_2$. Suppose that $x \in \tilde{N}_{1} \cap A_2 $ but $\varphi(x,t) \notin \tilde{N}_{1} \cap A_2 $ for some $t>0$. Notice that a flow cannot go from $A_2$ to $A_1$ since $(A_1, A_2)$ is a strong Morse decomposition. If $ \varphi(x,t) \in \tilde{N_1}$, we would have $\varphi(x,t) \notin A_2 $ which implies $\varphi(x,t) \notin A $ a contradiction. When $\varphi(x,t) \notin \tilde{N_1}$, we can use the fact that $\tilde{N}_{3}$ is an exit set for $\tilde{N}_{1}$. 
\item $\inv(A_{2})=\inv(\tilde{N}'_{2}\setminus \tilde{N}_{3})\subset \inti(\tilde{N}'_{2}\setminus \tilde{N}_{3})$: Suppose that we have $x\in \inv(\tilde{N}'_{2}\setminus
\tilde{N}_{3})$ such that $\varphi(x,t) \notin \tilde{N}_{2}\setminus \tilde{N}_{3}$ for some $t \in \mathbb{R}$.  Since $\varphi(x,(-\infty,\infty))$
does not intersect $\tilde{N}_{3}$, which is an exit set for both $\tilde{N}_{2}$ and $\tilde{N_{1}}\cap A_{2}$, one can deduce that $\varphi(x,(-\infty,\infty))\subset\tilde{N_{1}}\cap A_{2}$. This implies  $x\in \inv
(A_{2}) = \inv(\tilde{N}_{2}\setminus \tilde{N}_{3})$ which is a contradiction. 
Therefore, $\inv (\tilde{N}'_{2}\setminus \tilde{N}_{3}) \subset \inv(\tilde{N}_{2}\setminus
\tilde{N}_{3})$ while the converse is trivial. Consequently, $\inv(\tilde{N}'_{2}\setminus \tilde{N}_{3}) =\inv(A_{2}) $ is contained in $\inti(\tilde{N}_{2}\setminus \tilde{N}_{3}) \subset \inti(\tilde{N}'_{2}\setminus \tilde{N}_{3})$.
\end{itemize}\end{proof}

When an attractor-repeller pair arises from  a strong Morse decomposition, we will show that
canonical maps from pre-index pairs are compatible with the attractor and repeller maps.
\begin{pro}\label{pre-index map compatible with attractor}
Let $(A_{1},A_{2})$ be a strong Morse decomposition of $A$ and let $(K_{1},K_{2})$ be a pre-index pair in $A_{2}$. Then,
we have the following:
\begin{enumerate}[(1)]
\item $(K_{1},K_{2})$ is also a pre-index pair in $A$;
\item We have a commutative diagram 
\begin{equation*}
\xymatrix{
K_1 / K_2 \ar[dr]_{\iota} \ar[r]^{\iota_2} & I(\inv(A_2)) \ar[d]^{i}\\
 & I(\inv(A))}
\end{equation*}
where $\iota , \iota_2 $ are the canonical maps and $i \colon I(\inv (A_2))\rightarrow I(\inv (A_{}))$ is the attractor map.
\end{enumerate}
\end{pro}
\begin{proof} \
\begin{enumerate}
\item  
 Consider $x\in K_{1}$ satisfying $\varphi(x,[0,\infty))\subset A$. Since $A_{2}$ is positively invariant in $A$ and
$x\in K_1 \subset A_{2}$, we have $\varphi(x,[0,\infty))\subset A_{2}$. Consequently, we see that $\varphi(x,[0,\infty))\subset \inti(A_{2})\subset
\inti(A)$ because $(K_{1},K_{2})$ is an pre-index pair in $A_{2}$.
Now, consider $x\in K_{2} \cap A^+$. Again, since $A_{2}$ is positively invariant in $A$,
we have $\varphi(x,[0,\infty))\subset A_{2}$. This is impossible because $K_{2}\cap A^{+}_{2}=\emptyset$.

\item Let $\tilde{N}_{3}\subset \tilde{N}_{2}\subset \tilde{N}_{1}\subset A$ be an index triple for $(\inv(A_{2}),\inv(A_{1}))$
such that $\tilde{N}_{1}\cap A_{2}\subset \tilde{N}_{2}$ (cf. Lemma \ref{special index triple})
and let $L\subset N\subset A$ (resp. $L_{2}\subset N_{2}\subset A_{2}$) be an index pair for $\inv(A)$ (resp. $\inv(A_{2})$)
that contains $(K_{1},K_{2})$. By Theorem~\ref{from pre-index to index refined}, we may also assume that both $(N,L)$ and $(N_{2},L_{2})$ are $T$-tame.
By possibly increasing $T$, we also assume that we have flow maps $s_T \colon N/L\xrightarrow{}\tilde{N}_{1}/\tilde{N}_{3}$
and $s'_T \colon N_{2}/L_{2}\xrightarrow{}\tilde{N}_{2}/\tilde{N}_{3} $. 
Then, the map $i \circ \iota$ is represented by a composition
$$
K_{1}/K_{2}\xrightarrow{\iota_2}N_{2}/L_{2}\xrightarrow{s'_{T}}\tilde{N}_{2}/\tilde{N}_{3}\xrightarrow{i}\tilde{N}_{1}/\tilde{N}_{3}
$$
while the map $\iota$ is represented by the composition
$$
K_{1}/K_{2}\xrightarrow{\iota}N/L\xrightarrow{s_{T}}\tilde{N}_{1}/\tilde{N}_{3}.$$
We will show that these two compositions are in fact the same map.

Applying Lemma~\ref{flow map from tame index pair}, one can check that $i \circ s'_T \circ \iota_2 $ sends $[x]$ to $[\varphi(x,3T)]$
if
\begin{equation}\label{condition 3atrcomp}
\varphi(x,[0,3T])\subset A_{2},\ \varphi(x,[T,3T])\subset \tilde{N}_{2}\setminus \tilde{N}_{3}
\end{equation}
and to the basepoint otherwise. On the other hand, $s_T \circ \iota $ sends $[x]$ to $[\varphi(x,3T)]$
if
\begin{equation}\label{condition 4atrcomp}
\varphi(x,[0,3T])\subset A,\ \varphi(x,[T,3T])\subset \tilde{N}_{1}\setminus \tilde{N}_{3}
\end{equation}
and to the basepoint otherwise. It is obvious that condition (\ref{condition
3atrcomp}) implies (\ref{condition 4atrcomp}). On the other hand, condition (\ref{condition
4atrcomp}) implies (\ref{condition 3atrcomp}) for $x \in K_1 \subset A_2 $ simply because  $A_{2}$ is positively invariant in $A$ and $\tilde{N}_{1}\cap
A_{2}\subset \tilde{N}_{2}$.
\end{enumerate}
\end{proof}
\begin{pro}\label{pre-index map compatible with repellor}
Let $(A_{1},A_{2})$ be a strong Morse decomposition of $A$ and let $(K_{3},K_{4})$ be a pre-index pair in $A$.
Consider a pair $(K'_{3},K'_{4}) := (K_{3}\cap A_{1},(K_{4}\cap A_{1})\cup (K_{3}\cap A_{1}\cap A_{2}))$.  Then, we have the followings:
\begin{enumerate}[(1)]
\item The pair $(K'_{3},K'_{4})$ is a pre-index pair in $A_{1}$;
\item A map $q \colon K_{3}/K_{4}\rightarrow K'_{3}/K'_{4}$ given by
$$q([x])=\left\{
  \begin{array}{l l}
   [x] & \quad \text{if } x\in K'_{3}, \\
     \mathop{[K'_{4}]} & \quad \text{otherwise,}
  \end{array} \right.$$
  is well-defined and continuous;
\item We have a commutative diagram 
\begin{equation*}
\xymatrix{
K_3 / K_4 \ar[d]_{q} \ar[r]^{\iota} & I(\inv(A)) \ar[d]^{r}\\
K'_3 / K'_4 \ar[r]^{\iota'} & I(\inv(A_1))}
\end{equation*}
where $\iota , \iota' $ are the canonical maps and $r \colon I(\inv (A))\rightarrow I(\inv (A_{1}))$ is the repeller map.

\end{enumerate}
\end{pro}
\begin{proof}

\ \begin{enumerate}
\item We will check the two conditions of pre-index pair directly. Suppose that $x\in K'_{3}$ and $\varphi(x,[0,\infty))\subset A_1$. We can see that  $
\varphi(x,[0,\infty))\cap
(A_{1}\cap A_{2})=\emptyset
$ from the property (\ref{morse condition}) of strong Morse decomposition. 
Since  $(K_{3},K_{4})$ is a pre-index pair in $A$ and $x \in K_3 \cap A^{+} $  we have $\varphi(x,[0,\infty))\cap
\partial A=\emptyset$. Consequently, we can deduce that $\varphi(x,[0,+\infty))\cap
\partial A_{1}=\emptyset$ because $\partial A_{1} = (\partial A \cap A_1 ) \cup (A_{1}\cap A_{2})$.

Since $(K_{3},K_{4})$ is a pre-index pair in $A$, we have $K_4 \cap A^+ = \emptyset$. It follows directly that $(K_{4}\cap A_{1})\cap A^{+}_{1}=\emptyset$. On the other hand,  we can see that $(K_{3}\cap A_{1}\cap A_{2})\cap A^{+}_{1}=\emptyset$ as a point on $A_1 \cap A_2 $ leaves $A_1$ immediately. Therefore, $K'_{4}$ has empty intersection with $A^{+}_{1}$.

\item Note that $q$ is continuous because $(\overline{K_{3}\setminus K'_{3}})\cap K'_{3} = K_3 \cap A_1 \cap A_2\subset K'_{4}$.
For $x \in K_4 \cap K'_3 \subset K_4 \cap A_1 \subset K'_4 $, we see that $q$ is well-defined.

\item As in the proof of Proposition \ref{pre-index map compatible with attractor}, let $\tilde{N}_{3}\subset\tilde{N}_{2}\subset\tilde{N}_{1}\subset
A $ be an index triple for $(\inv (A_{2}),\inv(A_{1}))$ with $\tilde{N}_{1}\cap A_{2}\subset \tilde{N}_{2}$ and let $L\subset N\subset A$ (resp. $L_{1}\subset N_{1}\subset A_{1}$)
be an index pair for $A$ (resp. for $A_{1}$) that contains $(K_{3},K_{4})$ (resp. $(K'_{3},K'_{4})$).
By Theorem~\ref{from pre-index to index refined}, we can assume that $(N,L)$ and $(N_{1},L_{1})$ are both $T$-tame.
By possibly increasing $T$, we also assume that we have flow maps $s_{T} \colon N/L\xrightarrow{} \tilde{N}_{1}/\tilde{N}_{3} $ and $s'_T \colon N_{1}/L_{1}\xrightarrow{}\tilde{N}_{1}/\tilde{N}_{2} $.
 Then, the map
$q\circ \iota'$ is represented by
\begin{equation*}
K_{3}/K_{4}\xrightarrow{q} K'_{3}/K'_{4}\xrightarrow{\iota'}N_{1}/L_{1}\xrightarrow{s'_{T}}\tilde{N}_{1}/\tilde{N}_{2},
\end{equation*}
and the map $r\circ \iota$ is represented by
\begin{equation*}
K_{3}/K_{4}\xrightarrow{\iota}N/L\xrightarrow{s_{T}} \tilde{N}_{1}/\tilde{N}_{3}\xrightarrow{r} \tilde{N}_{1}/\tilde{N}_{2}.
\end{equation*}
We will show that these two compositions are in fact the same maps.

Applying Lemma~\ref{flow map from tame index pair},
one can check that  $s'_T \circ \iota' \circ q $ sends $[x]$ to $[\varphi(x,3T)]$ if
\begin{equation}\label{condition 1repcomp}
\varphi(x,[0,3T])\subset A_{1} \text{ and } \varphi(x,[T,3T])\subset \tilde{N}_{1}\setminus \tilde{N}_{2}
\end{equation}
and to the basepoint otherwise. On the other hand, $ r \circ s_T \circ \iota$ sends $[x]$ to $[\varphi(x,3T)]$
if
\begin{equation}\label{condition 2repcomp}
\begin{split}
&\varphi(x,[0,3T])\subset A, \\
& \varphi(x,[T,3T])\subset \tilde{N}_{1}\setminus \tilde{N}_{3} \\
& \varphi(x,3T)\notin
\tilde{N}_{2}
\end{split}
\end{equation}
and to the basepoint otherwise. Clearly, condition (\ref{condition 1repcomp}) implies condition (\ref{condition 2repcomp}). We will check that the two conditions are the same.
Consider an element $x\in K_{3}$ satisfying (\ref{condition 2repcomp}). We see that $\varphi(x,3T)\in\tilde{N}_{1}\setminus \tilde{N}_{2}\subset A_{1}$ because $\tilde{N}_{1}\cap A_{2}\subset \tilde{N}_{2}$.
Since $A_{1}$ is negatively invariant in $A$, we have $\varphi(x,[0,3T])\subset A_{1}$. Moreover, the facts $\varphi(x,3T)\notin
\tilde{N}_{2}$ and $\varphi(x,[T,3T])\cap \tilde{N}_{3}=\emptyset$ imply that $\varphi(x,[T,3T])\cap \tilde{N}_{2}=\emptyset$
since $\tilde{N}_{3}$ is an exit set for $\tilde{N}_{2}$. We have proved that $x$
satisfies condition (\ref{condition 1repcomp}).

\end{enumerate}
\end{proof}

\subsection{$T$-tame manifold isolating block for Seiberg-Witten flow}
\label{sec Ttamemfd}

\begin{defi}
For a compact set $N$ in $\Omega$, we consider the following subsets of its boundary:
\[
     \begin{split}
        n^+(N) := & \{ x \in \partial N | \exists \epsilon > 0 \ \text{s.t.} \ \varphi(-\epsilon, 0) \cap N = \emptyset 
\}, \\
        n^-(N) := & \{ x \in \partial N | \exists \epsilon > 0 \ \text{s.t.} \ \varphi(0, \epsilon) \cap N = \emptyset \}.
     \end{split}
\]
A compact set $N$ is called an \emph{isolating block} if $\partial N = n^+(N) \cup n^-(N)$. 

\end{defi}

It is straightforward to verify that an isolating block $N$ is an isolating neighborhood and that $(N, n^-(N))$ is an index pair.

\begin{defi}
Let $S$ be  a compact subset of $\Omega$. If $N$ is a compact submanifold of $\Omega$ and is also an isolating block with $\inv  N = S$, we call $N$ a \emph{manifold isolating block} of $S$. 
\end{defi}

In \cite{Conley-Easton}, it is proved that, for any isolating neighborhood $A$, we can always find a manifold isolating block
$N$ of $\inv A$ with $N \subset A$. 
We also introduce a notion of tameness for an isolating block as quantitative refinement as in Section~\ref{section T-tame}.

\begin{defi} \label{def T-tame isolating block}
Let $A$ be an isolating neighborhood and $T$ be a positive number. An isolating block $N$ in $A$ is called $T$-tame
if $A^{[-T, T]} \subset \inti(N)$. 
\end{defi}


We turn into special situation involving construction of the spectrum invariants a 3-manifold $Y$: $\swf^{A}(Y, \frak{s}, A_0, g; S^1)$ and $\swf^{R}({Y},\frak{s}, A_0, g; S^1)$.
Let $R_0$ be the universal constant from \cite[Theorem~3.2]{KLS1} such that all finite-type Seiberg-Witten trajectories are contained in $Str(R_{0})$ (see (\ref{eq: str})). Take a positive number $\tilde{R}$ with $\tilde{R} > R_0$,
 sequences $\lambda_{n} \rightarrow -\infty$, $\mu_{n} \rightarrow \infty$ and consider the sets
\[
     \begin{split}
        J_{m}^{\pm}  &:= Str( \tilde{R} ) \cap \bigcap_{ 1 \leq j \leq b_1 } g_{j, \pm}^{-1}(-\infty,  \theta + m], \\
        J_{m}^{n, \pm}  &:= J_{m}^{\pm} \cap V_{\lambda_n}^{\mu_n},
     \end{split}
\]
defined in Section 2.1 (see also \cite[Section~5.1]{KLS1} for more details.)

\begin{lem} \label{lem J 2T int}
For each positive integer $m$, there is a positive number $T_{m}$ independent of $n$ such that  
\[
              (J_{m}^{n,+})^{[-2T, 2T]} \subset \inti  \left\{ (J_{m}^{n, +})^{[-T, T]} \right\}, 
\]
for all $T > T_{m}$ and $n$ sufficiently large. In particular,   $(J_{m}^{n,+})^{[-2T, 2T]} \subset \inti (J_{m}^{n, +})$. Similar results hold for $J_{m}^{n, -}$. 
\end{lem}

\begin{proof}
If the statement is not true, we have a sequence $T_n \rightarrow \infty$ such that we can take elements
\[
              x_n \in  (J_{m}^{n, +})^{[-2T_n, 2T_n]} \cap \partial \left\{ (J_{m}^{n, +})^{[-T_n, T_n]} \right\}.
\]
In particular, we would have
\[
         \varphi_{m}^{n}(x_n, [-2T_n, 2T_n]) \subset J_{m}^{n, +} \text{ and }
         \varphi_{m}^{n}(x_n, t_n) \in \partial J_{m}^{n, +} 
\]
for some    $t_n\in [-T_n, T_n]$,
which implies 
\[
\varphi_{m}^{n}(x_n, t_n) \in (J_{m}^{n, +})^{[-T_n, T_n]}. 
\]
On the other hand, by \cite[Lemma~5.4]{KLS1}, we must have
\[
              \varphi_{m}^{n}(x_n, t_n) \in \partial Str(\tilde{R}). 
\]
This is a contradiction to \cite[Lemma 5.5 (a)]{KLS1}. 

\end{proof}

We now state the main result of this section.

\begin{pro} \label{prop mfdnhbdswf}
Let $T_m$ be the constant from Lemma~\ref{lem J 2T int}.  When $T > 4T_{m}$ and $n$ is sufficiently large, we can always find a $T$-tame manifold
isolating block $N_{m}^{n, +}$ of $\inv(J_{m}^{n, +})$ with $N_{m}^{n, +} \subset J_{m}^{n, +}$. 
A similar result holds for $J_{m}^{n, -}$. 
\end{pro}

\begin{proof}
Fix $m$ and suppose that $n$ is sufficiently large so that the statement of Lemma \ref{lem J 2T int} holds.
Take a positive number $T$ with $T > 4T_m$.  By Lemma \ref{lem J 2T int}, we have 
\[
\begin{split}
     &     (J_{m}^{n, +})^{[-T, T]} \subset \inti \left\{ (J_{m}^{n, +})^{[-T/2, T/2]} \right\} \text{ and }  \\
   &       (J_{m}^{n, +})^{[-T/2, T/2]} \subset \inti \left\{ (J_{m}^{n, +})^{[-T/4, T/4]} \right\}.
 \end{split}
\] 
We can take a smooth function $\tau : V_{\lambda_n}^{\mu_n} \rightarrow [0, 1]$ such that
\[
        \begin{split}
             \tau = 0 \ \text{on $(J_{m}^{n, +})^{[-T, T]} $, and}  \ 
             \tau = 1 \ \text{on $V_{\lambda_n}^{\mu_n} \setminus (J_{m}^{n, +})^{[-T/2, T/2]}$.}
        \end{split}
\]
and a smooth bump function $\iota_m: Coul(Y)\rightarrow [0,1]$ such that 
\[
     \begin{split}
             \iota_{m}^{-1}((0,1]) \ \text{ is bounded, and}  \ 
             \iota_{m} = 1 \ \text{in a neighborhood of $J^{\pm}_{m+1}$.}
        \end{split}
\]
Let $\tilde{\varphi}_{m}^{n}$ be the flow on $V_{\lambda_n}^{\mu_n}$ generated by $\tau \cdot \iota_m \cdot ( l + p_{\lambda_n}^{\mu_n}
\circ c)$.
We will prove that $J_{m}^{n, +}$ is an isolating neighborhood of $\inv(\tilde{\varphi}_{m}^{n}, J_{m}^{n, +})$.   If this
is not true, we can take 
\[   
            x \in \partial J_{m}^{n, +}  \cap \inv (\tilde{\varphi}_{m}^{n}, J_{m}^{n, +}).
\]

Put
\[
  \begin{split}
         P^+(x)  &:= \{ \varphi_{m}^{n}(x, t) | t \geq 0,  \varphi_{m}^{n}(x, [0, t]) \subset J_{m}^{n, +} \}, \\
         P^-(x) &:= \{ \varphi_{m}^{n}(x, t) | t \leq 0, \varphi_{m}^{n} (x, [t, 0]) \subset J_{m}^{n, +}  \}.
  \end{split}
\]
Suppose that $ P^+(x) \cap (J_{m}^{n, +})^{[-T/2, T/2]} = \emptyset $. This means a forward $\varphi_{m}^{n}$-trajectory of $x$ inside $ J_{m}^{n, +}$ lie outside $(J_{m}^{n, +})^{[-T/2, T/2]}$, so that a forward $\varphi_{m}^{n}$-trajectory agrees with a forward $\tilde{\varphi}_{m}^{n}$-trajectory. Consequently, we have
\(
      \varphi_{m}^{n}(x, [0,\infty)) = \tilde{\varphi}_{m}^{n}(x, [0, \infty)) \subset J_{m}^{n, +}. 
\)
Hence $\varphi_{m}^{n}(x, T/2) \in P^+(x) $ and $\varphi_{m}^{n}(x, T/2) \in  (J_{m}^{n, +})^{[-T/2, T/2]} $ which is a contradiction. We can now conclude that $ P^+(x) \cap (J_{m}^{n, +})^{[-T/2, T/2]} \neq \emptyset $ and, in particular, $x \in (J_{m}^{n, +})^{[0, T/2]}$. 

Similarly we can deduce that $x \in (J_{m}^{n,+})^{[ -T/2 , 0]}$.
These facts imply that
\[
        x \in (J_{m}^{n, +})^{[-T/2, T/2]} \cap \partial J_{m}^{n, +},
\]
which is a contradiction because
\[
       (J_{m}^{n, +})^{[-T/2, T/2]} \subset \inti \left\{ (J_{m}^{n, +})^{[-T/4, T/4]} \right\} \subset \inti (J_{m}^{n,
+}). 
\]

Therefore $J_{m}^{n, +}$ is an isolating neighborhood of $\inv( \tilde{\varphi}_{m}^{n}, J_{m}^{n, +})$. 
By the result of Conley and Easton \cite{Conley-Easton}, we can find a manifold isolating block $N_{m}^{n, +}$ of $\inv(\tilde{\varphi}_{m}^{n},
J_{m}^{n, +})$ with $N_{m}^{n, +} \subset J_{m}^{n,+}$.  Note that
\[
        (J_{m}^{n, +})^{[-T, T]}  \subset \inv(\tilde{\varphi}_{m}^{n}, J_{m}^{n, +}) \subset \inti N_{m}^{n,+}.
\]
Since the directions of the flows $\varphi_{m}^{n}$ and $\tilde{\varphi}_{m}^{n}$ coincide on $\partial  N_{m}^{n,+} \subset J_{m}^{n,+} \setminus \tau^{-1}(0)$, we see that
$N_{m}^{n, +}$ is also a manifold isolating block of $\inv( \varphi_{m}^{n}, J_{m}^{n, +} )$. Thus $N_{m}^{n, +}$ is a $T$-tame
manifold isolating block of $\inv (\varphi_{m}^{n}, J_{m}^{n, +})$ in $J_{m}^{n, +}$.

\end{proof}

%

\section{Stable homotopy categories}
\label{section stablecategory}

\subsection{Summary}

In this section, we will discuss the stable homotopy categories $\frak{C}$, $\frak{S}$, $\frak{S}^*$. The discussion in this section will be  needed to construct the gluing formula in Theorem \ref{thm gluing BF inv}. 

First let us briefly recall the definition of the categories. (See \cite{KLS1} for the details.)  An object of $\frak{C}$ is a triple $(A, m, n)$, where $A$ is  a pointed topological space with $S^1$-action which is $S^1$-homotopy equivalent to a finite $S^1$-CW complex (see \cite[Chapter I, Section 3]{May} for the definition of equivariant CW complex), $m$ is an integer and $n$ is a rational number.  The set of morphisms between $(A_1, m_1, n_1)$ and $(A_2, m_2, n_2)$ is given by
\[
 \begin{split}
    & \morp_{\frak{C}}((A_1, m_1, n_1), (A_2, m_2, n_2))   \\
   & \quad   = \lim_{u, v \rightarrow \infty} [  (\R^{u} \oplus \C^{v})^+ \wedge A_1, (\R^{u+m_1-m_2} \oplus \C^{v + n_1-n_2})^+ \wedge A_2]_{S^1}
 \end{split}
\]
if $n_1-n_2$ is an integer, and we define $\morp_{\frak{C}}((A_1, m_1, n_1), (A_2, m_2, n_2))$ to be the empty set if $n_1-n_2$ is not an integer. Here $[\cdot, \cdot]_{S^1}$ is the set of pointed $S^1$-homotopy classes, $\R$ is the one dimensional trivial representation of $S^1$ and $\C$ is the standard two dimensional representation of $S^1$. 
The category $\frak{S}$ is the category of direct systems 
\[
       Z : Z_1 \stackrel{j_1}{\rightarrow} Z_2 \stackrel{j_2}{\rightarrow} \cdots 
\]
in $\frak{C}$. Here $Z_m$ and $j_m$ are an object and morphism in $\frak{C}$ respectively.  For objects $Z, Z'$ in $\frak{S}$, the set morphism is defined by
\[
        \morp_{\frak{S}}(Z, Z') = \lim_{\infty \leftarrow m} \lim_{n \rightarrow \infty} \morp_{\frak{C}} (Z_m, Z_n').
\]
The category $\frak{S}^*$ is the category of inverse systems 
\[
    \bar{Z} : \bar{Z}_1 \stackrel{ \bar{j}_1}{\leftarrow} \bar{Z}_2  \stackrel{ \bar{j}_2 }{\leftarrow} \cdots
\]
in $\frak{C}$. Here $\bar{Z}_m$ and $\bar{j}_m$ are an object and morphism in $\frak{C}$ respectively.  For objects $\bar{Z}$,  $\bar{Z}'$ in $\frak{S}^*$, the set of morphisms is defined by
\[
      \morp_{\frak{S}^*} (\bar{Z}, \bar{Z}') = 
       \lim_{\infty \leftarrow n}   \lim_{m \rightarrow \infty} \morp_{\frak{C}} (\bar{Z}_m, \bar{Z}_n').
\]
In Section \ref{section smashpr}, we will define the smash product in the category $\frak{C}$ and prove that $\frak{C}$ is a symmetric, monoidal category (Lemma \ref{lem monoidal category}). 
In Section \ref{section spanierwhitehead}, we will introduce the notion of the $S^1$-equivariant Spanier-Whitehead duality between the categories $\frak{S}$ and $\frak{S}^*$.   We will say that $Z \in \ob \frak{S}$ and $\bar{Z} \in \ob \frak{S}^*$ are $S^1$-equivariant Spanier-Whitehead dual to each other if there are elements
\[
   \epsilon \in \lim_{\infty \leftarrow m} \lim_{n \rightarrow \infty} \morp_{\frak{C}} (\bar{Z}_n \wedge Z_m, S), \
   \eta \in \lim_{\infty \leftarrow n} \lim_{m \rightarrow \infty} (S, Z_m \wedge \bar{Z}_n),
\]
which satisfy certain conditions (Definition \ref{def duality}). Here $S = (S^0, 0, 0) \in \frak{C}$.   The elements $\epsilon, \eta$ are called duality morphisms. 
In Section \ref{section dualswf}, we will prove that the Seiberg-Witten Floer stable spectra $\swf^A(Y) \in \ob \frak{S}$ and $\swf^{R}(-Y) \in \ob \frak{S}^*$ are $S^1$-equivariant Spanier-Whitehead dual to each other (Proposition \ref{prop SWF duality}). We will construct natural duality morphisms for $\swf^{A}(Y)$ and $\swf^{R}(-Y)$ which will be needed for the gluing formula of the Bauer-Furuta invariants (Theorem \ref{thm gluing BF inv}).

We will focus on the $S^1$-equivariant stable homotopy categories. But the statements can be proved for the $Pin(2)$-equivariant stable homotopy categories in a similar way.

\subsection{Smash product}
\label{section smashpr}

In this subsection, we establish the symmetric monoidal structure on the category $\frak{C}$. To do this, we will define
the smash product as a bifunctor $\wedge : \frak{C} \times \frak{C} \rightarrow \frak{C}$. First we define the smash product
of two objects $(A_1, m_1, n_1), (A_2, m_2, n_2) \in \frak{C}$. Here $A_i$ is an $S^1$-topological space, $m_i \in 2\Z$,
$n_i \in \Q$. We define the smash product by
\[
       (A_1, m_1, n_1) \wedge (A_2, m_2, n_2) := ( A_1 \wedge A_2, m_1+m_2, n_1+n_2),
\]
where $A_1 \wedge A_2$ denotes the classical smash product on pointed topological spaces. Next we define the smash product
of morphisms. Suppose that for $i = 1, 2$ a map
\[
     f_i : ( \R^{ k_i } \oplus \C^{l_i} )^+ \wedge A_i 
            \rightarrow 
           ( \R^{ k_i + m_i - m_i' } \oplus \C^{l_i + n_i - n_i'} )^+ \wedge A_i'
\]
represents a morphism $[f_i] \in \morp_{ \frak{C} }((A_i, m_i, n_i), (A_i', m_i', n_i'))$.  We may suppose  that $k_i$ is
even. 
We define a map
\begin{gather*}
       f_1 \wedge f_2 : (\R^{ k_1 } \oplus \R^{ k_2 } \oplus \C^{ l_1 } \oplus \C^{l_2} )^+ \wedge A_1 \wedge A_2
       \rightarrow   \\
       (\R^{ k_1 + m_1 - m_1' } \oplus \R^{ k_2 + m_2 - m_2' } \oplus \C^{l_1+n_1-n_1'} \oplus \C^{ l_2+n_2-n_2' } )^+  \wedge
A_1' \wedge A_2'                         
\end{gather*}
by putting the suspension indices for $f_1$ on the left and those for $f_2$ on the right.  We define $[f_1] \wedge [f_2]$
to be the morphism represented by $f_1 \wedge f_2$.  To prove that this operation is well defined, we need to check that
for $a, b \in \Z_{>0}$, we have
\[
         \Sigma^{ (\R^{a} \oplus \C^{b})^+ }  (f_1 \wedge f_2) \cong
         (\Sigma^{ (\R^{a} \oplus \C^{b})^+ }  f_1 ) \wedge f_2 \cong
         f_1 \wedge (\Sigma^{ (\R^{a} \oplus \C^{b})^{+} }  f_2 ),
\]
where $\cong$ means $S^1$-equivariant stably homotopic.  The first equivalence is obvious.  The second equivalence follows
from the fact that the following diagram is commutative up to homotopy for $u_1 = k_1, k_1+m_1-m_1'$, $u_2 = k_2, k_2+m_2-m_2'$,
$v_1 = l_1, l_1+n_1-n_1'$, $v_2=l_2, l_2+n_2-n_2'$:
\[
          \xymatrix{
( \R^{a} \oplus \R^{u_1} \oplus \R^{u_2} )^+ \wedge ( \C^{b} \oplus \C^{v_1} \oplus \C^{v_2}  )^+  \ar[rd]^{\id} \ar[dd]_{\substack{(\gamma_{
\R^a, \R^{u_1} } \oplus \id_{\R^{u_2}})^+ \\
 \qquad \wedge ( \gamma_{ \C^{b}, \C^{v_1} } \oplus \id_{ \C^{v_2} })^+}}  
                 &     \\    
                         &        (\R^{u})^+ \wedge ( \C^{v} )^+ \\
    (\R^{u_1} \oplus \R^{a} \oplus \R^{u_2} )^+ \wedge ( \C^{v_1} \oplus \C^{b} \oplus \C^{v_2} )^+ \ar[ru]_{\id} &   }
\]
Here $u = a+u_1+u_2$, $v= b + v_1 +v_2$
and $\gamma_{ \R^{a}, \R^{u_1} }$ is the map which interchange $\R^a$ and $\R^{u_1}$. Similarly for $\gamma_{\C^{b}, \C^{v_1}}$.
Note that $u_1 \in 2\Z$ by the assumption on $k_1, m_1, m_1'$. 

There is an isomorphism
\begin{gather*}
    \gamma_{(A_1, m_1, n_1), (A_2, m_2, n_2)} :   
   (A_1, m_1, n_1) \wedge (A_2, m_2, n_2)   \\
 \rightarrow    
 (A_2, m_2, n_2) \wedge
(A_1, m_1, n_1)
\end{gather*}
represented by the obvious homeomorphism $A_1 \wedge A_2 \rightarrow A_2 \wedge A_1$. It is not difficult to see that $\gamma$
 is natural in $(A_i, m_i, n_i)$. That is, the following diagrams are commutative for $f_i \in \morp_{\frak{C}}( (A_i, m_i,
n_i), (A_i', m_i', n_i') )$:
\[
       \xymatrix{
         (A_1, m_1, n_1) \wedge (A_2, m_2, n_2) \ar[r]^{\gamma} \ar[d]_{f_1 \wedge f_2} & (A_2, m_2, n_2) \wedge (A_1, m_1,
n_1) \ar[d]^{f_2 \wedge f_1} \\
         (A_1', m_1', n_1') \wedge (A_2', m_2', n_2') \ar[r]_{\gamma} & (A_2', m_2', n_2') \wedge (A_1', m_1', n_1'). 
       }
\]
(Again, we need the assumption that $m_i$ is even here.) Once the well-definedness of $\wedge$ and the naturality are established
we can prove the following lemma easily by checking the axioms at the level of topological spaces.

\begin{lem} \label{lem monoidal category}
The category $\frak{C}$ equipped with $\wedge$ and $\gamma$ is a symmetric monoidal category with unit $S=(S^0, 0, 0)$.
\end{lem}

We briefly mention the $Pin(2)$-case. The smash product $\wedge$ and the interchanging operation $\gamma$ can be defined
on the category $\frak{C}_{Pin(2)}$ in exactly the same way as before. As a result, the category $\frak{C}_{Pin(2)}$ is also
an symmetric monoidal category.

\subsection{Equivariant Spanier-Whitehead duality}
\label{section spanierwhitehead}

In this subsection we will set up the equivariant Spanier-Whitehead duality between the categories $\frak{S}$ and $\frak{S}^*$.
Although we will mostly focus on the $S^1$-case for simplicity, all definitions and proofs can be easily adapted to the $Pin(2)$-case.
As a result, a duality between $\frak{S}_{Pin(2)}$ and $\frak{S}_{Pin(2)}^*$ can also be set up in a similar way. 

The following definition is motivated by \cite[Chapter III]{LMS} and  \cite[Chapter XVI Section 7]{May}.

\begin{defi} 
Let $U, W$ be objects of $\frak{C}$ and put $S = (S^0, 0, 0) \in \ob \frak{C}$.  Suppose that there exist morphisms 
\[
         \epsilon : W \wedge U \rightarrow S, \ \eta : S \rightarrow U \wedge W
\]
such that the compositions
\[
      U \cong S \wedge U \xrightarrow{\eta \wedge \id} U \wedge W \wedge U
      \xrightarrow{\id \wedge \epsilon}  U \wedge S \cong U
\]
and 
\[
     W \cong W \wedge S \xrightarrow{\id \wedge \eta } W \wedge U \wedge W
          \xrightarrow{\epsilon \wedge \id} S \wedge W \cong W
\]
are equal to the identity morphisms respectively. Then we say that $U$ and $W$ are Spanier-Whitehead dual to each other and
call $\epsilon$ and $\eta$ duality morphisms. 
\end{defi}

We generalize this definition to the duality between $\frak{S}$ and $\frak{S}^*$.

\begin{defi}\label{def duality}
Let 
\[
     Z : Z_1 \rightarrow Z_2 \rightarrow Z_3 \rightarrow \cdots
\]
be an object of $\frak{S}$ and 
\[
     \bar{Z} : \bar{Z}_1 \leftarrow \bar{Z}_2 \leftarrow \bar{Z}_3 \leftarrow \cdots 
\]
be an object of $\frak{S}^*$.  Suppose that we have an element
\[
     \epsilon \in \lim_{\infty \leftarrow m} \lim_{n \rightarrow \infty} \morp_{ \frak{C} } (\bar{Z}_n \wedge Z_m, S)
\]
represented by a collection $\{ \epsilon_{m, n} : \bar{Z}_n \wedge Z_m \rightarrow S \}_{m > 0, n \gg m}$ and an element
\[
          \eta \in \lim_{\infty \leftarrow n} \lim_{m \rightarrow \infty} \morp_{\frak{C}}(S, Z_m \wedge \bar{Z}_n)
\]
represented by a collection $\{ \eta_{m, n} : S \rightarrow Z_m \wedge \bar{Z}_n \}_{  n > 0, m \gg n }$ which satisfy the
following conditions: 

\begin{enumerate}[(i)]
\item
For any $m > 0$ there exists $n$ large enough relative to $m$ and $m'$  large enough relative to $n$ such that the composition
\[
     Z_m \cong S \wedge Z_m \xrightarrow{\eta_{m', n} \wedge \id}  Z_{m'} \wedge \bar{Z}_n \wedge Z_m  
     \xrightarrow{\id \wedge \epsilon_{m, n}} Z_{m'} \wedge S \cong Z_{m'}
\]
is equal to the connecting morphism $Z_{m} \rightarrow Z_{m'}$ of the inductive system $Z$. 

\item
For any $n > 0$, there exists $m$ large enough relative $n$ and $n'$ large enough to $m$ such that the composition
\[
  \bar{Z}_{n'} \cong \bar{Z}_{n'} \wedge S \xrightarrow{\id \wedge \eta_{m, n}} 
  \bar{Z}_{n'} \wedge Z_m \wedge \bar{Z}_n \xrightarrow{\epsilon_{m, n'} \wedge \id}
  S \wedge \bar{Z}_{n} \cong \bar{Z}_n
\]
is equal to the connecting morphism $\bar{Z}_{n'} \rightarrow \bar{Z}_n$ of the projective system $\bar{Z}$.

\end{enumerate}
Then we say that $Z$ and $\bar{Z}$ are $S^1$-equivariant Spanier-Whitehead dual to each other and we call $\epsilon$ and
$\eta$ duality morphisms. 

\end{defi}

We end this subsection with introducing a smashing operation $\tilde{\epsilon}$, which will be used to give  the statement
of the gluing theorem for the Bauer-Furuta invariant. 

\begin{defi}
Let $Z \in \ob \frak{S}$ and $\bar{Z} \in \ob \frak{S}^*$ be objects that are $S^1$-equivariant Spanier-Whitehead dual to
each other with duality morphisms $\epsilon, \eta$.  Suppose that we have objects $W \in \ob \frak{C} \ (\subset \ob \frak{S})$,
$\bar{W} \in \ob \frak{C} \ (\subset \ob \frak{S}^*)$ and morphisms
\[
     \rho \in \morp_{ \frak{S} } (W, Z),  \  \bar{\rho}  \in \morp_{ \frak{S}^* }(\bar{W}, \bar{Z}).
\]
Choose a morphism $\rho_m : W \rightarrow Z_m$ which represents $\rho$ and let $\{ \bar{\rho}_{n} : \bar{W} \rightarrow \bar{Z}_{n}
\}_{n > 0}$ be the collection which represents $\bar{\rho}$. 
We define the morphism $\tilde{\epsilon}(\rho, \bar{\rho}) \in \morp_{ \frak{C} }( W \wedge \bar{W}, S )$ by the composition
\[
       \bar{W} \wedge   W  \xrightarrow{ \bar{\rho}_n \wedge \rho_{m} }  \bar{Z}_n \wedge Z_{m} \xrightarrow{\epsilon_{m,
n}}  S.
\]
 It can be proved that $\tilde{\epsilon}(\rho, \bar{\rho})$ does not depend on the choices of $m, n$ and $\rho_m$.  (Note
that $\bar{\rho}_n$ is determined by $n$ and $\bar{\rho}$.)

\end{defi}

\subsection{Spanier-Whitehead duality of the unfolded Seiberg-Witten Floer spectra}
\label{section dualswf}

Let $Y$ be a closed, oriented 3-manifold with a Riemannian metric $g$ and spin$^c$ structure $\frak{s}$, and let $-Y$
be $Y$ with opposite orientation.  As in Section~\ref{subsec Unfolded}, the unfolded Seiberg-Witten Floer spectrum $\swf^{A}(Y, \frak{s}, A_0, g; S^1) \in \ob
\mathfrak{S}$  is represented by 
\[
   \swf^{A}(Y) : I_{1} \xrightarrow{j_{1}} I_2 \xrightarrow{j_{2}} \cdots   
\]
with $I_n := \Sigma^{-V_{\lambda_n}^{0}} I( \inv( V_{\lambda_n}^{\mu_n} \cap J_n^+), \varphi_n) $. It is not hard to see that the unfolded spectrum  $\swf^{R}(-Y, \frak{s}, A_0, g; S^1) \in \ob \mathfrak{S}^*$ can be represented by
\[ \swf^{R} (-Y): \bar{I}_1 \xleftarrow{\bar{j}_1} \bar{I}_2 \xleftarrow{\bar{j}_{2}} \cdots, \]
where $\bar{I}_n := \Sigma^{ -V^{\mu_n}_{0} } I ( \inv(V_{\lambda_n}^{\mu_n} \cap J_n^+), \overline{\varphi}_n) $ and $\overline{\varphi}_n$ is the reverse flow of $\varphi_n$. For integers $m, n$ with $m < n$ we also write $j_{m, n}$, $\bar{j}_{m,n}$ for the compositions
\[
  \begin{split}
      & I_{m} \xrightarrow{j_{m}} I_{m+1} \xrightarrow{j_{m+1}} \cdots \xrightarrow{j_{n-1}}   I_{n}, \\
      & \bar{I}_{n} \xrightarrow{\bar{j}_{n-1}} \bar{I}_{n-1} \xrightarrow{\bar{j}_{n-2}} \cdots \xrightarrow{\bar{j}_{m}}
\bar{I}_{m}.
   \end{split}
\]


%


We will define duality morphisms $\epsilon$ and $\eta$ between  $\swf^{A}(Y, \frak{s}, A_0, g; S^1)$ and $\swf^{R}(-Y, \frak{s},
A_0, g; S^1)$.  as follows.  Take an $S^{1}$-equivariant manifold isolating block $N_n$ for $\inv ( V_{\lambda_n}^{\mu_n}
\cap J_n^+)$.  That is, $N_n$ is a compact submanifold of $V_{\lambda_n}^{\mu_n}$ of codimension $0$ and there are submanifolds
$L_n, \overline{L}_n$ of $\partial N_n$ of codimension $0$ such that
\[
       L_n \cup \overline{L}_n = \partial N_n, \ \partial L_n = \partial \overline{L}_n = L_n \cap \overline{L}_n
\]
and that $(N_n, L_n)$, $(N_n, \overline{L}_n)$ are index pairs for $\inv( V_{\lambda_n}^{\mu_n} \cap J_n^+, \varphi_n)$,
$\inv (V_{\lambda_n}^{\mu_n} \cap J_n^+, \overline{\varphi}_n)$ respectively. 
Fix a small positive number $\delta > 0$.  For a subset $P \subset V_{\lambda_{n}}^{\mu_n}$ we write $\nu_{\delta}(P)$ for

\[
      \{ x \in V_{\lambda_n}^{\mu_n} | \dist(x, P) \leq \delta \}.
\]
Choose  $S^1$-equivariant homotopy equivalences
\[
       a_n : N_n \rightarrow N_n \setminus \nu_{\delta}( \overline{L}_n), \
       b_n : N_{n} \rightarrow N_{n} \setminus \nu_{\delta}( L_n )
\]
such that
\begin{equation}  \label{eq a b}
   \begin{split}
      &  \| a_n(x) - x \| < 2\delta \ \text{for $x \in N_{n}$},  \\  
      &     a_n(L_n) \subset L_n,    \ 
           a_n(x) = x  \
 \text{for $x \in N_n \setminus \nu_{3\delta}(\partial N_{n})$,} \\
      &   \| b_n(y) - y \| < 2 \delta \ \text{for $y \in N_n$}, \\
 &     b_n( \overline{L}_n) \subset \overline{L}_n, \
            b_n(y) = y \ \text{for $y \in N_n \setminus \nu_{3\delta}(\partial N_{n} )$}.
   \end{split}
\end{equation}
Put $B_{\delta} = \{ x \in V_{\lambda_n}^{\mu_n} | \| x \| \leq \delta \}$ and $S_{\delta} = \partial B_{\delta}$.   Define

\[  
    \hat{\epsilon}_{n, n} : (N_n/ \overline{L}_n) \wedge (N_n / L_n) \rightarrow   
      (V_{\lambda_n}^{\mu_n})^+ = B_{\delta} / S_{\delta} 
\]
by the formula
\[
      \hat{\epsilon}_{n, n}([y] \wedge [x]) = 
      \left\{
         \begin{array}{ll}
             [b_n(y) - a_n(x)] & \text{if $\| b_n(y) - a_n(x) \| < \delta$} \\
             * & \text{otherwise.}
         \end{array}
      \right.
\]
It is straightforward to see that $\hat{\epsilon}_n$ is a well-defined, continuous $S^1$-equivariant map.  Taking the desuspension by
$V_{\lambda_n}^{\mu_n}$ we get a morphism
\[
        \epsilon_{n,n} : \bar{I}_n \wedge I_{n} \rightarrow S.
\]
For $m, n$ with $m < n$,  we define a morphism $\epsilon_{m, n} : \bar{Z}_{n} \wedge Z_{m} \rightarrow S$ to be the composition
\[
     \bar{I}_{n} \wedge I_{m} \xrightarrow{\id \wedge j_{m,n}} \bar{I}_{n} \wedge I_{n} \xrightarrow{\epsilon_{n,n}} S.
\]

\begin{lem} \label{lem independence a b}
With the above notation, the morphism $\epsilon_{m, n} \in \morp_{ \frak{C} }( \bar{I}_n \wedge I_{m}, S )$ is independent
of the choices of $N_{n}$, $a_n, b_n$ and $\delta$.
\end{lem}

\begin{proof}
The proof of the independence from $\delta$ is straightforward. We prove the independence from $N_n$, $a_n$ and $b_{n}$. Fix an isolating
neighborhood $A (\subset V_{\lambda_n}^{\mu_n} \cap J_n^+)$ of $\inv(V_{\lambda_n}^{\mu_n} \cap J_n^+)$.  Take two manifold
isolating blocks $N_n, N_n'$ for $\inv ( V_{\lambda_n}^{\mu_n} \cap J_{n}^+)$ included in $\inti A$.  Then we get two maps
\[
   \hat{\epsilon}_{n, n} :  (N_n / \overline{L}_n) \wedge (N_{n} / L_{n}) \rightarrow B_{\delta} / S_{\delta}, \
   \hat{\epsilon}_{n, n}' :  (N_n' / \overline{L}_n') \wedge (N_{n}' / L_{n}') \rightarrow B_{\delta} / S_{\delta}.
\]
It is sufficient to show that the following diagram is commutative up to $S^1$-equivariant homotopy:
\[
     \xymatrix{
       (N_n / \overline{L}_n) \wedge (N_n / L_{n})  \ar[r]^(0.65){\hat{\epsilon}_{n,n}} \ar[d]_{ \bar{s} \wedge s}  &  B_{\delta}
/ S_{\delta}   \\
       (N_{n}' / \overline{L}_n')  \wedge (N_{n}' / L_{n}')  \ar[ru]_{\hat{\epsilon}_{n,n}'}
     }
\]
Here $s = s_{T} : N_{n} / L_{n} \rightarrow N_{n}' / L_{n}'$, $\bar{s} = \bar{s}_{T} : N_{n} / \overline{L}_n \rightarrow
N_{n}' / \overline{L}_n'$ are the flow maps with large $T  > 0$:
\[  
\begin{split}
    &   \qquad  s([x])  = \\
    &   \left\{
                       \begin{array}{ll}
          [\varphi(x, 3T)] & 
          \text{if $\varphi(x, [0, 2T]) \subset N_n \setminus L_{n}, \varphi(x, [T, 3T]) \subset N_{n}' \setminus L_{n}',$}
 \\
          * & \text{otherwise,}
                       \end{array}    
                   \right.     \\
   &  \qquad   \bar{s}([y]) =   \\
  &   \left\{
                       \begin{array}{ll}
          [\varphi(y, -3T)] & 
          \text{if $\varphi(y, [-2T,0]) \subset N_n \setminus \overline{L}_{n}, \varphi(y, [-3T, -T]) \subset N_{n}' \setminus
\overline{L}_{n}',$}  \\
          * & \text{otherwise.}
                       \end{array}    
                   \right.
     \end{split}
     \]
The proof can be reduced to the case $N_n' \subset \inti N_n$ since we can find a manifold isolating block $N''_{n}$ with
$N_n'' \subset  \inti N_{n},  \inti N_{n}'$.  Assume that $N_n' \subset \inti N_{n}$. Taking sufficiently large $T > 0$ we
have 
\begin{equation} \label{eq A frac{1}{2}T}
       A^{[-T, T]} \subset 
       (N_{n}' \setminus \nu_{3\delta}(\partial N_{n}') )  \subset (N_{n} \setminus \nu_{3\delta}( \partial N_{n})). 
\end{equation}
It is straightforward to see that $\hat{\epsilon}_{n,n}$ is homotopic to a map $\hat{\epsilon}_{n,n}^{(0)} : (N_n / \overline{L}_n)
\wedge (N_{n} / L_{n}) \rightarrow B_{\delta} / S_{\delta}$ defined by
\[
   \hat{\epsilon}_{n,n}^{(0)}( [y] \wedge [x]) =    \left\{
              \begin{array}{ll}
                [\Delta_{T,n}(y, x)] & \text{if}
                   \left\{
                      \begin{array}{l}
                         \varphi(x, [0, 3T]) \subset N_{n} \setminus L_{n}, \\
                          \varphi(y, [-3T, 0]) \subset N_n \setminus \overline{L}_{n},  \\
                          \| \Delta_{T,n}(y,x) \| < \delta
                      \end{array}
                   \right.    \\
               * & \text{otherwise.}
              \end{array}
          \right.
\]
Here 
\[
    \Delta_{T,n}( y, x) = b_n(\varphi(y,-3T)) - a_n(\varphi(x, 3T)).
\]
Suppose that $\epsilon^{(0)}( [y] \wedge [x]) \not= *$. Then 
\[
     \varphi(x, 3T) \in N_{n}^{[-3T, 0]}, \varphi(y, -3T) \in N_{n}^{[0, 3T]}, \ \| \varphi(y, -3T) - \varphi(x, 3T) \| <
5\delta.
\]
Taking small $\delta > 0$ and the using the fact that $N_n \subset \inti A$,  we may suppose that
\[
            \varphi(x, 3T), \varphi(y, -3T) \in A^{[-3T, 3T]},
\]
which implies 
\[
       a_n( \varphi(x, 3T) ) = \varphi(x, 3T), \ b_n(\varphi(y, -3T)) = \varphi(y, -3T). 
\]
Here we have used (\ref{eq a b}) and (\ref{eq A frac{1}{2}T}). 
We can assume that $\delta$ is independent of $x, y$ since $N_n$ is compact. 
So we have
\[
   \begin{split}
  &   \hat{\epsilon}^{(0)}_{n,n}( [y] \wedge [x]) =  \\
 &    \qquad     \left\{
              \begin{array}{ll}
                   [\varphi(y, -3T) -   \varphi(x, 3T)] & \text{if}
                   \left\{
                      \begin{array}{l}
                         \varphi(x, [0, 3T]) \subset N_{n} \setminus L_{n}, \\
                          \varphi(y, [-3T, 0]) \subset N_n \setminus \overline{L}_{n},  \\
                          \| \varphi(y, -3T) - \varphi(x, 3T) \| < \delta,
                      \end{array}
                   \right.    \\
               * & \text{otherwise.}
              \end{array}
          \right.
     \end{split}
\]
 On the hand, we can write
\[
\begin{split}
    &    \hat{\epsilon}'_{n,n} \circ (s \wedge \bar{s})([y] \wedge [x]) =   \\
 &     \qquad  \left\{
              \begin{array}{ll}
                   [\Delta_{T,n}'(y,x)] & \text{if}
                   \left\{
                      \begin{array}{l}
                         \varphi(x, [0, 2T]) \subset N_{n} \setminus L_{n}, \\
                         \varphi(x, [T, 3T]) \subset N_{n}' \setminus L_{n}', \\
                          \varphi(y, [-2T, 0]) \subset N_n \setminus \overline{L}_n,  \\
                          \varphi(y, [-3T, -T]) \subset N_{n}' \setminus \overline{L}_n', \\
                          \| \Delta_{T,n}'(y,x) \| < \delta,
                      \end{array}
                   \right.    \\
               * & \text{otherwise.}
              \end{array}
          \right.
  \end{split}
\]
Here
\[
        \Delta_{T,n}'(y, x) =
        b_n'(\varphi(y, -3T)) -   a_{n}'(\varphi(x, 3T))
\]
As before, if $\hat{\epsilon}'_{n,n} \circ (s \wedge \bar{s})([y] \wedge [x]) \not= *$ we have
\[
     \varphi(x, 3T), \varphi(y, -3T) \in A^{[- 3T, 3T]}
\]
and we can write
\[
  \begin{split}
     &  \hat{\epsilon}_{n,n}' \circ (s \wedge \bar{s})([y] \wedge [x]) =   \\
    & \qquad      \left\{
              \begin{array}{ll}
                  [\varphi(y, -3T) -   \varphi(x, 3T) ] & \text{if}
                   \left\{
                      \begin{array}{l}
                         \varphi(x, [0, 2T]) \subset N_{n} \setminus L_{n}, \\
                         \varphi(x, [T, 3T]) \subset N_{n}' \setminus L_{n}', \\
                          \varphi(y, [-2T, 0]) \subset N_n \setminus \overline{L}_n,  \\
                          \varphi(y, [-3T, -T]) \subset N_{n}' \setminus \overline{L}_n', \\
                          \| \varphi(y, -3T)) -  \varphi(x, 3T) \| < \delta,
                      \end{array}
                   \right.    \\
               * & \text{otherwise.}
              \end{array}
          \right.
    \end{split}
\]
We will show that $\hat{\epsilon}_{n,n}^{(0)} =  \hat{\epsilon}_{n,n}' \circ (s \wedge \bar{s})$.  It is sufficient to prove
that $\hat{\epsilon}^{(0)}_{n,n}([y] \wedge [x]) \not= *$ if and only if $\hat{\epsilon}_{n,n}' \circ (s \wedge \bar{s})([y]
\wedge [x]) \not=*$.  It is straightforward to see that if $\hat{\epsilon}_{n,n}' \circ (s \wedge \bar{s})([y] \wedge [x]) \not=*$ then
$\hat{\epsilon}_{n,n}^{(0)}([y] \wedge [x]) \not= *$ using the assumption that $N_{n}' \subset \inti N_{n}$.  Conversely,
 suppose that $\hat{\epsilon}^{(0)}_{n,n}([y] \wedge [x]) \not= *$.  Then  $\varphi(x, 3T), \varphi(y, -3T) \in A^{ [-3T,
3T] }$ and we have
\[
    \begin{split}
    &  \varphi(x, [2T, 3T]) =  \varphi(\varphi(x, 3T), [-T, 0])  \subset A^{ [-2T, 2T]  } \subset  \inti N_n,  \\
    & \varphi(y, [-3T, 2T]) = \varphi( \varphi(y, -3T), [0, T]  )  \subset A^{  [- 2T, 2T] }  \subset \inti N_n.
    \end{split}
\]
This implies that $\hat{\epsilon}_{n,n}^{(0)}( [y] \wedge [x]) \not= *$.

\end{proof}

A calculation similar to that in the proof of Lemma \ref{lem independence a b} proves the following:

\begin{lem}
Suppose that $\lambda < \lambda_n$, $\mu > \mu_{n}$.  Take $S^{1}$-equivariant manifold isolating blocks $N_n, N_n'$ for $\inv ( V_{\lambda_n}^{\mu_n}
\cap J_n^+)$, $\inv( V_{\lambda}^{\mu} \cap J_n^+)$. Note that we have canonical homotopy equivalences
\[
       \Sigma^{ V_{\lambda}^{\lambda_n} } (N_{n} / L_{n}) \cong N_{n}' / L_{n}', \
       \Sigma^{ V_{\mu_n}^{\mu} } (N_{n} / \overline{L}_n) \cong N_{n}' / \overline{L}_n'. 
\]
See Proposition 5.6 of \cite{KLS1}. 
The following diagram is commutative up to  $S^1$-equivariant homotopy:
\[
    \xymatrix{
 \Sigma^{ V_{\mu_n}^{\mu} }(N_n / \overline{L}_n) \wedge \Sigma^{ V_{\lambda}^{\lambda_n} } (N_n / L_n) 
 \ar[rr]^(0.7){ \Sigma^{ W }  \hat{\epsilon}_{n,n} } \ar[d]_{\cong}
               &  & (V_{\lambda}^{\mu})^+ \\
        (N_n' / \overline{L}_n') \wedge (N_n' / L_n') \ar[rru]_{ \hat{\epsilon}_{n,n}'}
    }
\]
Here $W = V_{\lambda}^{\lambda_n} \oplus V_{\mu_n}^{\mu}$. 
\end{lem}

This lemma implies that the morphism $\epsilon_{n,n}$ (and hence $\epsilon_{m,n}$) is independent of the choices of $\lambda_n,
\mu_n$.

We have obtained a collection $\{ \epsilon_{m, n} : \bar{I}_{n} \wedge I_{m} \rightarrow S \}_{ n \geq m }$ of morphisms.
 Since $j_{m, n} = j_{m+1, n} \circ j_{m, m+1}$,  the following diagram is commutative:
\begin{equation} \label{eq commu epsilon j}
         \xymatrix{
             \bar{I}_{n} \wedge I_{m} \ar[d]_{\id \wedge j_{m, m+1}} \ar[r]^{\epsilon_{m,n}}   & S \\
             \bar{I}_n \wedge I_{m+1} \ar[ru]_{\epsilon_{m+1, n}} &
         }
\end{equation}

\begin{lem} \label{lem commu epsilon bar{j}}
For $m < n$, the following diagram is commutative: 
\begin{equation} \label{eq commu epsilon bar{j}}
       \xymatrix{
           \bar{I}_n \wedge I_{m} \ar[r]^{\epsilon_{m,n}}        & S  \\
           \bar{I}_{n+1} \wedge I_{m} \ar[u]^{\bar{j}_{n, n+1} \wedge \id}  \ar[ru]_{\epsilon_{m, n+1} } &
       }
\end{equation}
\end{lem}

\begin{proof}
We have to prove that the following diagram is commutative up to $S^1$-equivariant homotopy:
\begin{equation}  \label{eq bar{i} epsilon}
     \xymatrix{
     (N_{n} /  \overline{L}_{n}) \wedge ( N_m / L_{m} ) \ar[r]^(0.65){\hat{\epsilon}_{n, m}}  
                 & B_{\delta} / S_{\delta}    \\
     (N_{n+1} /  \overline{L}_{n+1}) \wedge (N_{m} / L_{m}) 
                 \ar[ru]_{\hat{\epsilon}_{m, n+1}} \ar[u]^{ \bar{i}_{n, n+1} \wedge \id}
     }
\end{equation}
By Lemma \ref{lem independence a b}, we can use the following specific manifold isolating blocks (with corners).  First take
a manifold isolating block $N_{n+1}$ for $\inv (V_{\lambda_{n+1}}^{\mu_{n+1}} \cap J_{n+1}^+ )$.   We have compact submanifolds
$L_{n+1}, \overline{L}_{n+1}$ in $\partial N_{n+1}$ with
\[
        \partial N_{n+1} = L_{n+1} \cup \overline{L}_{n+1}, \  
        \partial L_{n+1} = \partial \overline{L}_{n+1}  = L_{n+1} \cap \overline{L}_{n+1}.
\]
Moreover $(N_{n+1}, L_{n+1})$ is an index pair for $(\inv (V_{\lambda_{n+1}}^{\mu_{n+1}} \cap J_{n+1}^+), \varphi_{n+1})$
and $(N_{n+1}, \overline{L}_{n+1})$ is an index pair for $(\inv(V_{\lambda_{n+1}}^{\mu_{n+1}} \cap J_{n+1}^+), \overline{\varphi}_{n+1})$,
where $\overline{\varphi}_{n+1}$ is the reverse flow of $\varphi_{n+1}$.
Put
\[
   \begin{split}
        &  N_{m} := N_{n+1} \cap J_{m}^+ = N_{n+1} \cap \bigcap_{j=1}^{b_1} g_{j, +}^{-1}((-\infty, m+ \theta]), \\
        & L_{m} := L_{n+1} \cap N_{m}, \\
        & \overline{L}_{m} := (\overline{L}_{n+1} \cap N_{m}) \cup \bigcup_{j=1}^{b_1} N_{m} \cap g_{j,+}^{-1}(m+\theta),
\\
        &  N_{n} := N_{n+1} \cap J_{n}^+ = N_{n+1} \cap \bigcap_{j=1}^{b_1} g_{j, +}^{-1}((-\infty, n + \theta]), \\
        & L_{n} := L_{n+1} \cap N_{n}, \\
        & \overline{L}_{n} :=  (\overline{L}_{n+1} \cap N_{n}) \cup \bigcup_{j=1}^{b_1} N_{n} \cap g_{j,+}^{-1}(n+\theta)
   \end{split}
\]
Then $N_{m}, N_{n}$ are isolating blocks for $\inv (V_{\lambda_{n+1}}^{\mu_{n+1}}  \cap J_m^+)$, $\inv ( V_{\lambda_{n+1}}^{\mu_{n+1}}
\cap   J_{n}^+)$ and $N_{m}$, $N_{n}$, $L_{m}$, $\overline{L}_{m}$, $L_n$, $\overline{L}_{n}$ are manifolds with corners
(for generic $\theta$). Moreover $(N_{m}, L_{m})$, $(N_{m}, \overline{L}_{m})$, $(N_{n}, L_{n})$, $(N_{n}, \overline{L}_{n})$
are index pairs for $(\inv (V_{\lambda_{n+1}}^{\mu_{n+1}} \cap J_{m}^+), \varphi_{n+1} )$,  $(\inv (V_{\lambda_{n+1}}^{\mu_{n+1}}
\cap J_{m}^+), \overline{\varphi}_{n+1} )$, $(\inv (V_{\lambda_{n+1}}^{\mu_{n+1}} \cap J_{n}), \varphi_{n+1} )$, $(\inv (V_{\lambda_{n+1}}^{\mu_{n+1}}
\cap J_{n}), \overline{\varphi}_{n+1} )$ respectively.  Also we have
\[
   \begin{split}
       & L_{m} \cup \overline{L}_m = \partial N_{m}, \   \partial L_{m} = \partial \overline{L}_{m} = L_m \cap \overline{L}_m,
\\
       & L_{n} \cup \overline{L}_{n} = \partial N_{n}, \   \partial L_{n} = \partial \overline{L}_{n} = L_m \cap \overline{L}_{n}.
   \end{split}
\]

The connecting morphisms $j_{m, n} : I_{m} \rightarrow I_{n}$, $j_{m, n+1} : I_{m} \rightarrow I_{n+1}$ and $\bar{j}_{n,n+1}
: \bar{I}_{n+1} \rightarrow \bar{I}_{n}$ are induced by the inclusions
\[
     i_{m, n} : N_{m} / L_{m} \rightarrow N_{n} / L_{n}, \   i_{m, n+1} :  N_{m} / L_{m} \rightarrow N_{n+1} / L_{n+1}
\]
and projection 
\[
\begin{split}
   & \bar{i}_{n, n+1} :   N_{n+1} / \overline{L}_{n+1} \rightarrow  \\
& \qquad  N_{n+1} \left/ \left\{  \overline{L}_{n+1} \cup \bigcup_{j}
\left( N_{n+1} \cap  g_{j, +}^{-1}([n+\theta, \infty))  \right) \right\} \right.
     = N_{n} / \overline{L}_{n} .
 \end{split}
\]
With the index pairs we have taken above,  for $x \in N_m, y \in N_{n+1}$ we can write
\[
      \hat{\epsilon}_{m, n+1}( [y] \wedge [x]) = 
         \left\{
            \begin{array}{ll}
                 [b_{n+1}(y) - a_{n+1}(x)] & \text{if $\| b_{n+1}(y) - a_{n+1}(x) \| < \delta$, }  \\
                 * & \text{otherwise. }
             \end{array}
         \right. 
\]
Also we have
\[
\begin{split}
   &      \hat{\epsilon}_{m, n} \circ (\bar{i}_{ n, n+1 } \wedge \id) ([y] \wedge [x]) =  \\
   &  \qquad
           \left\{
              \begin{array}{ll}
                   [b_n(y) - a_{n}(x)]   & \text{if $y \in N_{n}$, $\| b_n(y) - a_n(x) \| < \delta$,}   \\
                   * & \text{otherwise.}
              \end{array}
              \right.
    \end{split}
\]
We may suppose that $a_{n}(x) = a_{n+1}(x)$ for $x \in N_{m}$.  Note that if $\hat{\epsilon}_{m, n+1}([y] \wedge [x]) \not=
*$ or $\hat{\epsilon}_{m, n} \circ ( \bar{i}_{n,n+1} \wedge \id ) ([y] \wedge [x]) \not= *$ we have $y \in \nu_{5\delta}(N_{m})$.
 For small $\delta > 0$  we can suppose that $\nu_{5\delta}(N_m) \cap N_{n+1} \subset N_{n}$ and that $b_{n+1}(y) = b_{n}(y)$
for $y \in \nu_{5\delta}(N_m) \cap N_{n+1}$.  This implies that (\ref{eq bar{i} epsilon}) commutes. 

 \end{proof}

The commutativity of the diagrams (\ref{eq commu epsilon j}) and (\ref{eq commu epsilon bar{j}}) means that the collection
$\{ \epsilon_{m, n} \}_{m, n}$ defines an element $\epsilon$ of ${\displaystyle \lim_{\infty \leftarrow m} \lim_{n \rightarrow
\infty} } \morp_{\frak{C}}( \bar{I}_n \wedge I_m, S)$.

Next we will define ${\displaystyle \eta \in \lim_{\infty \leftarrow n} \lim_{m \rightarrow \infty}  \morp_{\frak{C}}( S,
I_{m} \wedge \bar{I}_{n} )}$. Take a manifold isolating block $N_{n} (\subset V_{\lambda_n}^{\mu_n})$ of $\inv( V_{\lambda_n}^{\mu_n}
\cap J_n^+)$.  As usual we have compact submanifolds $L_{n}, \overline{L}_{n}$ of $\partial N_{n}$ such that 
\[
              \partial N_{n} = L_n \cup \overline{L}_{n}, \ \partial L_{n} = \partial \overline{L}_n = L_{n} \cap \overline{L}_n
\]
and that $(N_n, L_n)$, $(N_n, \overline{L}_n)$ are index pairs for $(\inv (V_{\lambda_n}^{\mu_n} \cap J_{n}^+), \varphi_{n})$,
$(\inv (V_{\lambda_n}^{\mu_n} \cap J_{n}^+), \overline{\varphi}_n)$ respectively. 
  Taking a large positive number $R > 0$ we may suppose that $N_{n} \subset B_{R/2}$, where $B_{R/2} = \{ x \in V_{\lambda_n}^{\mu_n}
| \| x \| \leq R / 2 \}$. 
We  define 
\[
     \hat{\eta}_{n, n} : (V_{\lambda_n}^{\mu_n})^+ = B_{R} / S_{R} \rightarrow (N_n / L_n) \wedge (N_n / \overline{L}_n)
\]
by
\[
    \hat{\eta}_{n,n}( [x] ) =
       \left\{
           \begin{array}{ll}
               [x] \wedge [x] & \text{if $x \in  N_n$, }  \\
               *  & \text{otherwise.}
           \end{array}
       \right.
\]
We can see that $\hat{\eta}_{n,n}$ is a well-defined continuous map and induces a morphism
\[
        \eta_{n,n} : S  \rightarrow  I_n \wedge \bar{I}_n.
\]
For $m > n$, we define $\eta_{m, n} : S \rightarrow  I_{m} \wedge \bar{I}_{n}$ to be the composition
\[
         S  \xrightarrow{\eta_{n,n}}  I_{n} \wedge \bar{I}_{n} \xrightarrow{j_{n, m} \wedge \id}  I_{m} \wedge \bar{I}_{n}.
\]

\begin{lem} \label{lem eta independence}
The morphism $\eta_{m, n} \in \morp_{\frak{S}}(S, I_{m} \wedge \bar{I}_n)$ is independent of the choices of  $R$ and $N_{n}$.

\end{lem}

\begin{proof}
The independence from $R$ is straightforward. We prove the independence from the choice of $N_n$. 
Take another manifold isolating block $N_{n}'$ of $\inv ( V_{\lambda_n}^{\mu_n} \cap J_n^+ )$. We may assume that $N_n, N_n'
\subset A$ for an isolating neighborhood $A$ of $\inv ( V_{\lambda_n}^{\mu_n} \cap J_{n}^+ )$.   It is sufficient to show
that the following diagram is commutative up to $S^1$-equivariant homotopy:
\[
   \xymatrix{
          B_{R} / S_{R} \ar[r]^(0.35){\hat{\eta}_{n,n}} \ar[rd]_{\hat{\eta}'_{n,n}}  
                                  & (N_n / L_{n}) \wedge (N_{n} / \overline{L}_n)   \ar[d]^{ s \wedge \bar{s} }  \\
                                                    & (N_n' / L_{n}') \wedge (N_{n}' / \overline{L}_n')
     }
\]
Here $s = s_{T}, \bar{s} = \bar{s}_{T}$ are the flow maps with $T \gg 0$.  For $x \in B_R$  we have
\[
   \begin{split}
    &   (s \wedge \bar{s}) \circ \hat{\eta}_{n, n}([x]) =  \\
    &      \left\{
          \begin{array}{ll}
                [ \varphi(x, 3T) ] \wedge [ \varphi(x, -3T) ] & 
                         \text{if}  
                         \left\{
                         \begin{array}{l}
      \varphi(x, [0, 2T]) \subset N_n \setminus L_{n}, \\
    \varphi(x, [T, 3T]) \subset N_{n}' \setminus L_{n}', \\
     \varphi(x, [-2T, 0]) \subset N_{n} \setminus \overline{L}_n, \\
      \varphi(x, [-3T, -T]) \subset N_{n}' \setminus \overline{L}_{n}',
                         \end{array}    
                           \right.  \\
                  * & \text{otherwise}
          \end{array}
        \right.
   \end{split}
\]
and
\[
     \hat{\eta}_{n,n}'([x]) =
       \left\{
          \begin{array}{ll}
            [x] \wedge [x] & \text{if $x \in \inti N_n'$,}  \\
            * & \text{otherwise.}
          \end{array}
       \right.
\]
We can reduce the proof to the case $N_n \subset \inti N_{n}'$.   
Suppose $N_n \subset \inti N_{n}'$. Also we may assume that $A^{[-T, T]} \subset \inti N_{n}$, choosing a sufficiently large
$T$.  If $(s \wedge \bar{s}) \circ \hat{\eta}_{n,n}([x]) \not= *$, we have 
\[
         \varphi(x, [-3T, 3T]) \subset \inti N_{n}'. 
\]
Conversely, suppose that $\varphi(x, [-3T, 3T]) \subset \inti N_{n}'$. Then we have $x \in A^{[-3T, 3T]}$.  Hence
\[
       \varphi(x, [-2T, 2T]) \subset A^{[-T, T]} \subset \inti N_{n}. 
\]
Therefore $\varphi(x, [0, 2T]) \subset N_{n} \setminus L_{n}, \varphi(x, [-2T, 0]) \subset N_{n} \setminus \overline{L}_{n}$.
 Thus $(s \wedge \bar{s}) \circ \hat{\eta}_{n,n}([x]) \not= *$.  We have obtained:
\[
 \begin{split}
    &    (s \wedge \bar{s}) \circ \hat{\eta}_{n,n}([x]) =   \\
     &    \left\{
          \begin{array}{ll}
                [ \varphi(x, 3T) ] \wedge [ \varphi(x, -3T) ] & \text{if} \   \varphi(x, [-3T, 3T]) \subset  \inti N_{n}',
\\
                           * & \text{otherwise.}
          \end{array}
        \right.
   \end{split}
\]
This is homotopic to $\hat{\eta}'_{n,n}$ through a homotopy $H$ which maps $([x], s)$ to 
\[
       [ \varphi(x, 3(1-s)T) ] \wedge [ \varphi(x, -3(1-s)T) ]
\]
if 
\[
        \varphi(x, [-3(1-s)T, 3(1-s)T]) \subset
 \inti N_{n}'
\]
and to the basepoint otherwise. 
\end{proof}

\begin{lem}
Let $\lambda < \lambda_n, \mu > \mu_n$. Take manifolds index pairs $N_n, N_n'$ for $\inv(J_n \cap V_{\lambda_n}^{\mu_n})$,
$\inv( J_n \cap V_{\lambda}^{\mu} )$. Then we have the canonical $S^1$-equivariant homotopy equivalence:
\[
      \Sigma^{V_{\lambda}^{\lambda_n}} ( N_n / L_n ) \cong N_n' / L_n',  \
      \Sigma^{ V_{\mu_n}^{\mu} } (N_n / \overline{L}_n) \cong N_n'/ \overline{L}_n'.
\]
See Proposition 5.6 of \cite{KLS1}. 
The following diagram is commutative up to $S^1$-equivariant homotopy:
\[
    \xymatrix{
         (V_{\lambda}^{\mu})^+ \ar[rr]^(0.35){\Sigma^{W} \hat{\eta}_{n,n} } \ar[rrd]_{ \hat{\eta}_{n,n}' } &   &   \Sigma^{W}
(N_n / L_n) \wedge (N_n / \overline{L}_n) \ar[d]   \\
                                                        &  &   (N_{n}' / L_{n}' ) \wedge ( N_{n}' / \overline{L}_n')
             }
\]
Here $W = V_{\lambda}^{\lambda_n} \oplus V_{\mu_n}^{\mu}$.
\end{lem}

This lemma implies that $\eta_{n,n}$ (and hence $\eta_{m,n}$) is independent of the choice of $\lambda_n, \mu_n$.

Since $j_{n, m+1} = j_{m, m+1} \circ j_{n, m}$ for $m \geq n$,   the following diagram is commutative:
\begin{equation} \label{eq comm eta j}
        \xymatrix{
   S \ar[r]^(0.35){\eta_{m, n} }  \ar[rd]_{{\eta}_{m+1, n}}   &  I_{m} \wedge \bar{I}_{n} \ar[d]^{j_{m, m+1} \wedge \id}
 \\
                                                                                        &    I_{m+1} \wedge \bar{I}_{n}
        }    
\end{equation}

\begin{lem}
For $m \geq n+1$, the following diagram is commutative:
\begin{equation} \label{eq comm eta bar{j}}
     \xymatrix{
        S \ar[r]^(0.35){\eta_{m,n}} \ar[rd]_{\eta_{m,n+1}} &   I_{m} \wedge \bar{I}_{n}  \\
                                                    &   I_{m} \wedge \bar{I}_{n+1} \ar[u]_{\id \wedge \bar{j}_{n, n+1}}                                  }                                                                              
\end{equation}
\end{lem}

\begin{proof}
Let $m \geq n+1$.
We have to show that the following diagram is commutative up to $S^1$-equivariant homotopy:
\begin{equation}   \label{eq eta bar{i}}
    \xymatrix{
   B_{R} / S_{R} \ar[r]^(0.35){ \hat{\eta}_{m,n}} \ar[rd]_{\hat{\eta}_{m,n+1}} 
                   & (N_m / L_m)  \wedge (N_{n} / \overline{L}_{n})    \\
                   & (N_{m} / L_{m}) \wedge (N_{n+1} / \overline{L}_{n+1}) \ar[u]_{\id \wedge \bar{i}_{n, n+1}}.
    }
\end{equation}
By Lemma \ref{lem eta independence}, we can use the following specific manifold isolating blocks $N_{m}, N_{n}, N_{n+1}$
(with corners). Fix a manifold isolating block $N_{m}$ for $\inv (V_{\lambda_m}^{\mu_m} \cap J_m^+)$. Then we have compact
submanifolds $L_{m}, \overline{L}_{m}$ in $\partial N_{m}$ such that 
\[
       \partial N_{m} = L_{m} \cap \overline{L}_{m}, \ \partial L_{m} 
                                 = \partial \overline{L}_{m} 
                                 = L_{m} \cap \overline{L}_{m}.
\]
Moreover $(N_{m}, L_{m})$ is an index pair for $(\inv (V_{\lambda_m}^{\mu_m} \cap J_{m}^+), \varphi_m)$ and $(N_{m}, \overline{L}_{m})$
is an index pair for $(\inv (V_{\lambda_m}^{\mu_m} \cap J_{m}^+), \overline{\varphi}_m)$. 
 Put
\[
   \begin{split}
     & N_{n+1} :=  N_{m} \cap J_{n+1}^+ = N_{m} \cap \bigcap g_{j,+}^{-1} (  (-\infty, n+1+ \theta] ),  \\
     & L_{n+1} := N_{n+1} \cap L_{m}, \\ 
     & \overline{L}_{n+1} :=  (\overline{L}_{m} \cap N_{n+1}) \cup \bigcup_{j=1}^{b_1} (N_{n+1} \cap g_{j, +}^{-1}(n+1+\theta)).
\\
   \end{split}
\]
Then  $N_{n+1}$, $L_{n+1}$ and $\overline{L}_{n+1}$ are manifolds with corners (for generic $\theta$),  and $(N_{n+1}, L_{n+1})$,
$(N_{n+1}, \overline{L}_{n+1})$ are index pairs for $(\inv( V_{\lambda_m}^{\mu_m} \cap J_{n+1}^+), \varphi_m)$, $(\inv (
V_{\lambda_m}^{\mu_m} \cap J_{n+1}^+), \overline{\varphi}_m)$ respectively.  We define $N_{n}, L_{n}, \overline{L}_{n}$ similarly.

The attractor maps $i_{n, m} : N_{n} / L_{n} \rightarrow N_{m} / L_{m}$, $i_{n+1, m} : N_{n+1} / L_{n+1} \rightarrow N_{m}
/ L_{m}$ are the inclusions. The repeller map $\bar{i}_{n,n+1} : N_{n+1} / \overline{L}_{n+1} \rightarrow N_{n} / \overline{L}_{n}$
is the projection:
\[
\begin{split}
    &   N_{n+1} / \overline{L}_{n+1} \rightarrow   \\
    &  \qquad  N_{n+1} \left/  \left\{ \overline{L}_{n+1} \cup \bigcup_{j=1}^{b_1} ( N_{n+1} \cap g_{j,+}^{-1}( [n+ \theta, \infty)
) ) \right\}   \right.
       = N_{n} / \overline{L}_{n}.
 \end{split}
\]
With these index pairs,  for $x \in B_{R}$ we can write
\[
       \hat{\eta}_{m, n}([x]) = 
          \left\{
             \begin{array}{ll}
                 [x] \wedge [x] & \text{if $x \in N_n$,}  \\
                 * & \text{otherwise,}
             \end{array}
          \right.
\]
and
\[
     (\id \wedge \bar{i}_{n, n+1}) \circ \hat{\eta}_{m, n+1} ([x]) =
       \left\{
             \begin{array}{ll}
                 [x] \wedge [x] & \text{if $x \in N_n$,}  \\
                 * & \text{otherwise.}
             \end{array}
          \right.
\]
Thus the diagram (\ref{eq eta bar{i}}) is commutative. 
\end{proof}

The commutativity of the diagrams (\ref{eq comm eta j}), (\ref{eq comm eta bar{j}}) implies that the collection $\{ \eta_{m,
n} \}_{m, n}$ defines an element ${\displaystyle \eta \in \lim_{\infty \leftarrow n} \lim_{m \rightarrow \infty} \morp_{\mathfrak{C}}(
S, I_m \wedge \bar{I}_{n} )}$.

\begin{pro} \label{prop SWF duality}
The morphisms $\epsilon$ and $\eta$ are duality morphisms between $\swf^A(Y)$ and $\swf^R(-Y)$. 
\end{pro}

\begin{proof}
Fix positive numbers $R, \delta$ with $0 < \delta \ll 1 \ll R$. 
Let $\pi : B_{R} / S_{R} \rightarrow B_{\delta} / S_{\delta}$ be the projection
\[
         B_{R} / S_{R}  \rightarrow B_{R} / (B_{R} \setminus \inti B_{\delta}) = B_{\delta} / S_{\delta},
\]
which is a homotopy equivalence.  We have to prove that the  diagrams (\ref{eq eta epsilon 1}) below is commutative for $m
\ll n \ll m'$ and that the diagram (\ref{eq eta epsilon 2}) below is commutative up to $S^1$-equivariant homotopy for $n
\ll m \ll n'$. (See Lemma 3.5 of \cite{LMS}.) 
\begin{equation}  \label{eq eta epsilon 1}
  \xymatrix{
     (B_{R} / S_{R}) \wedge (N_{m} / L_{m})  
     \ar[rr]^(0.42){ \hat{\eta}_{m',n} \wedge \id }  \ar[rrd]_{\gamma \circ (\pi \wedge i_{m, m'})}     &  & 
         (N_{m'} / L_{m'}) \wedge ( N_{n} / \bar{L}_{n}) \wedge (N_{m} / L_{m}) 
        \ar[d]^{\id \wedge \hat{\epsilon}_{m,n}}    \\
       &  &
    (N_{m'} / L_{m'}) \wedge (B_{\delta} / S_{\delta})
 }
\end{equation}
Here $B_R = B(V_{\lambda_{m'}}^{\mu_{m'}}, R)$, $S_R = \partial B(V_{\lambda_{m'}}^{\mu_{m'}}, R)$, $N_m, N_{n}, N_{m'}$
are isolating blocks for $\inv (V_{\lambda_{m'}}^{\mu_{m'}} \cap J_{m}^+), \inv ( V_{\lambda_{m'}}^{\mu_{m'}}  \cap J_n^+
)$, $\inv(V_{\lambda_{m'}}^{\mu_{m'}} \cap J_{m'}^+ )$ and $\gamma$ is the interchanging map $(B_{\delta} / S_{\delta}) \wedge
(N_{m'} / L_{m'}) \rightarrow (N_{m'}/L_{m'}) \wedge (B_{\delta} / S_{\delta})$.  
\begin{equation} \label{eq eta epsilon 2}
  \xymatrix{
    (N_{n'} / \bar{L}_{n'}) \wedge (B_{R} / S_{R}) 
           \ar[rr]^(0.42){\id \wedge \hat{\eta}_{m,n}}  \ar[d]_{ \gamma \circ ( \bar{i}_{n, n'} \wedge \pi ) }   &   &
   (N_{n'} / \bar{L}_{n'}) \wedge (N_m / L_{m}) \wedge (N_n / \bar{L}_n)   \ar[d]^{\hat{\epsilon}_{m,n'} \wedge \id}    \\
    (B_{\delta} / S_{\delta}) \wedge (N_n / \overline{L}_{n})      
        \ar[rr]_{ \sigma \wedge \id  }  &     &   
    (B_{\delta} / S_{\delta}) \wedge (N_{n} / \bar{L}_n)
  }
\end{equation}
Here $B_R = B(V_{\lambda_{n'}}^{\mu_{n'}}, R)$, $S_R = \partial  B(V_{\lambda_{n'}}^{\mu_{n'}}, R)$, $N_{m}, N_{n}, N_{n'}$
are isolating blocks for $\inv(V_{\lambda_{n'}}^{\mu_{n'}} \cap J_{m}^+ )$, $\inv(  V_{\lambda_{n'}}^{\mu_{n'}} \cap J_{n}^+
)$, $\inv(V_{\lambda_{n'}}^{\mu_{n'}} \cap J_{n'}^+ )$,  $\gamma$ is the interchanging map $(N_{n} / L_{n}) \wedge ( B_{\delta}
/ S_{\delta}) \rightarrow (B_{\delta} / S_{\delta}) \wedge (N_{n} / L_{n})$ and $\sigma : B_{\delta} / S_{\delta} \rightarrow
B_{\delta} / S_{\delta}$ is defined by $\sigma(v) = -v$. 

\vspace{2mm}

First we consider (\ref{eq eta epsilon 1}).  Let $m \ll n \ll m'$.  Take a manifold isolating block $N_{m'}$ for $\inv( V_{\lambda_{m'}}^{\mu_{m'}}
\cap J_{m'}^+ )$. As in the proof of Lemma \ref{lem commu epsilon bar{j}}, from $N_{m'}$ and the functions $g_{j, +}$, we
get index pairs 
\[
    (N_{n}, L_{n}), \ (N_{n}, \overline{L}_n),  \ (N_{m}, L_{m}), \ (N_{m}, \overline{L}_{m})
\]
 for 
 \[
 \begin{split}
   &  (\inv (V_{\lambda_{m'}}^{\mu_{m'}} \cap J_{n}^+ ), \varphi_{m'}),  \ 
     (\inv (V_{\lambda_{m'}}^{\mu_{m'}} \cap J_{n}^+ ), \overline{\varphi}_{m'}),  \\
  &   (\inv ( V_{\lambda_{m'}}^{\mu_{m'}} \cap J_{m}^+ ), \varphi_{m'}),  \
     (\inv ( V_{\lambda_{m'}}^{\mu_{m'}} \cap J_{m}^+ ), \overline{\varphi}_{m'}).
   \end{split}
\]
The attractor map 
\[
     i_{m, n} :N_m / L_m \rightarrow N_n / L_n,   \
     i_{n, m'} : N_n /L_n \rightarrow N_{m'} / L_{m'}
\]
are the injections, and the repeller maps 
\[ 
     \bar{i}_{n, m'} : N_{m'} / \overline{L}_{m'} \rightarrow N_{n} / \overline{L}_{n},  \
     \bar{i}_{m, n} : N_n / \overline{L}_n \rightarrow N_m / \overline{L}_m
\]
are the projections.    

For $x \in N_{m}$ and $y \in B_{R} (= B(V_{\lambda_{m'}}^{\mu_{m'}}, R))$, we can write
\[
\begin{split}
   &  (\id \wedge \hat{\epsilon}_{m, n}) \circ ( \hat{\eta}_{m', n} \wedge \id) ( [y] \wedge [x]) =  \\
  & \qquad    \left\{
          \begin{array}{ll}
                  [y] \wedge [b_n(y) - a_n( x ) ] & 
                         \text{if}  \left\{
                         \begin{array}{l}
                             y \in N_{n}, \\
                            \| b_{n}(y) - a_{n}(x) \| < \delta,
                         \end{array}    
                           \right.  \\
                  * & \text{otherwise.}
          \end{array}
        \right.
    \end{split}
\]
Note that if $\| b_{n}(y) - a_{n}(x) \| < \delta$ for some $x \in N_{m}$ we have $y \in \nu_{5\delta}(N_{m})$.  Fix an $S^1$-equivariant
homotopy equivalence
\[
        r : \nu_{5\delta}(N_{m}) \rightarrow N_{m}
\]
which is close to the identity such that
\[
      r(\nu_{5\delta}(L_{n}) \cap \nu_{5\delta}(N_m) ) \subset L_{m}, \      r( \nu_{5\delta}(L_m)   ) \subset L_m.
\]
Then $(\id \wedge \hat{\epsilon}_{m, n}) \circ (\hat{\eta}_{m', n} \wedge \id)$ is homotopic to a map 
\[
    f : (B_R/S_R) \wedge (N_m / L_m) \rightarrow (N_{m'} / L_{m'}) \wedge (B_{\delta} / S_{\delta})
\]
defined by
\[
\begin{split}
   &   f([y] \wedge [x])  =  \\
    & \qquad  \left\{
          \begin{array}{ll}
                  [r(y)] \wedge [b_n(y) - a_n( x ) ] & 
                         \text{if}  \left\{
                         \begin{array}{l}
                             x \in N_{m},  y \in N_{n} \cap \nu_{5\delta}(N_m), \\
                            \| b_{n}(y) - a_{n}(x) \| < \delta,
                         \end{array}    
                           \right.  \\
                  * & \text{otherwise.}
          \end{array}
        \right.
    \end{split}
\]
Define 
\[
      H : (B_R/S_R) \wedge (N_m / L_m) \times [0, 1] \rightarrow (N_{m'} / L_{m'}) \wedge (B_{\delta} / S_{\delta})
\]
by
\[
    \begin{split}
       & H([y] \wedge [x], s)  = \\
       &  \left\{
          \begin{array}{ll}
                  [r( (1-s)y + sx )] \wedge [b_n(y) - a_n( x ) ] & 
                         \text{if}  \left\{
                         \begin{array}{l}
                             x \in N_{m},  \\
                             y \in N_{n} \cap \nu_{5\delta}(N_m), \\
                            \| b_{n}(y) - a_{n}(x) \| < \delta,
                         \end{array}    
                           \right.  \\
                  * & \text{otherwise.}
          \end{array}
        \right.
   \end{split}
\]
We can easily see that $H$ is well-defined. We will show that $H$ is continuous. It is sufficient to show that if we have
a sequence $(x_j, y_j, s_j)$ in $N_{m} \times N_n \times [0, 1]$ with $y_j \rightarrow y \in \partial N_n = L_n \cup \overline{L}_n$
we have $H([y_j] \wedge [x_j], s_j) \rightarrow *$.  If $y \in \overline{L}_{n}$ we have $\| b_n(y_j) - a_{n}(x_j) \| \geq
\delta$ for large $j$. Hence $H([y_j] \wedge [x_j], s_j) \rightarrow *$.  Consider the case $y \in L_n$.  Assume that ${\displaystyle
\lim_{j \rightarrow \infty} H([y_j] \wedge [x_j], s_j)} \not= *$.  After passing to a subsequence we may suppose that $H([y_j]
\wedge [x_j], s_j) \not= *$ for all $j$.   Then $\| y_j - x_j \| < 5\delta$ for all $j$. For large $j$ we have $(1-s_j) y_j
+ s_j x_j \in \nu_{5\delta}(L_n) \cap \nu_{5\delta}(N_m) $.  Hence $r((1-s_j) y_j + s_j x_j) \in L_{m} \subset L_{m'}$, which
implies $H([y_j] \wedge [x_j], s_j) =*$.  This is a contradiction.  Therefore $H$ is continuous.

We have $H(\cdot, 0) = f$ and 
\[
        H([y] \wedge [x], 1)  
        = \left\{
          \begin{array}{ll}
                  [r(x)] \wedge [b_n(y) - a_n( x ) ] & 
                         \text{if}  \left\{
                         \begin{array}{l}
                             x \in N_{m},  y \in N_{n}, \\
                            \| b_{n}(y) - a_{n}(x) \| < \delta,
                         \end{array}    
                           \right.  \\
                  * & \text{otherwise.}
          \end{array}
        \right.
\]
Fix a positive number $\delta' > 0$ with $0 \ll \delta' \ll \delta$. Take an $S^1$-equivariant continuous map $a_{n}' : N_n
\rightarrow N_n$  such that
\[
         \| a_{n}'(x) - a_{n}(x) \| < 2\delta',  \ a_{n}' (N_n) \subset N_n \setminus \nu_{\delta'} ( \partial N_{n} ). 
\]
Through the homotopy equivalence
\[
        B_{\delta} / S_{\delta} = V_{\lambda_{m'}}^{\mu_{m'}} / ( V_{\lambda_{m'}}^{\mu_{m'}} -  \inti B_{\delta} )
        \rightarrow
        V_{\lambda_{m'}}^{\mu_{m'}} /  (  V_{\lambda_{m'}}^{\mu_{m'}} - \inti B_{\delta'}  ) = B_{\delta'} / S_{\delta'},
\]
$H(\cdot, 1)$ is homotopic to a map 
\[    
     f' : (B_R / S_R) \wedge (N_m / L_m) \rightarrow (N_{m'} / L_{m'}) \wedge (B_{\delta'} / S_{\delta'})
\]
defined by
\[
     f' ([y] \wedge [x]) 
     = \left\{
          \begin{array}{ll}
                  [r(x)] \wedge [b_n(y) - a'_n( x ) ] & 
                         \text{if}  \left\{
                         \begin{array}{l}
                             x \in N_{m},  y \in N_{n}, \\
                            \| b_{n}(y) - a'_{n}(x) \| < \delta',
                         \end{array}    
                           \right.  \\
                  * & \text{otherwise.}
          \end{array}
        \right.
\]
There is  a homotopy $h : N_n \times [0, 1] \rightarrow N_n$ from $b_n$ to the identity such that
\[
      h(\overline{L}_n, s) \subset  \overline{L}_{n}, \ \| h(y, s) - y \| < 2 \delta
\]
for all $y \in N_n$ and $s \in [0, 1]$. Then $h$ naturally induces a homotopy
\[
      H' : (B_R / S_R) \wedge (N_m / L_m) \rightarrow (N_{m'} / L_{m'}) \wedge (B_{\delta'} / S_{\delta'})
\]
defined by
\[
      H'([y] \wedge [x], s) = 
        \left\{
          \begin{array}{ll}
                  [r(x)] \wedge [ h(y, s) - a'_n( x ) ] & 
                         \text{if}  \left\{
                         \begin{array}{l}
                             x \in N_{m},  y \in N_{n}, \\
                            \| b_{n}(y) - a'_{n}(x) \| < \delta',
                         \end{array}    
                           \right.  \\
                  * & \text{otherwise.}
          \end{array}
        \right.
\]
It is straightforward to see that $H'$ is well-defined. To show that $H'$ is continuous, it is sufficient to prove that if we have a
sequence $(x_j, y_j, s_j)$  in $N_{m} \times N_{n} \times [0, 1]$ with $y_j \rightarrow y \in  \partial N_n = L_n \cup \overline{L}_n$
then $H'([y_j] \wedge [x_j], s_j) \rightarrow *$.  Suppose that $y \in \overline{L}_n$. Then for large $j$ we have $\| h(y_j,
s) - a_{n}'(x) \| \geq \delta'$. Thus $H([y_j] \wedge [x_j], s_j) \rightarrow *$. Suppose that $y \in L_n$.  If $\lim_{j
\rightarrow \infty} H'([y_j] \wedge [x_j], s_j) \not= *$, after passing to a subsequence, we may assume that $H([y_j] \wedge
 [x_j], s_j) \not= *$ for all $j$, which implies that $\| y_j - x_j \| < 5\delta$.  So we have $x_j \in \nu_{5\delta}(L_n)
\cap \nu_{5\delta}(N_m)$ for large $j$.  Hence $r(x_j) \in N_m$, which means $H'([y_j] \wedge  [x_j], s_j) = *$ . This is
contradiction.  Therefore $H'$ is continuous.

We can  see that $H'$ is a homotopy  from $f'$ to a map $f'' : (B_R / S_R) \wedge (N_m / L_m) \rightarrow (N_{m'} / L_{m'})
\wedge (B_{\delta'} / S_{\delta'})$ defined by 
\[
       f''([y] \wedge [x]) = 
         \left\{
          \begin{array}{ll}
                  [r(x)] \wedge [y - a'_n( x ) ] & 
                         \text{if}  \left\{
                         \begin{array}{l}
                             x \in N_{m},   y \in B_R, \\
                            \| y - a'_{n}(x) \| < \delta',
                         \end{array}    
                           \right.  \\
                  * & \text{otherwise.}
          \end{array}
        \right.
\]
Note that for $y \in B_R \setminus N_{n}$ we have $f''([y] \wedge [x]) = *$ since $\| y - a_{n}'(x) \| \geq \delta'$.  
Define
\[
       H'' : (B_{R} / S_{R}) \wedge (N_m / L_{m}) \times [0, 1] \rightarrow  (N_{m'} / L_{m'}) \wedge (B_{\delta'} / S_{\delta'})
\]
by
\[
 \begin{split}
   &   H''([y] \wedge [x], s) = \\
   &  \qquad     \left\{
            \begin{array}{ll}
                [r(x)] \wedge [  y - (1-s) a_n'(x)  ] & \text{if}
                    \left\{
                       \begin{array}{l}
                        x \in N_{m},  y \in B_R, \\
                        \| y - (1-s) a_{n}'(x) \| < \delta', 
                       \end{array}
                    \right.    \\
                * & \text{otherwise.}
            \end{array}   
         \right.
     \end{split}
\]
It is straightforward to see that $H''$ is well-defined and continuous.  We have
\[
     H''([y] \wedge [x], 1) =    \left\{
            \begin{array}{ll}
                [r(x)] \wedge [  y  ] & \text{if}
                    \left\{
                       \begin{array}{l}
                        x \in N_{m},  y \in B_R, \\
                        \| y  \| < \delta', 
                       \end{array}
                    \right.    \\
                * & \text{otherwise.}
            \end{array}   
         \right. 
\]
We can easily  show that $H''(\cdot, 1)$ is homotopic to $\gamma \circ (\pi \wedge i_{m, m'})$.

We have proved that $(\id \wedge \hat{\epsilon}_{m, n}) \circ ( \hat{\eta}_{m', n} \wedge \id)$ is $S^1$-equivariantly homotopic
to $\gamma \circ (\pi \wedge i_{m, m'})$, which implies that the diagram (\ref{eq eta epsilon 1}) is commutative up to $S^1$-equivariant
homotopy.

\vspace{3mm}

\vspace{3mm}

Let us consider  (\ref{eq eta epsilon 2}). 
We have to prove that for $n \ll m \ll n'$  the composition
\[
 \begin{split}
    (N_{n'} / \overline{L}_{n'}) \wedge ( B_{R} / S_{R}) & \xrightarrow{ \id \wedge \hat{\eta}_{m,n} }   
   (N_{n'} / \overline{L}_{n'}) \wedge (N_{m} / L_{m}) \wedge (N_{n} / \overline{L}_{n})  \\
   & \xrightarrow{\hat{\epsilon}_{m,n'}
\wedge \id}
       (B_{\delta} / S_{\delta}) \wedge (N_{n} / \overline{L}_{n})
   \end{split}
\]
is $S^1$-equivariantly homotopic to $(\sigma \wedge \id) \circ  \gamma \circ ( \bar{i}_{n', n} \wedge \pi)$.

For $x \in B_R = B(V_{\lambda_{n'}}^{\mu_{n'}}, R)$, $y \in N_{n'}$ we have
\[
   \begin{split}
       & (\hat{\epsilon}_{m,n'} \wedge \id) \circ (\id \wedge \hat{\eta}_{m,n})( [y] \wedge [x]) =  \\
       & \quad \left\{
           \begin{array}{ll}
               [b_{n'}(y) - a_{n'}(x)] \wedge [x] & \text{if}
                    \left\{
                       \begin{array}{l}
                           x \in N_{n}, \\
                           \| b_{n'}(y) - a_{n'}( x )  \| < \delta,
                       \end{array}
                    \right.  \\
             * & \text{otherwise.}
           \end{array}
          \right.
   \end{split}
\]
Take a homotopy equivalence $\bar{r} : \nu_{5\delta}(N_{n}) \rightarrow N_{n}$ which is close to the indentiy such that
\begin{equation} \label{eq bar{r}}
        \bar{r} (  \nu_{5\delta}(\overline{L}_{n'} )\cap \nu_{5\delta} (N_{n}) ) \subset \overline{L}_{n}, \
         \bar{r}( \nu_{5\delta} ( \overline{L}_{n})) \subset \overline{L}_{n}.
\end{equation}
Note that  if $\| b_{n'}(y) - a_{n'}(x) \| < \delta$ for some $x \in N_n$ we have $y \in \nu_{5\delta}(N_n)$.

It is straightforward to see that $(\hat{\epsilon}_{m,n'} \wedge \id) \circ (\id \wedge \hat{\eta}_{m,n})$ is homotopic to a map 
\[
    f : ( N_{n'} / \overline{L}_{n'} ) \wedge (B_R / S_R) \rightarrow (B_{\delta} / S_{\delta}) \wedge (N_n / L_n)
\]
defined by
\[
\begin{split}
     &    f([y] \wedge [x]) = \\
      & \qquad        \left\{
            \begin{array}{ll}
                 [b_{n'}( y ) - a_{n'}(x)] \wedge [ \bar{r}(x)] & \text{if} 
                  \left\{
                    \begin{array}{ll}
                          x \in N_n,  \\
                          y \in N_{n'} \cap \nu_{5\delta}(N_n), \\
                          \| b_{n'}( y ) - a_{n'}(x) \| < \delta, \\
                    \end{array}
                 \right.  \\
             * & \text{otherwise.}
            \end{array}
         \right.
    \end{split}
\]

Define a homotopy $H : ( N_{n'} / \overline{L}_{n'} ) \wedge (B_{R} / S_{R}) \times [0, 1] \rightarrow (B_{\delta} / S_{\delta})
\wedge (N_{n} / \overline{L}_{n})$ by

\[
   \begin{split}
   &   H ([y] \wedge [x], s) = \\
   & \quad    \left\{
            \begin{array}{ll}
                 [b_{n'}( y ) - a_{n'}(x)] \wedge [ \bar{r}( (1-s)x + sy )] & \text{if} 
                  \left\{
                    \begin{array}{ll}
                          x \in N_n,  \\
                          y \in N_{n'} \cap \nu_{5\delta}(N_n), \\
                          \| b_{n'}(y) - a_{n'}(x) \| < \delta \\
                    \end{array}
                 \right.  \\
             * & \text{otherwise.}
            \end{array}
         \right.
     \end{split}
\]
Then $H$ is a well-defined and continuous homotopy from $( \id \wedge \hat{\epsilon}_{m, n'} ) \circ (\id \wedge \hat{\eta}_{m,n})$
to $H(\cdot, 1)$.  We have
\[
\begin{split}
        &    H([y] \wedge [x], 1) = \\
       &       \left\{
            \begin{array}{ll}
                 [b_{n'}(y) - a_{n'}(x)] \wedge [ \bar{r} (y) ] & \text{if} 
                  \left\{
                    \begin{array}{ll}
                          x \in N_n,   \\
                          y \in N_{n'} \cap \nu_{5\delta}(N_n), \\
                          \| b_{n'}(y) - a_{n'}(x) \| < \delta, \\
                    \end{array}
                 \right.  \\
             * & \text{otherwise.}
            \end{array}
         \right.
     \end{split}
\]

Fix a positive number $\delta'$ with  $0 < \delta' \ll \delta$. Take a continuous map $b'_{n'} : N_{n} \rightarrow N_{n}$
such that
\[
      b_{n'}'(N_{n}) \subset N_{n} \setminus \nu_{\delta'}( \partial N_{n}), \
      \| b_{n'}'(y) - b_{n'}(y) \| < 2\delta' \quad (\text{for $x \in N_{n}$}).
\]
Then through the homotopy equivariance $B_{\delta} / S_{\delta} \rightarrow B_{\delta'} / S_{\delta'}$,   $f$ is homotopic
to a map 
\[
    f' : (N_{n'} / \overline{L}_{n'}) \wedge (B_{R} / S_{R}) \rightarrow (B_{\delta'} / S_{\delta'}) \wedge (N_{n} / \overline{L}_{n})
\]
defined by
\[
         f'([y] \wedge [x]) =
              \left\{
            \begin{array}{ll}
                 [b_{n'}'( y) - a_{n'}(x)] \wedge [ \bar{r} (y) ] & \text{if} 
                  \left\{
                    \begin{array}{ll}
                          x \in N_n,  \\
                          y \in N_{n'} \cap \nu_{5\delta}(N_n), \\
                          \| b_{n'}'( y ) - a_{n'}(x) \| < \delta', \\
                    \end{array}
                 \right.  \\
             * & \text{otherwise.}
            \end{array}
         \right.
\]
There is a homotopy $h : N_{n'} \times [0, 1] \rightarrow N_{n'}$ from $a_{n'}$ to the identity such that
\[
        h( L_{n'}) \subset L_{n'}, \ \| h(y, s) - y \| < 2\delta. 
\]
We can see that $h$ induces a homotopy $H'$ from $f'$ to a map $f''$ defined by
\[
      f''( [y] \wedge [x] ) =
          \left\{
            \begin{array}{ll}
                 [ b_{n'}'( y ) - x] \wedge [ \bar{r} (y) ] & \text{if} 
                  \left\{
                    \begin{array}{ll}
                          x \in B_{R},   \\
                          y \in N_{n'} \cap \nu_{5\delta}(N_{n}), \\
                          \| b_{n'}'( y ) - x \| < \delta', \\
                    \end{array}
                 \right.  \\
             * & \text{otherwise.}
            \end{array}
         \right.
\]
Note that for $x \in B_R \setminus \inti N_{n}$ we have $f'([y] \wedge [x]) = *$ since $\| b_{n'}'(y) - x \| \geq \delta'$.
 Define a homotopy $H'' : (N_{n'} / \overline{L}_{n'}) \wedge (B_R / S_{R}) \times[0, 1] \rightarrow (B_{\delta'} / S_{\delta'})
\wedge (N_{n} / \overline{L}_{n})$ by
\[
\begin{split}
     &  H''([y] \wedge [x], s) := \\
    &    \qquad  \left\{
            \begin{array}{ll}
                 [ (1-s)b_{n'}'( y ) - x] \wedge [ \bar{r} (y) ] & \text{if} 
                  \left\{
                    \begin{array}{ll}
                          x \in B_{R},   \\
                          y \in N_{n'} \cap \nu_{5\delta}(N_n), \\
                          \| (1-s) b_{n'}'( y ) -  x \| < \delta', \\
                    \end{array}
                 \right.  \\
             * & \text{otherwise.}
            \end{array}
         \right.
     \end{split}
\]
Then $H''$ is well-defined and continuous, and we have
\[
      H''([y] \wedge [x], 1) =
          \left\{
            \begin{array}{ll}
                 [-x] \wedge [ \bar{r} (y) ] & \text{if} 
                  \left\{
                    \begin{array}{ll}
                          x \in B_{R},   \\
                          y \in N_{n'} \cap \nu_{5\delta}(N_n), \\
                          \| x \| < \delta', \\
                    \end{array}
                 \right.  \\
             * & \text{otherwise.}
            \end{array}
         \right.
\]
It is straightforward to see that $H''(\cdot, 1)$ is homotopic to $(\sigma \wedge \id) \circ \gamma \circ (\bar{i}_{m, n} \wedge \pi)$.
Thus the diagram (\ref{eq eta epsilon 2}) is commutative up $S^1$-equivariant homotopy.

\end{proof}


\section{Relative Bauer-Furuta invariants for 4-manifolds}
\label{sec 4mfd}
  \subsection{Setup}

Let $X$ be a compact, connected, oriented, Riemannian 4--manifold with nonempty boundary $\partial X :=Y$ not necessarily connected. 
Equip $X$ with  a spin$^c$ structure
$\hat{\mathfrak{s}}$ which induces a spin$^c$ structure $\mathfrak{s}$ on $Y$. 
Denote by ${S}_X = S^+ \oplus S^-$ the spinor bundle of $X$ and denote by $\hat{\rho} $ the Clifford multiplication. Choose a metric $\hat{g} $ on $X$ so that a neighborhood of the boundary is isometric to the cylinder $[-3,0]
\times Y$ with the product metric and $\partial X $ identified with $\{ 0 \} \times Y$. To make some distinction, we will often decorate notations associated
to $X$
with hat.   
For instance, let $g$ be the Riemannian metric on $Y$ restricted from $\hat{g}$ on $X$. 
Let $S_{Y}$ be the associated spinor bundle on $Y$  and $\rho \colon TY\rightarrow \text{End}(S_{Y})$ be
the Clifford multiplication.

We write $Y = \coprod Y_j $ as a union of connected component. From now on, we will treat $X$ as a spin$^c$ cobordism, i.e. we label each connected component of $Y $ as either incoming or outgoing satisfying $Y = -Y_{\text{in}} \sqcup Y_{\text{out}} $. We sometimes write this cobordism as $X \colon Y_{\text{in}} \rightarrow Y_{\text{out}}$.
Denote by $ \iota \colon Y \hookrightarrow X $ the inclusion map.
We also choose the following auxiliary data when defining our invariants
\begin{itemize}
\item 
A basepoint $\hat{o} \in X $.
\item A set of loops $\{ \alpha_1, \ldots , \alpha_{b_{1,\alpha}} \} $ in $X$ representing a basis of the cokernel of the induced map $ \iota_*
\colon H_1 (Y ; \mathbb{R}) \rightarrow H_1 (X ; \mathbb{R})$.
\item A set of loops $\{ \beta_1, \ldots , \beta_{b_{\text{in}}} \} $ in $Y_{\text{in}}$ representing a basis of a subspace complementary to kernel of the induced map $ \iota_*
\colon H_1 (Y_{\text{in}} ; \mathbb{R}) \rightarrow H_1 (X ; \mathbb{R})$.
\item A set of loops $\{ \beta_{b_{\text{in}}+1}, \ldots , \beta_{b_{1,\beta}} \} $ in $Y_{\text{out}}$ such that $\{ \beta_1, \ldots , \beta_{b_{1,\beta}} \} $ represents a basis of a subspace
complementary to the kernel of the induced map $ \iota_*
\colon H_1 (Y_{} ; \mathbb{R}) \rightarrow H_1 (X ; \mathbb{R})$.
\item A \emph {based path data} $\left[\vec{\eta}\right]$, whose definition is given below. 
 \end{itemize}

\begin{defi}\label{def based path data}
A based path data is an equivalence class of paths $(\eta_{1},\eta_{2},\ldots,\eta_{b_{0}(Y)})$, where each $\eta_{j}$ is a path from
$\hat{o}$ to a point in $Y_{j}$. We say that $(\eta_{1},\ldots,\eta_{b_{0}(Y)})$ and $(\eta'_{1},\ldots,\eta'_{b_{0}(Y)})$ are
equivalent if the composed path $\eta'_{j} \ast (-\eta_{j})$ represents the zero class in $H_{1}(X,Y;\mathbb{R})$ for  each $j = 1, \ldots, b_0 (Y)$.
\end{defi}

\begin{rmk} 
(i) The set of loops $\{ \alpha_1, \ldots , \alpha_{b_{1,\alpha}} \} $ corresponds to a dual basis of kernel of $ \iota^*
\colon H^1 (X ; \mathbb{R}) \rightarrow H^1 (Y ; \mathbb{R})$.  \\
(ii) The set of loops $\{ \beta_1, \ldots , \beta_{b_{1,\beta}} \} $ corresponds to a dual basis of image of $ \iota^*
\colon H^1 (X ; \mathbb{R}) \rightarrow H^1 (Y ; \mathbb{R})$. \\
(iii) It follows that $\, b_{1,\alpha} = \dim{\ker{\iota^*}}$, $b_{1,\beta} = \dim{\im{\iota^*}}$,   and $b_{1,\alpha} +b_{1,\beta} = b_1
(X) $.
\end{rmk}

As usual, we will set up the Seiberg--Witten equations on a particular slice of the configuration space.  For the manifold with boundary $X$, we will consider the double Coulomb condition introduced by the first author \cite{Khandhawit1} rather than the classical Coulomb--Neumann condition.
Let us briefly recall the definition.

\begin{defi} \label{def doublecoul} For a 1-form $\hat{a}$ on $X$, we have a decomposition $\hat{a}|_{Y}=\mathbf{t}\hat{a}+\mathbf{n}\hat{a}$ on the boundary, where
$\mathbf{t}\hat{a}$ and $\mathbf{n}\hat{a}$ are the tangential part and the normal part respectively. When $Y = \coprod Y_i$
has several
connected components, we denote by $\mathbf{t}_{i}\hat{a}$ and $\mathbf{n}_{i}\hat{a} $ the corresponding parts of $\hat{a}|_{Y_{i}}$.
We say that a 1-form $\hat{a}$ satisfies the double Coulomb condition if:
\begin{enumerate}
\item $\hat{a}$ is coclosed, i.e. $d^{*}\hat{a}=0$;
\item Its restriction to the boundary is coclosed, i.e. $d^{*}(\mathbf{t}\hat{a})=0$;
\item \label{item Coulomb3rd} For each $j$, we have $\int _{Y_{j}} \mathbf{t_{j}}(*\hat{a}) =0$.
\end{enumerate}
We denote by $\Omega^{1}_{CC}(X)$ the space of 1-forms satisfying the double  Coulomb condition.
\end{defi}

As a consequence of \cite[Proposition~2.2]{Khandhawit1}, we can identify $H^1(X ; \mathbb{R}) $ with the space of harmonic 1-forms satisfying double Coulomb condition
\begin{align*}  H^1(X ; \mathbb{R}) \cong \mathcal{H}^1_{CC}(X):= \{ \hat{a} \in \Omega^{1}_{CC}(X) \mid d\hat{a} = 0 \}.
\end{align*}
Since $X$ is connected, we observe that the cohomology long exact sequence of the pair $(X,Y)$ gives rise to a short exact sequence
\begin{align*} 
0 \rightarrow \mathbb{R}^{b_0 (Y) - 1} \rightarrow H^1 (X,Y ; \mathbb{R}) \rightarrow \ker{\iota^*} \rightarrow 0.
\end{align*} 
By the classical Hodge Theorem, each element of the relative cohomology group $ H^1 (X,Y ; \mathbb{R}) $ is represented by a harmonic 1-form with Dirichlet boundary condition. Since condition~(\ref{item Coulomb3rd}) from Definition~\ref{def doublecoul} is of codimension $b_0(Y) - 1 $, we can conclude that the space of harmonic 1-forms
satisfying both Dirichlet boundary condition and condition~(\ref{item Coulomb3rd}) from Definition~\ref{def doublecoul} is isomorphic to $\ker{\iota^*} $. Notice that such 1-forms trivially satisfy other double Coulomb conditions. Hence, we make an identification
\begin{align}    \label{eq H1DC}
\ker{\iota^*} \cong \mathcal{H}^1_{DC}(X) :=  \{ \hat{a} \in \Omega^{1}_{CC}(X) \mid d\hat{a} = 0, \ \mathbf{t}\hat{a}=0 \}.
 \end{align}

 The double Coulomb slice
$Coul^{CC}(X)$ is defined as
\begin{align*}
Coul^{CC}(X):=L^{2}_{k+1/2}(i\Omega^{1}_{CC}(X)\oplus \Gamma(S_{}^{+})),
\end{align*}
where $k$ is an integer greater than 4 fixed throughout the paper. Next, we introduce projections from $Coul^{CC}(X) $ related to the chosen loops  $\{ \alpha_1 , \ldots , \alpha_{b_{1,\alpha}} \} $ and $\{ \beta_1, \ldots , \beta_{b_{1,\beta}} \} $. We define a (nonorthogonal) projection
\begin{align} \label{eq defp_alpha}
    \hat{p}_{\alpha} \colon Coul^{CC}(X)\rightarrow  \mathcal{H}^1_{DC}(X)  
\end{align}
by sending $(\hat{a},\hat{\phi})$ to  the unique
element in $  \mathcal{H}^1_{DC}(X)$ satisfying
$$
\int_{\alpha_{j}}\hat{a}=i\int_{\alpha_{j}}\hat{p}_{\alpha}(\hat{a}, \hat{\phi})\text{ for every }j=1,2,\ldots,b_{1,\alpha}.
$$

 On the other hand, we define a map
\begin{align*}
\hat{p}_{\beta} \colon Coul^{CC}(X) &\rightarrow \mathbb{R}^{b_{1,\beta}}\\
(\hat{a},\hat{\phi}) &\mapsto (-i\int_{\beta_{1}}\mathbf{t}\hat{a}, \, \ldots  \, ,-i\int_{\beta_{b_{1,\beta}}}\mathbf{t}\hat{a}). \nonumber
\end{align*}
Note that $\hat{p}_{\alpha} $ and $\hat{p}_{\beta} $ together  keep track of the $H^1(X; \mathbb{R}) $-component of $(\hat{a},\hat{\phi})$.   
We have a decomposition
\[
        \hat{p}_{\beta} = \hat{p}_{\beta, \text{in}} \oplus \hat{p}_{\beta, \text{out} },
\]
where
\[
  \begin{split}
    & \hat{p}_{\beta, \text{in}}(\hat{a}, \hat{\phi}) = 
         ( - i \int_{\beta_1} \mathbf{t} \hat{a}, \dots, - i \int_{\beta_{b_{\text{in}}}} \mathbf{t} \hat{a}  ),  \\
     & \hat{p}_{\beta, \text{out}}(\hat{a}, \hat{\phi}) = 
        ( - i \int_{\beta_{b_{\text{in}} + 1}} \mathbf{t} \hat{a}, \dots, - i \int_{\beta_{b_{1, \beta}}} \mathbf{t} \hat{a} ).
  \end{split}
\]

We now proceed to describe the group of gauge transformations.
Denote by $\mathcal{G}_{X}$  the $L^{2}_{k+3/2}$-completion of $\operatorname{Map}(X,S^{1})$.
The action of an element  $\hat{u} \in \operatorname{Map}(X,S^{1})$  is given by
\begin{align*}\hat{u} \cdot (\hat{a}, \hat{\phi}) = (\hat{a} - \hat{u}^{-1} d \hat{u}, \hat{u} \hat{\phi}) .
\end{align*}
The proof of the following lemma is a slight adaptation of \cite[Proposition~2.2]{Khandhawit1} and we omit it.
\begin{lem}\label{harmonic gauge tranformation}
Inside each connected component of $\mathcal{G}_{X}$, there is a unique element $\hat{u}:X\rightarrow S^{1}$ satisfying
$$
\hat{u}(\hat{o})=1,\ u^{-1}du\in i\Omega^{1}_{CC}(X).
$$
These elements form a subgroup, denoted by $\mathcal{G}^{h,\hat{o}}_{X}$,
 of harmonic gauge transformation with double Coulomb condition. \end{lem}

Consequently, there is a natural isomorphism
\begin{equation*} 
\mathcal{G}^{h,\hat{o}}_{X}\cong \pi_{0}(\mathcal{G}_{X})\cong H^{1}(X;\mathbb{Z}).
\end{equation*}
We also denote by $\mathcal{G}^{h,\hat{o}}_{X,Y}$ the subgroup of $\mathcal{G}^{h,\hat{o}}_{X}$
 that corresponds to the subgroup $\ker(H^{1}(X;\mathbb{Z})\rightarrow H^{1}(Y;\mathbb{Z}))$ of $H^{1}(X;\mathbb{Z})$. Observe that each element in $\mathcal{G}^{h,\hat{o}}_{X,Y}$
restricts to a constant function on each component of $Y$.

Now we define the relative Picard torus
\begin{equation*}  
\begin{split}
\operatorname{Pic}^{0}(X,Y):&=\mathcal{H}^1_{DC}(X)/\mathcal{G}^{h,\hat{o}}_{X,Y}\\&\cong \ker(H^{1}(X;\mathbb{R})\rightarrow
H^{1}(Y;\mathbb{R}))/\ker(H^{1}(X;\mathbb{Z})\rightarrow H^{1}(Y;\mathbb{Z})).
\end{split}
\end{equation*}
This is a torus of dimension $b_{1,\alpha} $.
 The double Coulomb slice $Coul^{CC}(X)$ is preserved by $\mathcal{G}^{h,\hat{o}}_{X}$ and thus $\mathcal{G}^{h,\hat{o}}_{X,Y}$.

Our main object of interest will be the quotient space $Coul^{CC}(X)/\mathcal{G}^{h,\hat{o}}_{X,Y}$ regarded as a Hilbert bundle
over $\operatorname{Pic}^{0}(X,Y)$ with bundle structure induced by the projection $\hat{p}_{\alpha}$. The bundle will be denoted by 
\[
            \mathcal{W}_X := Coul^{CC}(X)/\mathcal{G}^{h,\hat{o}}_{X,Y}.
\]

\begin{rmk}
A different Hilbert bundle structure of $\mathcal{W}_X $ can be induced by the  orthogonal projection
$$\hat{p}_{\perp}:Coul^{CC}(X)\rightarrow \mathcal{H}^1_{DC}(X).$$
However, we prefer $\hat{p}_{\alpha}$ because $\hat{p}_{\alpha}$ behaves better than $\hat{p}_{\perp}$ and simplifies our argument in the proof of gluing theorem for relative Bauer-Furuta invariants.
\end{rmk}

\begin{defi}
For a pair $(\hat{a},\hat{\phi})\in Coul^{CC}(X)$, we denote by $[\hat{a},\hat{\phi}]$ the corresponding element in the Hilbert
bundle $\mathcal{W}_X $. We write $\|\cdot\|_{F}$ for the fiber-direction norm on $\mathcal{W}_X $.
\end{defi}

Note that the norm $\|\cdot \|_{F}$ is not directly given by the restriction of the $L^{2}_{k+1/2}$-norm on $Coul^{CC}(X)$ because the 
latter is not invariant under $\mathcal{G}^{h,\hat{o}}_{X,Y}$. However, we can construct $\|\cdot \|_{F}$ as follows: We cover $\operatorname{Pic}^{0}(X,Y)$ by finitely many small balls $\{B_{i}\}$. Each $B_{i}$ can be lifted as a subset of $\mathcal{H}^{1}_{DC}$. With such a lift chosen, we can identify the total space of $\mathcal{W}_{X}|_{B_{i}}$ as a subset of $Coul^{CC}(X)$. Then we use the restriction of the $L^{2}_{k+1/2}$-norm on $Coul^{CC}(X)$ to define the fiber direction norm on $\mathcal{W}_{X}|_{B_{i}}$. Using a partition of unity, we patch these norms together to form $\|\cdot \|_{F}$.

 Let us fix a fundamental domain $\mathfrak{D}\subset \mathcal{H}^1_{DC}(X)$  throughout this section. The following equivalence of norms is a consequence of the compactness of $\mathfrak{D}$.
 
\begin{lem}\label{fiberdirection norm}
There exists a positive constant $C$ such that for any $(\hat{a},\hat{\phi})\in Coul^{CC}(X)$ such that $\hat{p}_{\alpha}(\hat{a},\hat{\phi})\in
\mathfrak{D}$, we have
$$
\frac{\|[\hat{a},\hat{\phi}]\|_{F}}{C}\leq \|(\hat{a},\hat{\phi})\|_{L^{2}_{k+1/2}}\leq C\cdot (\|[\hat{a},\hat{\phi}]\|_{F}+1).
$$
\end{lem}


Lastly, we will consider some restriction maps on the bundle. Recall that the Coulomb slice on 3-manifolds is given by
\begin{align*}Coul(Y) :=  \{(a,\phi) \in L^2_k  \left(i\Omega^{1}(Y)\oplus\Gamma(S_{Y})\right) \mid d^{*}a=0\}. \end{align*}
 From the definition of double Coulomb slice, we obtain a natural restriction map 
 \begin{eqnarray*}  
r \colon Coul^{CC}(X) &\rightarrow &Coul(Y) \\
(\hat{a},\hat{\phi}) &  \mapsto &(\mathbf{t}\hat{a},\hat{\phi}|_{Y}). \nonumber
\end{eqnarray*}
We would want to also define a restriction map from $\mathcal{W}_X $ to $Coul(Y)$.
Notice that $r(\hat{u} \cdot (\hat{a},\hat{\phi}))$ might not be equal to $r (\hat{a},\hat{\phi}) $ even if $\hat{u} \in  \mathcal{G}^{h,\hat{o}}_{X,Y}$ because $\hat{u}|_{Y}\neq 1  $ in general. 
This is where we use the based path data $[\vec{\eta}]$ to define a ``twisted'' restriction map
\begin{equation*} 
\begin{split}
   r'  = r'_{ [\vec{\eta}] }  \colon Coul^{CC}(X)   &  
                     \rightarrow \prod\limits_{j=1}^{b_{0}(Y)}Coul(Y_{j})  = Coul(Y)    \\
    (\hat{a},\hat{\phi}) 
        &   \mapsto \prod\limits_{j=1}^{b_{0}(Y)}   (\mathbf{t}_{j}\hat{a},  e^{ i \int_{\eta_{j}}\hat{p}_{\alpha}(\hat{a},\hat{\phi})}     \cdot   \hat{\phi}|_{Y_{j}}).
\end{split}
\end{equation*}
The following result can be verified by a simple calculation.
\begin{lem} For each $\hat{u} \in \mathcal{G}^{h,\hat{o}}_{X,Y}$, we have $r'(\hat{u} \cdot (\hat{a},\hat{\phi})) = r'(\hat{a},\hat{\phi})$. Moreover, the twisted restriction map $r'$ does not depend on the choice of the representative $\vec{\eta}$ in its equivalent class.
\end{lem}
As a result, we can
define the induced twisted restriction map
\begin{equation*} 
      \tilde{r} = \tilde{r}_{[\vec{\eta}]}  \colon \mathcal{W}_X \rightarrow Coul(Y).
\end{equation*}
Note that $ \tilde{r}$ is fiberwise linear since $\hat{p}_{\alpha}(\hat{a},\hat{\phi}) $ is constant on each fiber. Moreover, there is a decomposition $(\tilde{r}_{\text{in}} ,\tilde{r}_{\text{out}}) \colon \mathcal{W}_X \rightarrow Coul(-Y_{\text{in}}) \times Coul(Y_{\text{out}}) $

\subsection{Seiberg--Witten maps and finite-dimensional approximation} \label{sec SWfinite}
On the boundary 3-manifold $Y$, we fix a base $\text{spin}^{c}$ connection $A_{0}$. We require that the induced curvature $F_{A_{0}^{t}}$
on $\text{det}(S_Y) $ equals $2\pi i\nu_{0}$,
where $\nu_{0}$ is the harmonic $2$-form representing $-c_{1}(\mathfrak{s})$. Furthermore, we pick a good perturbation $f
= (\bar{f} , \delta)$ where $\bar{f}$ is an extended cylinder function and $\delta $ is a real number (see \cite[Definition~2.3]{KLS1}
for details). Auxiliary choices in the construction of the unfolded spectrum $\underline{\text{SWF}}^{}(Y,\mathfrak{s}) $
will be made but not mentioned at this point. 
 
On the 4-manifold $X$, we fix
a base $\text{spin}^{c}$ connection $\hat{A}_{0}$ such that $\nabla_{\hat{A}_{0}} = \frac{d}{dt} + \nabla_{A_0} $ on $[-3,0]
\times Y$. 
As usual, the space of $\text{spin}^{c}$ connections
on $S_{X}$ can be identified with $i\Omega^{1}(X)$ via the correspondence $\hat{A} \mapsto \hat{A}-\hat{A}_{0}$. For a \mbox{1-form}
$\hat{a} \in i\Omega^{1}(X) $, we let $\slashed{D}^{+}_{\hat{a}} \colon \Gamma(S^+) \rightarrow \Gamma(S^-)$ be the Dirac
operator associated to the connection
$\hat{A}_{0}+ \hat{a}$.
We also denote by $\slashed{D}^+ := \slashed{D}^+_{0}$ the Dirac operator corresponding to the base connection $\hat{A}_0$,
so we can write
$\slashed{D}^+_{\hat{a}} = \slashed{D}^+_{} + \hat{\rho}(\hat{a}) $.
On $Y$, we denote by $\slashed{D}_{A_0 +a}$ the Dirac operator associated to the connection
$A_{0}+a$ where $a \in i\Omega^{1}(Y) $ and denote by $\slashed{D} := \slashed{D}_{A_{0}}$ 

Furthermore, we  perturb the Seiberg--Witten map by choosing the following data
\begin{itemize}

\item Pick a closed 2-form $\omega_{0} \in i\Omega^{2}(X)$ such 
that $\omega_{0}|_{[-3,0]\times Y}=\pi i\nu_{0}$. 
\item Pick a bump-function $\iota \colon [-3,0]\rightarrow [0,1]$
satisfying $\iota \equiv 0$ on $[-3,-2]$ and $\iota \equiv 1$ on $[-1,0] $ and $0\leq \iota'(x)\leq 2$. For $t\in [-3,0]$, denote by $a_{t}$ the pull
back of $\hat{a}$ by the inclusion $\{t\}\times Y\rightarrow X$ and let $\phi_{t}=\hat{\phi}|_{\{t\}\times Y}$.
    We define a perturbation on $X $ supported in the collar neighborhood of $Y$  by  
\begin{equation*}  
 \hat{q}(\hat{a},\hat{\phi}):=\iota(t)((dt \wedge \grad^{1}f(a_{t},\phi_{t})+*\grad^{1}f(a_{t},\phi_{t})),\grad^{2}f(a_{t},\phi_{t})).\end{equation*}

\end{itemize}

The (perturbed) Seiberg-Witten map is then given by
\begin{align*}
& SW\colon Coul^{CC}(X)\rightarrow L^{2}_{k-1/2}(i\Omega_{+}^{2}(X)\oplus \Gamma(S^{-}_{X})) \\
& SW(\hat{a},\hat{\phi})=(d^{+}\hat{a}, \slashed {D}^{+}\hat{\phi})+ (\frac{1}{2} F^{+}_{\hat{A}_{0}^{t}}-\hat{\rho}^{-1}(\hat{\phi}\hat{\phi}^{*})_{0}-\omega_{0}^{+},\hat{\rho}
(\hat{a})\hat{\phi})+\hat{q}(\hat{a},\hat{\phi}), \nonumber
\end{align*}
where $(\hat{\phi}\hat{\phi}^{*})_{0}$ denotes the trace-free part of $\hat{\phi}\hat{\phi}^{*}\in \Gamma(\text{End}(S_{X}^{+}))$. 
We consider a decomposition 
\begin{equation}\label{decomposition of sw}
SW=L+Q
\end{equation}
where
$$
L(\hat{a},\hat{\phi})=(d^{+}\hat{a},\slashed{D}^{+}_{\hat{p}_{\alpha}(\hat{a})}\hat{\phi})\text { and }Q=SW-L.
$$
 By computation similar to that in the proof of Proposition 11.4.1 of \cite{Kronheimer-Mrowka}, making use of the tameness condition on $\grad f$ (see \cite[Definition 10.5.1]{Kronheimer-Mrowka}),
we can deduce the following lemma:
\begin{lem}\label{quadratic part}
For any number $j\geq 2$, if a subset $U\subset Coul^{CC}(X)$ is bounded in $L^{2}_{j}$, then the set $Q(U)$ is also bounded
in $L^{2}_{j}$.
\end{lem}

We will next consider Seiberg--Witten maps on the Hilbert bundle $\mathcal{W}_X $. Notice that the map
\begin{align*} 
(SW,\hat{p}_{\alpha}) \colon Coul^{CC}(X)\rightarrow L^{2}_{k-1/2}(i\Omega^{2}_{+}(X)\oplus \Gamma(S_{X}^{-}))\times \mathcal{H}^1_{DC}(X)
\end{align*}
 is equivariant under the action of $\mathcal{G}^{h,\hat{o}}_{X,Y}$, where the action on the target space is given by
 $$
\hat{u} \cdot ((\omega,\hat{\phi}),\hat{h}):=((\omega,\hat u\hat{\phi}),\hat{h}-{\hat u}^{-1}d \hat u).
$$
Consequently, $(SW,\hat{p}_{\alpha})$ induces a bundle map over $\operatorname{Pic}^{0}(X,Y) $ denoted by
$$
       \overline{SW} \colon \mathcal{W}_X \rightarrow  (L^{2}_{k-1/2}(i\Omega^{2}_{+}(X)\oplus \Gamma(S_{X}^{-}))\times
\mathcal{H}^1_{DC}(X))/\mathcal{G}^{h,\hat{o}}_{X,Y}.
$$

By Kuiper's theorem, the Hilbert bundle $(L^{2}_{k-1/2}(i\Omega^{2}_{+}(X)\oplus \Gamma(S_{X}^{-}))\times \mathcal{H}^1_{DC}(X))/\mathcal{G}^{h,\hat{o}}_{X,Y}$ can be trivialized. We fix a trivialization and consider the induced projection from this bundle to its fiber $L^{2}_{k-1/2}(i\Omega^{2}_{+}(X)\oplus \Gamma(S_{X}^{-}))$. 
Composing the map $\overline{SW}$ with this projection, we obtain a map
$$
\widetilde{SW} \colon \mathcal{W}_X \rightarrow L^{2}_{k-1/2}(i\Omega^{2}_{+}(X)\oplus \Gamma(S_{X}^{-})).
$$

As the map $(L,\hat{p}_{\alpha})$ is also equivariant under the action of $\mathcal{G}^{h,\hat{o}}_{X,Y}$, the decomposition (\ref{decomposition of sw}) induces a decomposition
$$
\widetilde{SW}=\tilde{L}+\tilde{Q},
$$
where $\tilde{L}$ is a fiberwise linear map.

On the
3-dimensional Coulomb slice $Coul(Y) $, a Seiberg--Witten trajectory is a trajectory $\gamma \colon I \rightarrow Coul(Y) $ on some interval $I \subset \mathbb{R} $ satisfying an equation
\begin{equation*}
-\frac{d\gamma}{dt}(t)= (l+ c)(\gamma(t)), 
\end{equation*}
where $l+c $ comes from gradient of the perturbed Chern--Simons--Dirac functional $CSD_{\nu_{0},f}$ (cf. \cite[Section~2]{KLS1}). Recall that $l=(*d,\slashed{D}) $ and $c$ has nice compactness properties.

Let $V^{\mu}_{\lambda} \subset Coul(Y)$ be the span of eigenspaces of $l$
with eigenvalues in the interval $(\lambda,\mu]$ and let $p^{\mu}_{\lambda}$ be the $L^2$-orthogonal
projection onto $V^{\mu}_{\lambda}$. An approximated Seiberg--Witten
trajectory is a trajectory on a finite-dimensional subspace $\gamma \colon I \rightarrow V^{\mu}_{\lambda}$ satisfying
an equation
\begin{equation*}
-\frac{d\gamma}{dt}(t) = (l+p^{\mu}_{\lambda}\circ c)(\gamma(t)).
\end{equation*} 

From now on, let us fix a decreasing sequence of negative real numbers $\{\lambda_{n}\}$ and an increasing sequence
of positive real numbers $\{\mu_{n}\}$ such that $-\lambda_{n}, \mu_n \rightarrow \infty$. 
As a consequence of \cite[Proposition~3.1]{Khandhawit1}, the linear part
\begin{align*}
(\tilde{L},p^{\mu_{n}}_{-\infty}\circ \tilde{r}) \colon \mathcal{W}_X \rightarrow
L^{2}_{k-1/2}(i\Omega^{2}_{+}(X)\oplus \Gamma(S_{X}^{-}))\oplus V^{\mu_{n}}_{-\infty}
\end{align*} 
is fiberwise Fredholm.
Now we choose an increasing sequence $\{U_{n}\}$ of finite-dimensional subspaces of $L^{2}_{k-1/2}(i\Omega^{2}_{+}(X)\oplus
\Gamma(S_{X}^{-}))$ with the following two properties:
\begin{enumerate}[(i)]
\item As $n\rightarrow \infty$, the orthogonal projection $P_{U_{n}}:L^{2}_{k-1/2}(i\Omega^{2}_{+}(X)\oplus \Gamma(S_{X}^{-}))\rightarrow
U_{n}$ converges to the identity map pointwisely.
\item\label{item transversality} For any point $p\in \text{Pic}^{0}(X,Y)$ and any $n$, the restriction of $(\tilde{L}^{},p^{\mu_{n}}_{-\infty}\circ
\tilde{r}^{})$ to the fiber over $p$ is transverse to $U_{n}$.
\end{enumerate}

Note that $\hat{p}_{\alpha}(\hat{a}) = 0$ on $\partial X$ and hence the family of the Dirac operators $\slashed{D}^{+}_{\hat{p}_{\alpha}(\hat{a})}$ has no spectral flow. 
Consequently, we see that 
\begin{align*}
W_{n} := (\tilde{L},p^{\mu_{n}}_{-\infty}\circ \tilde{r})^{-1} ( U_n \times V^{\mu_{n}}_{\lambda_n} )
\end{align*}
is a finite-dimensional vector bundle over the Picard torus $\text{Pic}^{0}(X,Y)$. We
define an approximated Seiberg-Witten
map as 
\begin{align*}
\widetilde{SW}_{n} :=\tilde{L}+P_{U_{n}}\circ \tilde{Q} \colon W_{n}\rightarrow U_{n}.
\end{align*}

\subsection{Boundedness results} 
In this section, we will establish analytical results needed to set up our construction of the relative Bauer--Furuta invariants. Uniform boundedness of the following objects and their approximated analogues will be our main focus here.
\begin{defi}\label{defi: X-trajectory} A finite type $X$-trajectory is a pair $(\tilde{x},\gamma) $ such that
\begin{itemize}
\item  $\tilde{x} \in \mathcal{W}_X$ satisfying $\widetilde{SW} (\tilde{x}) = 0$;
\item $\gamma \colon [0 , \infty) \rightarrow Coul(Y) $ is a finite type Seiberg--Witten trajectory;
\item $\tilde{r} (\tilde{x}) = \gamma(0)$.
\end{itemize}
Recall that a smooth path in $Coul(Y)$ is called \emph{finite type} if it
is contained in a fixed bounded set (in the $L^{2}_{k}$-norm).
\end{defi}

With a basepoint chosen on each connected component $Y_j$, we recall that we can define  the based harmonic
gauge group $\mathcal{G}^{h,o}_{Y}\cong H^{1}(Y; \mathbb{Z}) $. The group $\mathcal{G}^{h,o}_{Y}$ has a residual
action on $Coul(Y)$. Then we consider a strip
of balls in $Coul(Y)$ translated by this action 
\begin{align*}
Str(R)=\{x\in Coul(Y) \mid \exists h\in \mathcal{G}^{h,o}_{Y} \text{ s.t. } \|h\cdot x\|_{L^{2}_{k}}\leq
R\}.
\end{align*}

Loosely speaking, a finite type $X$-trajectory corresponds to a Seiberg--Witten solution on $X^* := X \cup ([0,\infty) \times Y) $. The following result resembles \cite[Corollary~4.3]{Khandhawit1} but we give a more direct proof without relying on broken trajectories and regular perturbations.

\begin{thm}\label{boundedness for X-trajectory}
For any $M>0$, there exists a constant $R_0 (M)>0$ such that for any finite type $X$-trajectory $(\tilde{x},\gamma)$ satisfying
\begin{equation}\label{homology bounded condition}
\hat{p}_{\beta}(\tilde{x})\in [-M,M]^{b_{1,\beta}}
\end{equation}
we have $$\|\tilde{x}\|_{F}< R_0 (M)\text{ and }
\gamma([0,\infty))\subset \operatorname{int}(Str(R_0 (M))).$$
\end{thm}

\begin{proof}
Let $\{(\tilde{x}_n,\gamma_{n})\}$ be a sequence of finite type $X$-trajectories satisfying (\ref{homology bounded condition}). 
Without loss of generality, we may pick a representative $\tilde{x}_n = [(\hat{a}_{n},\hat{\phi}_{n}) ]  $ such that
\begin{equation}\label{beta bounded}
\hat{p}_{\alpha}(\hat{a}_{n},\hat{\phi}_{n})\in \mathfrak{D}
\end{equation}
where $\mathfrak{D}$ is the fundamental domain fixed before Lemma \ref{fiberdirection norm}.

Since $\gamma_n$ is finite type, we see that the energy of $\gamma_{n}|_{[t-1,t+1]}$ goes to
$0$ as $t \rightarrow \infty$ for any $n$. In particular, the energy of $\gamma_{n}|_{[t-1,t+1]}$ is bounded above by $1$ for any $n$ and any $t$ large enough compared to $n$. Then, it is not hard to show that there is a constant $R'$ such that 
$\gamma_{n}(t)\in \operatorname{int}({Str(R')})$
for any $n$ and any $t$ large enough compared to $n$.  
Since $CSD_{\nu_{0},f}$
is bounded on $\operatorname{int}({Str(R')})$ and $CSD_{\nu_{0},f}$ is decreasing along $\gamma_{n}$, we can obtain a uniform lower bound $C_{1}$ of $CSD_{\nu_{0},f}(\gamma_{n}(t))$ for any $n\in \mathbb{N}, t\geq 0$.

We now  consider solutions on $X' :=X\cup([0,1]\times Y)$ obtained by gluing together $(\hat{a}_{n},\hat{\phi}_{n})$ and $\gamma_{n}|_{[0,1]}$.
Remark that the condition $\tilde{r} (\tilde{x}) = \gamma(0) $ from the twisted restriction is slightly different from the setup in \cite[Corollary~4.3]{Khandhawit1}.
However, we can still glue in a controlled manner since we control $\hat{p}_{\alpha}(\hat{a}_{n},\hat{\phi}_{n})$ in (\ref{beta bounded}). The uniform lower bound $C_1$ of $CSD_{\nu_{0},f}(\gamma_{n}(t))$ implied that
the energy of these solutions on $X'$  (see \cite[(4.21),(24.25)]{Kronheimer-Mrowka} for definition) has a uniform upper bound.  We now apply the compactness theorem \cite[Theorem~~24.5.2]{Kronheimer-Mrowka} adapted
to the balanced situation; after passing to a subsequence and applying suitable gauge transformations, the solution on $X'$ converges in $C^{\infty}$ on the interior domain $X$. In particular, we can find $\hat{u}_{n}\in \mathcal{G}_{X}^{h,\hat{o}}$ such that $\hat{u}_{n}\cdot (\hat{a}_{n},\hat{\phi}_{n})$
converges in $L^{2}_{k+1/2}$ to some $(\hat{a}_{\infty},\hat{\phi}_{\infty})\in Coul^{CC}(X)$.

By (\ref{homology bounded condition}) and (\ref{beta bounded}), we have controlled values of $\hat{p}_\alpha $ and $\hat{p}_\beta $ of $(\hat{a}_{n},\hat{\phi}_{n})$. This implies that $\{\hat{u}_{n}\}$ takes only finitely many values in $\mathcal{G}^{h,\hat{o}}_{X}$.
  After passing to a subsequence,
we can assume that $\hat{u}_{n}$ does not depend on $n$ and $(\hat{a}_{n},\hat{\phi}_{n})$
converges in $L^{2}_{k+1/2}$.

On the collar neighborhood $[-1,0] \times Y $ of $X$, the solution $(\hat{a}_{n},\hat{\phi}_{n})$ can be transformed to a Seiberg--Witten trajectory in a controlled manner. We subsequently glue this part together with $\gamma_n $ to obtain a Seiberg--Witten trajectory
$$
\gamma'_{n} \colon [-1,\infty)\rightarrow Coul(Y).
$$ 
Since $(\hat{a}_{n},\hat{\phi}_{n})$ converges in $L^{2}_{k+1/2}$, we have a uniform upper bound $C_{2}$ on
$ CSD_{\nu_{0},f}(\gamma'_n (-1))$. As a result, the energy of a trajectory $\gamma'_{n}|_{[t-1,t+1]}$ is bounded above by $C_{2}-C_{1}$ for any $t \ge 0 $ and $n \in \mathbb{N} $. We can then conclude that there is a constant $R''$ such that
$\gamma_{n}(t)\in\operatorname{int}(Str(R''))$ for any $t\ge0$ and $n\in \mathbb{N} $. This finishes the proof. 
\end{proof}

\begin{cor}\label{boundedness on cylinder}
There exists a uniform constant $R_{1}$ such that for any finite type $X$-trajectory $(\tilde{x},\gamma)$, we have $\gamma(t)\in
Str(R_{1})$ for any $t\in [0,\infty)$.
\end{cor}
\begin{proof}
By looking at the lattice induced by the chosen basis on $\im \iota^* $, there is a constant $C$ such that, for any $\tilde{x} \in \mathcal{W}_X $, one can find a gauge transformation $\hat{u}\in \mathcal{G}^{h,\hat{o}}_{X} $ satisfying $\hat{p}_\beta (\hat{u}\cdot \tilde{x}) \in [-C,C]^{b_{1,\beta}} $.

Let $(\tilde{x},\gamma)$ be an arbitrary finite type $X$-trajectory.
We then apply Theorem~\ref{boundedness for X-trajectory} to $(\hat{u}\cdot \tilde{x},(\hat{u}|_{Y})\cdot \gamma)$ with $M = C$ and $\hat{u}$ chosen as in the previous paragraph. Consequently, we may set $R_1 = R_0(C)$ so that $(\hat{u}|_{Y})\cdot
\gamma(t)\in \operatorname{int}(Str(R_1))$ for any $t\in [0,\infty)$.
This implies $\gamma(t)\in \operatorname{int}(Str(R_1))$ for any $t\in [0,\infty)$.

\end{proof}   

Now we consider an approximated version of $X$-trajectories. 
\begin{defi}\label{the approximated X-trajectory}
For $n\in \mathbb{N}$, $\epsilon\in [0,\infty)$, and $T \in (0,\infty] $, a finite type $(n,\epsilon)$-approximated $X$-trajectory of length $T$ is a pair $(\tilde{x},\gamma)$ such that
\begin{itemize}
\item $\tilde{x}\in W_n$ satisfies $\| \widetilde{SW_{n}}(\tilde{x} )\|_{L^{2}_{k-1/2}} \leq \epsilon$;
\item $\gamma : [0,T) \rightarrow V^{\mu_{n}}_{\lambda_{n}}$ is a finite type trajectory satisfying
 $-\frac{d\gamma}{dt}(t)= (l+p^{\mu_n}_{\lambda_n}\circ c)(\gamma(t))$; 
\item $\gamma(0)=p^{\mu_{n}}_{-\infty}\circ \tilde{r}(\tilde{x})$.
\end{itemize}
 Note that $p^{\lambda_{n}}_{-\infty} \circ \tilde{r}(\tilde{x})$
always belongs to $V^{\mu_{n}}_{\lambda_{n}}$ from the definition of $W_{n}$.
\end{defi}

The proof of the following convergence of approximated trajectories is a slight adaption of \cite[Lemma~4.4]{Khandhawit1} and we omit it.
\begin{lem}\label{convegence of approximated X-trajectories}
Let $\tilde{S},S$ be bounded subsets of $\mathcal{W}_{X}$ and $Coul(Y)$ respectively. Let $\{(\tilde{x}_{j},\gamma_{j})\}$ be a sequence of finite type $(n_j, \epsilon_j)$-approximated $X$-trajectory of length $T_j$ such that $\tilde{x}_{j}\in \tilde{S}, \gamma_{j}\subset
S$ for any $j$ and $(n_j , \epsilon_j , T_j ) \rightarrow (\infty ,0 , \infty) $. 
Then there exists a finite type $X$-trajectory $(\tilde{x}_{\infty},\gamma_{\infty})$ such that, after passing to a subsequence,
we have
\begin{itemize}
\item $\tilde{x}_{j}$ converges  to $\tilde{x}_{\infty}$ in $\mathcal{W}_{X}$;
\item $\gamma_{j}$ converges to $\gamma_{\infty}$ uniformly in $L^{2}_{k}$ on any compact subset of $[0,\infty)$.
\end{itemize}
\end{lem}

As a result of this lemma, we can deduce boundedness for approximated $X$-trajectories.

\begin{pro}\label{type A boundedness}
Let $M\geq 0$ be a fixed number.  For any bounded subsets $\tilde{S} \subset \mathcal{W}_{X}$ and $S\subset Coul(Y)$, there exist $\epsilon_{0}, N,\bar{T}\in (0,\infty)$ such that: for any finite type
$(n,\epsilon)$-approximated $X$-trajectory $(\tilde{x},\gamma)$ of length $T\geq \bar{T}$ satisfying
$$
n\geq N,\ \epsilon\leq \epsilon_{0},\ \tilde{x}\in \tilde{S}, \gamma\subset\tilde{S}\text{ and }\hat{p}_{\beta}(\tilde{x})\in
[-M,M]^{b_{1,\beta}} ,
$$
we have $\|\tilde{x}\|_{F}<R_0 ({M})$ where $R_0(M)$ is the constant from Theorem~\ref{boundedness for X-trajectory}.
\end{pro}

\begin{proof}
Suppose that the result is not true for some $\tilde{S},S$. There would be  a sequence $\{(\tilde{x}_{j},\gamma_{j})\}$ of finite type $(n_j, \epsilon_j)$-approximated $X$-trajectory of length $T_j$ with $\tilde{x}_{j}\in
\tilde{S}, \gamma_{j}\subset
S$ and $(n_j , \epsilon_j , T_j ) \rightarrow (\infty ,0 , \infty) $ such that $\|\tilde{x}_{j}\|_{F}\geq R_{0}(M) $ and $\hat{p}_{\beta}(\tilde{x})\in [-M,M]^{b_{1,\beta}}$.

By Lemma~\ref{convegence of approximated X-trajectories}, after passing to a subsequence, we can find a finite type $X$-trajectory
$(\tilde{x}_{\infty},\gamma_{\infty})$ such that $\tilde{x}_{j}\rightarrow \tilde{x}_{\infty}$ in $\mathcal{W}_X $.
In particular, this implies
\[
\begin{split}
&\|\tilde{x}_{\infty}\|_{F}=\lim_{j\rightarrow \infty}\|\tilde{x}_{j}\|_{F}\geq R_{0}(M)\text{ and }  \\
& \tilde{p}_{\beta}(x_{\infty})=\lim_{j\rightarrow
\infty}\tilde{p}_{\beta}(\tilde{x}_{j}) \in [-M,M]^{b_{1,\beta}},
\end{split}
\]
which is a contradiction with Theorem \ref{boundedness for X-trajectory}.
\end{proof}

\begin{pro}\label{boundedness on cylinder for approximated solutions}
There exists a constant $R_{2}$ with the following significance: for any bounded subsets $\tilde{S}\subset \mathcal{W}_{X}$
and $S\subset Coul(Y)$, there exist $\epsilon_{0},N,\bar{T}\in (0,+\infty)$ such that for any finite type $(n,\epsilon)$-approximated
$X$-trajectory $(\tilde{x},\gamma)$ of length $T\geq \bar{T}$ satisfying
$$
n\geq N,\ \epsilon\leq \epsilon_{0},\ \tilde{x} \in \tilde{S} \text{ and } \gamma \subset S
$$
We have  $\gamma|_{[0,T-\bar{T}]}\subset \text{Str}(R_{2})$.    
\end{pro}

\begin{proof}
Recall that there is a universal constant $R_0$ such that any sufficiently approximated Seiberg--Witten trajectory $\gamma' : [-T, T] \rightarrow V_{\lambda}^{\mu}$ with sufficiently long length $T$ and with $\gamma' \subset S$  must satisfy $\gamma(0) \in \text{Str}(R_0) $ (cf. the constant $R_0$ from \cite[Corollary~3.7]{KLS1}). We pick $R_{2}=\max\{R_{0},R_{1}\}$ where $R_{1}$ is the constant from Corollary \ref{boundedness on cylinder}. 

Suppose the result is not true for some $\tilde{S},S$. Then we can find  sequences $n_j, \epsilon_j, \bar{T}_j, T_j$ with $n_j \rightarrow \infty, \bar{T}_j \leq T_j, \bar{T}_j \rightarrow \infty$ such that there is a sequence $\{(\tilde{x}_{j},\gamma_{j})\}$ of finite type $(n_j, \epsilon_j)$-approximated $X$-trajectory of length $T_j$ with $\tilde{x}_{j} \subset \tilde{S}, \gamma_{j}\subset S$ and  with $\gamma_j([0, T_j - \bar{T}_j]) \not \subset Str(R_2)$.

We have a number  $t_j \in [0, T_j - \bar{T}_j]$ such that $\gamma_{j}(t_{j}) \notin Str(R_{2})$. 
The property of $R_0$ forces $t_{j} $ to converge to a finite number $t_\infty $ after passing to a subsequence.

 By Lemma \ref{convegence of approximated X-trajectories},
there exists an finite type $X$-trajectory $(\tilde{x}_{\infty},\gamma_{\infty})$ such that, after passing to a subsequence,
$\gamma_{j}$ converges to $\gamma_{\infty}$ uniformly in $L^{2}_{k}$ on any compact subset  of $[0,\infty)$. In particular, $\gamma_j (t_j) \rightarrow \gamma_\infty (t_\infty)$ which contradicts with Corollary~\ref{boundedness on cylinder}.

\end{proof}

\subsection{Construction}   \label{sec bfconstuct}
The majority of this section, in fact, is dedicated to construction of type-A unfolded relative invariant. The construction of type-R invariant can be obtained almost immediately after applying duality argument. 

Let us pick $\tilde{R}$ a number greater than the constant $R_{2}$ from Proposition \ref{boundedness on cylinder for approximated solutions}. Recall that the unfolded spectra $\swf^{A}(Y_{\text{out}})$ and $\swf^{R}(-{Y_{\text{in}}})$  are obtained by cutting the unbounded set $Str_{Y}({\tilde{R}})$
into bounded subsets and applying finite dimensional approximations. With a choice of cutting functions, we obtain increasing sequences of bounded sets $\{ J_{m}^{n, -}(-Y_{\text{in}}) \}_{m} $ contained in $Str_{Y_{in}}({\tilde{R}}) $ 
and $\{ J_{m}^{n, +}(Y_{\text{out}}) \}_{m} $ contained in $Str_{Y_{out}}({\tilde{R}}) $ for each positive integer $n$.
See Section~\ref{subsec Unfolded} for brief summary.   

For a normed vector bundle $V$, we will denote by $B(V,r)$ the disk bundle of radius $r$ and denote by $S(V,r)$ the sphere
bundle of radius $r$. We will consider a subbundle of $\mathcal{W}_X $ given by
$$
\mathcal{W}_{X,\beta}:=\{  \tilde{x} \in \mathcal{W}_{X}\mid \hat{p}_{\beta, \text{out} }( \tilde{x} )=0\}.
$$
We also denote $W_{n,\beta} =  W_n \cap \mathcal{W}_{X, \beta}$ and let $\widetilde{SW}_{n,\beta}^{} $ be the restriction 
of $\widetilde{SW}_{n}^{} $ on $W_{n,\beta}$. 

For a fixed positive integer $m_{0}$, since $\{ J_{m_{0}}^{n, -}(-Y_{\text{in}}) \} $ is bounded, we can find a number $ M(m_0)$ such that $|\int_{\beta_{j}}ia | \leq M(m_0)  $ for all $(a,\phi)\in  J_{m_{0}}^{-}(-Y_{\text{in}})  $ and  $j = 1, \ldots,  b_{\text{in}}  $. We then choose a number $R$ greater than $R_0 (M(m_0)) $ the constant from Theorem~\ref{boundedness for X-trajectory}. Since $\tilde{r}_{\text{out}}(B(\mathcal{W}_{X},R)) $ is bounded, we can find a positive integer $m_1$ such that
\begin{equation*}  
\tilde{r}_{\text{out}}(B(\mathcal{W}_{X},R))\cap Str_{Y_{\text{out}}}(\tilde{R})\subset J^{+}_{m_{1}}(Y_{\text{out}}).
\end{equation*}

For $\epsilon>0,\ n\in \mathbb{N}$, we consider the following subsets of  $V_{\lambda_n}^{\mu_n}$ :
\begin{equation} \label{eq K_1 K_2}
\begin{split}
& K_{1}(n,m_{0},R,\epsilon) =    \\
&\quad  \left( J_{m_{0}}^{n,-}(-Y_{\text{in}}) \times Str_{Y_{\text{out}}} (\tilde{R})\right)  \cap \\
&\qquad  \qquad p^{\mu_{n}}_{-\infty}\circ\tilde{r}\left(\widetilde{SW}_{n,\beta}^{-1}(B(U_{n},\epsilon))\cap
B(W_{n,\beta},R)\right),  \\
& K_{2}(n,m_{0},R,\epsilon) =    \\
& \quad  
\left\{ \left( J_{m_{0}}^{n,-}(-Y_{\text{in}}) \times Str_{Y_{\text{out}}}(\tilde{R})\right)  \cap    \right. \\
& \left.   \qquad \qquad
p^{\mu_{n}}_{-\infty}\circ\tilde{r}\left(\widetilde{SW}_{n,\beta}^{-1}(B(U_{n},\epsilon))\cap
S(W_{n,\beta},R)\right) \right\} \cup \\
&  \quad 
\left\{ \partial\left(J_{m_{0}}^{n,-}(-Y_{\text{in}}) \times Str_{Y_{\text{out}}}(\tilde{R})\right) \cap
\right.  \\
& \left.
\qquad \qquad
\left( p^{\mu_{n}}_{-\infty}\circ\tilde{r}\left(\widetilde{SW}_{n,\beta}^{-1}(B(U_{n},\epsilon))\cap
B(W_{n,\beta}, R)\right)  \right) \right\}.  
\end{split}
\end{equation}
Notice that 
$K_{1}(n,m_{0},R,\epsilon)\subset J^{n,-}_{m_{0}}(-Y_{\text{in}})\times J^{n,+}_{m_{1}}(Y_{\text{out}})$ from our choice of $m_{1}$ and $K_{2}(n,m_{0},R,\epsilon)$ plays a role of a boundary of $K_{1}(n,m_{0},R,\epsilon)$.

The following is the key result of this section (cf. \cite[Proposition~4.5]{Khandhawit1}).

\begin{pro} \label{prop 4dimpreindex}
For a choice of $m_{0},m_{1}$ and $R$ chosen above, there exist $N\in \mathbb{N}$ and $\bar{T},\epsilon_{0}>0$ such
that, for any $n\geq N$ and $\epsilon\leq \epsilon_{0}$, the pair $(K_{1}(n,m_{0},R,\epsilon),K_{2}(n,m_{0},R,\epsilon))$ is a $\bar{T}$-tame pre-index pair in an isolating neighborhood $J^{n,-}_{m_{0}}(-Y_{\text{in}})\times J^{n,+}_{m_{1}}(Y_{\text{out}}) $.
\end{pro}

\begin{proof} We choose numbers $(N, \bar{T} , \epsilon_0 ) $ satisfying both Proposition~\ref{type A boundedness} and Proposition~\ref{boundedness on cylinder for approximated solutions} with $\tilde{S}=B( \mathcal{W}_X, R) $, 
$S=J^{-}_{m_{0}}(-Y_{\text{in}})\times J^{+}_{m_{1}}(Y_{\text{out}})$, and 
$M=M(m_0)$ . Moreover, we may pick a larger $N$ so that 
\[
 J^{n,-}_{m_{0}}(-Y_{\text{in}})\times J^{n,+}_{m_{1}}(Y_{\text{out}})
 \]
 is an isolating neighborhood for all $n > N $ (cf. \cite[Lemma~5.5 and Proposition~5.8]{KLS1}).
  We will check directly that  
  \[
  (K_{1}(n,m_{0},R,\epsilon), K_{2}(n,m_{0},R,\epsilon))
  \]
is a $\bar{T}$-tame pre-index pair.

Suppose that $y \in K_{1}(n,m_{0},R,\epsilon)$ and $\varphi_n(y, [0,T]) \subset J^{n,-}_{m_{0}}(-Y_{\text{in}})\times J^{n,+}_{m_{1}}(Y_{\text{out}}) $ with $T \ge \bar{T} $. From definition, there is $\tilde{x} \in W_{n,\beta} $ such that $\|\widetilde{SW_{n}}(\tilde{x})\|\leq \epsilon$ and $p^{\mu_{n}}_{-\infty}\circ\tilde{r}(\tilde{x}) = y $. These give rise to a finite type $(n,\epsilon)$-approximated $X$-trajectory $(\tilde{x},\gamma)$ of length $T$. By Proposition~\ref{boundedness on cylinder for approximated solutions}, we have $\varphi_n(y, [0, T-\bar{T}]) \subset Str(R_2) \subset \operatorname{int}(Str(\tilde{R})) $. From our choices of $J^{-}_{m_{0}} , J^{+}_{m_{1}}$, it is not hard to check that $\varphi_n(y,  [0,  T-\bar{T}])$ lies in some compact subset inside the interior of $ J^{n,-}_{m_{0}}(-Y_{\text{in}})\times J^{n,+}_{m_{1}}(Y_{\text{out}})  $.  

For the second pre-index pair condition, let us assume that $y \in K_{2}(n,m_{0},R,\epsilon)$ 
 and $\varphi_{n} (y,  [0,\bar{T}]) \subset  J^{n,-}_{m_{0}}(-Y_{\text{in}})\times J^{n,+}_{m_{1}}(Y_{\text{out}})  $. This also gives rise to a finite type $(n,\epsilon)$-approximated $X$-trajectory $(\tilde{x},\gamma)$ of length $\bar{T}$.
Since $p^{\mu_{n}}_{-\infty}\circ\tilde{r}_{ \text{in} } (\tilde{x}) \in J^{n,-}_{m_{0}}(-Y_{\text{in}}) $ and $\tilde{x} \in \mathcal{W}_{X,\beta}$, we can see that $\hat{p}_{\beta}(\tilde{x})\in [-M(m_0),M(m_0)]^{b_{1,\beta}} $.

By Proposition \ref{type A boundedness}, we have $\|\tilde{x}\|_{F}<R_0 ({M}) < R$, which implies that 
\[
   y \in \partial\left(J_{m_{0}}^{n,-}(-Y_{\text{in}}) \times Str_{Y_{\text{out}}}(\tilde{R})\right).
\]
Again, from Proposition~\ref{boundedness on cylinder for approximated solutions}, we must have 
\[
    y \in   \left\{ \partial J_{m_0}^{n, -}(- Y_{\text{in}}) \setminus \partial Str_{Y_{\text{in}}}(\tilde{R})  \right\}  \times 
            Str_{Y_{\text{out}}}( \tilde{R} ).
\]
This is impossible because the approximated trajectories  on 
\[
\partial J_{m_0}^{n, -}(- Y_{\text{in}}) \setminus \partial Str_{Y_{\text{in}}}(\tilde{R})
\]
immediately  leave $J_{m_0}^{n, -}(-Y_{in})$.
           
\end{proof}

The proposition allows us to consider a map
\begin{equation}   \label{eq map v}
\begin{split}
 & \upsilon(n,m_{0},R,\epsilon) \colon  B(W_{n,\beta},R)/ S(W_{n,\beta},R)\\ 
& \quad \rightarrow (B(U_{n},\epsilon)/S(U_{n},\epsilon))\wedge
(K_{1}(n,m_{0},R,\epsilon)/K_{2}(n,m_{0},R,\epsilon))
\end{split}
\end{equation}
given by

\[
\begin{split}
&\upsilon(n,m_{0},R,\epsilon)(x):=  \\
& \left\{  \begin{array}{l l}
   (\widetilde{SW}_{n,\beta}(x), [p^{\mu_{n}}_{-\infty}\circ \tilde{r}(x)] ) & \text{if } 
   \left\{
   \begin{array}{l}
   \|\widetilde{SW}_{n, \beta}(x)\|_{L^{2}_{k-1/2}}\leq
\epsilon, \\
  p^{\mu_{n}}_{-\infty}\circ \tilde{r}(x)\in K_1(n,m_0, R,\epsilon)
\end{array}
\right.    \\
    \ \ast   &  \text{otherwise.}
  \end{array} \right.
  \end{split}
\]
It follows from our construction that this map is well-defined and continuous.
By Proposition~\ref{prop 4dimpreindex} and Theorem~\ref{from pre-index to index}, we  have a canonical map from 
$K_{1}(n,m_{0},R,\epsilon)/K_{2}(n,m_{0},R,\epsilon)$ to the Conley index of $J^{n,-}_{m_{0}}(-Y_{\text{in}})\times J^{n,+}_{m_{1}}(Y_{\text{out}})
$. This gives a map
\begin{equation}  \label{equation: definition of upsilon}
\begin{split}
&\tilde{\upsilon}(n,m_{0},R,\epsilon) \colon B(W_{n,\beta},R)/S(W_{n,\beta},R) \\ 
& \quad \rightarrow (B(U_{n},\epsilon)/S(U_{n},\epsilon)) \\
& \qquad \qquad \wedge
I(\inv(J^{n,-}_{m_{0}}(-Y_{\text{in}}))) \wedge I(\inv(J^{n,+}_{m_{1}}(Y_{\text{out}}))).
\end{split}\end{equation}
It is a standard argument to check that $\tilde{\upsilon}(n,m_{0},R,\epsilon)$ does not depend on $R $ or $\epsilon $ as long as they satisfy all the requirements to define ${\upsilon}(n,m_{0},R,\epsilon) $.

Before proceeding, let us describe the Thom space 
\[
B(W_{n,\beta},R)/S(W_{n,\beta},R)
\]
in term of index bundle.
Consider a family of Dirac operators
\begin{align*}
\mathbf{D} \colon  L^{2}_{k+1/2}( S_{X}^{+})\times \mathcal{H}^1_{DC}(X) &\rightarrow L^{2}_{k-1/2}( S_{X}^{-}) \times H^-_{Dir} \times \mathcal{H}^1_{DC}(X) \\
(\hat{\phi} , h) & \mapsto ( \slashed{D}^{+}_{h}\hat{\phi} , \Pi^{-}_{Dir}(\hat{\phi}|_{Y} ) , h),
\end{align*}
where $H^-_{Dir}$ is the closure in $L^{2}_{k}(\Gamma(S_{Y}))$ of the subspace spanned by the eigenvectors of $\slashed{D}_{}$ 
with nonpositive eigenvalues and let $\Pi^{-}_{Dir}$ be the orthogonal projection. 
As in Section~\ref{sec SWfinite}, this map is equivariant under an action by $\mathcal{G}^{h,\hat{o}}_{X,Y}$. We then take the quotient to obtain a map between Hilbert bundles over $\operatorname{Pic}^{0}(X,Y) $ and trivialize the right hand side so that we have\begin{align*}
\widetilde{\mathbf{D}} \colon  (L^{2}_{k+1/2}( S_{X}^{+})\times \mathcal{H}^1_{DC}(X))/ \mathcal{G}^{h,\hat{o}}_{X,Y} &\rightarrow L^{2}_{k-1/2}( S_{X}^{-})\times H^-_{Dir}.
\end{align*}
Since $\widetilde{\mathbf{D}} $ is fiberwise Fredholm, the preimage $\widetilde{\mathbf{D}}^{-1}(U) $ is a finite dimensional subbundle for a finite dimensional subspace $U \subset L^{2}_{k-1/2}( S_{X}^{-})\times H^-_{Dir} $ transverse to the image of  the restriction of $\widetilde{\mathbf{D}}$ to any fiber.
Here we use the fact that  the rank of $\widetilde{\mathbf{D}}^{-1}(U)$ is constant because $h|_{Y} = 0$  and there is no spectral flow. 

 We consider the desuspension $\Sigma^{-U}  B(\widetilde{\mathbf{D}}^{-1}(U),R) / S(\widetilde{\mathbf{D}}^{-1}(U),R)$ of the Thom space in the stable category $\mathfrak{C} $. The following lemma follows from a standard homotopy argument.

\begin{lem} \label{lem thombundle} 
The object  $\Sigma^{-U}  B(\widetilde{\mathbf{D}}^{-1}(U),R) / S(\widetilde{\mathbf{D}}^{-1}(U),R)$ does not depend on any choice in the construction given that  $\hat{g}_{}|_{Y}=g$ and $\hat{A}_{0}|_{Y}={A}_{0}$. 
 We will call this object Thom spectrum
of virtual index bundle associated to the Dirac operators, denoted by $T(X,\hat{\mathfrak{s}},A_{0},g,\hat{o};S^{1}) $.
\end{lem}

\begin{rmk} For different choices of base points, one can construct an isomorphism
by choosing a path between them. However, isomorphisms given by different paths 
are  different unless they are homotopic relative to $Y$.
\end{rmk}

Recall from Section~\ref{subsec Unfolded} that we have desuspended Conley indices
\begin{equation*} 
\begin{split}
& I^{n,-}_{m_0}(-Y_{\text{in}}) = \Sigma^{-V^0_{\lambda_n}(-Y_{\text{in}})}I(\inv(J^{n,-}_{m_{0}}(-Y_{\text{in}}))), \\
& I^{n,+}_{m_1}(Y_{\text{out}}) = \Sigma^{-\bar{V}^0_{\lambda_n}(Y_{\text{out}})}I(\inv(J^{n,+}_{m_{1}}(Y_{\text{out}}))).
\end{split}
\end{equation*}
We see that if we desuspend the map $\tilde{\upsilon}(n,m_{0},R,\epsilon)$ by $V^{0}_{\lambda_{n}}(-Y_{\text{in}})\oplus
\bar{V}^{0}_{\lambda_{n}}(Y_{\text{out}})\oplus U_{n}$, the right hand side will become $I^{n,-}_{m_{0}}(-Y_{\text{in}})\wedge
I^{n,+}_{m_{1}}(Y_{\text{out}})$. As a consequence of Lemma~\ref{lem thombundle}, we can also identify the left hand side after desuspension as follows:

\begin{lem}
 Let $V^+_X$ be a maximal positive subspace of 
 \[
 \operatorname{im} (H^2(X, \partial X;\R) \rightarrow H^2(X;\R))
 \]
 with respect to the intersection form and let $V_{\text{in}}$ be the cokernel of $\iota^* \colon H^{1}(X;\mathbb{R})\rightarrow H^{1}(Y_{\text{in}};\mathbb{R})$.
Then, we have
\begin{align*}  
&\Sigma^{-(V^{0}_{\lambda_{n}}(-Y_{\text{in}})  \oplus
\bar{V}^{0}_{\lambda_{n}}(Y_{\text{out}})\oplus U_{n})} B(W_{n,\beta},R)/S(W_{n,\beta},R)   \\
&  \qquad  \cong 
\Sigma^{-(V^{+}_X\oplus  V_{\text{in}})}T(X,\hat{\mathfrak{s}},A_{0},g,\hat{o};S^{1}).
\end{align*}

\end{lem}

\begin{proof} 
This is a bundle version of index computation in \cite[Proposition~3.1]{Khandhawit1}. From there, we are only left to keep track of $H^1(X;\mathbb{R}) $ and $H^1(Y;\R)$ as we pass to bundle and subspace, i.e. the base of the bundle is the torus of dimension $b_{1,\alpha} $ and we take a slice of codimension $b_{1,\beta}-b_{\text{in}} $. Note that we desuspend by  $\bar{V}^0_{\lambda_n}(Y_{\text{out}}) $, the orthogonal complement  of $H^1 (Y_{\text{out}};\mathbb{R}) $ in ${V}^0_{\lambda_n}(Y_{\text{out}})$. One may compute the rank of the Thom space of the index bundle of the real part of $(\tilde{L}, p^0 \circ \tilde{r})$ suspended by $H^1(Y_{\text{out}};\R)$ as follows
\[
\begin{split}
&b_1(X) - b^+(X) - b_1(Y) - b_{1,\alpha}- (b_{1,\beta}-b_{\text{in}}) + b_1(Y_{\text{out}})   \\
&\quad  = 
-b^+(X) - (b_1(Y_{\text{in}})-b_{\text{in}}).
\end{split}
\]
 The desired isomorphism follows in the same manner.    
 
 \end{proof}

Consequently, we obtain a morphism 
\begin{equation}\label{morphism 1}
\psi^{n}_{m_{0},m_{1}}\colon \Sigma^{-(V^{+}_X\oplus  V_{\text{in}})}T(X,\hat{\mathfrak{s}},A_{0},g,\hat{o};S^{1})\rightarrow I^{n,-}_{m_{0}}(-Y_{\text{in}})\wedge
I^{n,+}_{m_{1}}(Y_{\text{out}})
\end{equation}
in the stable category $\mathfrak{C}$.  Note that such a morphism is defined for any positive integer $m_{0}$ with $m_{1}$ large relative to $m_{0}$ and $n$
large relative to $m_{0},m_{1}$. 

Recall that, to define unfolded spectra  $\swf^{A}(Y_{\text{out}})$ and $\swf^{R}(-{Y_{\text{in}}})$, we have canonical isomorphisms
\begin{align*}
&\tilde{\rho}_{m_{0}}^{n,-}(-Y_{\text{in}}) \colon I^{n,-}_{m_{0}}(-Y_{\text{in}}) \rightarrow I^{n+1,-}_{m_{0}}(-Y_{\text{in}}),  \\
& \tilde{\rho}_{m_{1}}^{n,+}(Y_{\text{out}}) \colon I^{n,+}_{m_{1}}(Y_{\text{out}}) \rightarrow I^{n+1,+}_{m_{1}}(Y_{\text{out}})
\end{align*}
and also morphisms
\begin{align*} 
\tilde{i}_{m_{0}-1}^{n,-} \colon I^{n,-}_{m_{0}}(-Y_{\text{in}}) \rightarrow I^{n,-}_{m_{0}-1}(-Y_{\text{in}}) 
\text{ and } \tilde{i}_{m_{1}}^{n,+} \colon I^{n,+}_{m_{1}}(Y_{\text{out}}) \rightarrow I^{n,+}_{m_{1}+1}(Y_{\text{out}})
\end{align*}
induced by repeller and attractor respectively.
To have a morphism to the unfolded spectra, we have to check that the maps $\{ \psi^{n}_{m_{0},m_{1}} \} $ are compatible with all such morphisms.

\begin{lem}\label{morphism compatible}
When $n$ is large enough relative to $m_{0},m_{1}$, we have the following:
\begin{enumerate}
\item $(\tilde{\rho}_{m_{0}}^{n,-}(-Y_{\text{in}})\wedge \tilde{\rho}_{m_{1}}^{n,+}(Y_{\text{out}}))\circ \psi^{n}_{m_{0},m_{1}}=\psi^{n+1}_{m_{0},m_{1}};
$
\item $(\tilde{i}_{m_{0}-1}^{n,-}\wedge \id_{I^{n,+}_{m_{1}}(Y_{\text{out}})})\circ \psi^{n}_{m_{0},m_{1}}=\psi^{n}_{m_{0}-1,m_{1}};
$
\item $(\id_{I^{n,-}_{m_{0}}(-Y_{\text{in}})}\wedge \tilde{i}_{m_{1}}^{n,+})\circ \psi^{n}_{m_{0},m_{1}}=\psi^{n}_{m_{0},m_{1}+1}.$
\end{enumerate}

\end{lem}
\begin{proof}
The proof of (1) can be given by standard homotopy arguments similar to \cite[Section~9]{Manolescu1} and \cite[Proposition~5.6]{KLS1}. Whereas (2) and (3) follow from Proposition~\ref{pre-index map compatible with attractor} and \ref{pre-index map compatible with repellor} respectively.
\end{proof}

The last step is to apply  Spanier--Whitehead duality between  $I^{n,+}_{m_{0}}(Y_{\text{in}})$  and $I^{n,-}_{m_{0}}(-Y_{\text{in}})$(see Section~\ref{section spanierwhitehead} and \ref{section dualswf} for details). As a result, we can turn  the morphism $\psi^{n}_{m_{0},m_{1}} $  to a morphism 
\begin{equation*} 
\widetilde{\psi}^{n}_{m_{0},m_{1}} \colon \Sigma^{-(V^{+}_X\oplus  V_{\text{in}})}T(X,\hat{\mathfrak{s}},A_{0},g,\hat{o};S^{1})\wedge  I^{n,+}_{m_{0}}(Y_{\text{in}})\rightarrow
I^{n,+}_{m_{1}}(Y_{\text{out}}),
\end{equation*}
which will define the relative Bauer--Furuta invariant.

\begin{defi} \label{def BFtypeA}
For the cobordism $X \colon Y_{\text{in}} \rightarrow Y_{\text{out}}$,  
the collection of morphisms $\{\widetilde{\psi}^{n}_{m_{0},m_{1}} \mid m_{0}\in \mathbb{N},\ m_{1}\gg m_{0}, n\gg m_{0},m_{1}\}$ in $\mathfrak{C} $ gives rise to a morphism
\begin{align}
& \underline{\textnormal{bf}}^{A}(X,\hat{\mathfrak{s}},A_{0},g,\hat{o},[\vec{\eta}];S^{1}) \colon   \nonumber \\
& \Sigma^{-(V^{+}_X\oplus  V_{\text{in}})}T(X,\hat{\mathfrak{s}},A_{0},g,\hat{o};S^{1})\wedge \underline{\textnormal{swf}}_{}^{A}(Y_{\text{in}},\mathfrak{s}_\text{in},A_{\text{in}}, g_\text{in} ; S^1) \\
& \quad  \rightarrow
\underline{\textnormal{swf}}_{}^{A}(Y_\text{out},\mathfrak{s}_\text{out},A_{\text{out}}, g_\text{out}; S^1) \nonumber
\end{align}
in $\mathfrak{S} $. This will be called the type-A unfolded relative Bauer--Furuta invariant of $X$.
\end{defi}

Note that Lemma~\ref{morphism compatible} and compatibility of the dual maps ensure that $\{\widetilde{\psi}^{n}_{m_{0},m_{1}}\}$ are compatible with the direct systems.
When $\mathfrak{s}=\hat{\mathfrak{s}}|_{Y}$ is torsion, we can also define the normalized relative Bauer--Furuta invariant. In this torsion case, let us define the normalized Thom spectrum
\begin{align*}
\tilde{T}(X,\hat{\mathfrak{s}},\hat{o};S^{1}) := (T(X,\hat{\mathfrak{s}},A_{0},g,\hat{o};S^{1}),0,n(Y,{\mathfrak{s}},A_{0},g)),
\end{align*}
where $n(Y,{\mathfrak{s}},A_{0},g) $ is given by $\frac{1}{2} \left(\eta( \slashed{D}) - \dim_\C (\ker  \slashed{D}) + \frac{\eta_{\text{sign}}}{4}  \right) $ (see (21) of \cite{KLS1}).
 
\begin{defi} \label{def normalized BFA}
When $\mathfrak{s}=\hat{\mathfrak{s}}|_{Y}$ is torsion, the normalized type-A unfolded relative Bauer--Furuta invariant of $X$
\[
 \begin{split}
& \underline{\textnormal{BF}}^{A}(X,\hat{\mathfrak{s}},\hat{o},[\vec{\eta}];S^{1}) \colon  \\
& \Sigma^{-(V^{+}_X\oplus  V_{\text{in}})}\tilde{T}(X,\hat{\mathfrak{s}},\hat{o};S^{1})\wedge
\underline{\textnormal{SWF}}_{}^{A}
(Y_{\text{in}},\mathfrak{s}_{\text{in}}; S^1)  
\rightarrow
\underline{\textnormal{SWF}}_{}^{A}(Y_{\text{out}},\mathfrak{s}_{\text{out}};
S^1)
\end{split}
\]
is given by desuspending $\underline{\textnormal{bf}}^{A}(X,\hat{\mathfrak{s}},A_{0},g,\hat{o},[\vec{\eta}];S^{1}) $ 
by $n(Y,{\mathfrak{s}},A_{0},g) $.  
\end{defi}

We then define the type-R invariant by simply considering the dual of type-$A$ 
invariant of the adjoint cobordism $X^{\dagger} \colon -Y_{\text{out}}\rightarrow -Y_{\text{in}}$.
In particular, the dual of the morphism
\begin{align*}
&\widetilde{\psi}^{n}_{m_{1},m_{0}}(X^{\dagger}) \colon  \\
& \Sigma^{-(V^{+}_{X^\dagger}\oplus  V_{\text{in}}(X^{\dagger}))}T(X^{\dagger},\hat{\mathfrak{s}},A_{0},g,\hat{o};S^{1})\wedge
 I^{n,+}_{m_{1}}(-Y_{\text{out}})\rightarrow
I^{n,+}_{m_{0}}(-Y_{\text{in}}),
\end{align*}
gives a morphism
\begin{align*}
{\breve\psi}^{n}_{m_{0},m_{1}} \colon \Sigma^{-(V^{+}_{X}\oplus  V_{\text{out}})}T(X^{},\hat{\mathfrak{s}},A_{0},g,\hat{o};S^{1})\wedge
 I^{n,-}_{m_{0}}(Y_{\text{in}})\rightarrow
I^{n,-}_{m_{1}}(Y_{\text{out}}).
\end{align*}
Note that $ V_{\text{in}}(X^{\dagger})  := V_{\text{out}}$ is the cokernel of $\iota^* \colon H^{1}(X;\mathbb{R})\rightarrow H^{1}(Y_{\text{out}};\mathbb{R})$.  The morphism ${\breve\psi}^{n}_{m_{0},m_{1}} $ is defined for any positive integer $m_{1}$ with $m_{0}$
large relative to $m_{1}$ and $n$
large relative to $m_{0},m_{1}$. We can now give a definition in a similar fashion.

\begin{defi} \label{def BFtypeR}
For the cobordism $X \colon Y_{\text{in}} \rightarrow Y_{\text{out}}$, the type-R unfolded relative Bauer--Furuta invariant of $X$ is a morphism
\begin{align*}
&\underline{\textnormal{bf}}^{R}(X,\hat{\mathfrak{s}},A_{0},g,\hat{o},[\vec{\eta}];S^{1}) \colon \\
& \Sigma^{-(V^{+}_X\oplus  V_{\text{out}})}T(X,\hat{\mathfrak{s}},A_{0},g,\hat{o}; S^{1})\wedge \underline{\textnormal{swf}}_{}^{R}(Y_{\text{in}},\mathfrak{s}_\text{in},A_{\text{in}}, g_\text{in} ; S^1)  \\
& \quad \rightarrow \underline{\textnormal{swf}}_{}^{R}(Y_\text{out},\mathfrak{s}_\text{out},A_{\text{out}}, g_\text{out}
; S^1) 
\end{align*}
in $\mathfrak{S}^* $ given by the collection of morphisms $\{{\breve \psi}^{n}_{m_{0},m_{1}} \mid m_{1}\in \mathbb{N},\ m_{0}\gg m_{1}, n\gg m_{0},m_{1}\}$.
When $\mathfrak{s}=\hat{\mathfrak{s}}|_{Y}$ is torsion, one can also desuspend $\underline{\textnormal{bf}}^{R}(X,\hat{\mathfrak{s}},A_{0},g,\hat{o},[\vec{\eta}];S^{1}) $ 
by $n(Y,{\mathfrak{s}},A_{0},g) $ to obtain the normalized type-R unfolded relative Bauer--Furuta
invariant of $X$
\[
\begin{split}
&\underline{\textnormal{BF}}^{R}(X,\hat{\mathfrak{s}},\hat{o},[\vec{\eta}];S^{1}) \colon  \\
& \Sigma^{-(V^{+}_X\oplus  V_{\text{out}})}\tilde{T}(X,\hat{\mathfrak{s}},\hat{o};S^{1})\wedge
\underline{\textnormal{SWF}}_{}^{R}(Y_{\text{in}},\mathfrak{s}_{\text{in}}; S^1)\rightarrow \underline{\textnormal{SWF}}_{}^{R}(Y_{\text{out}},\mathfrak{s}_{\text{out}};
S^1).
\end{split}
\]
\end{defi}

\begin{rmk} 
One can also construct the maps ${\breve\psi}^{n}_{m_{0},m_{1}} $ directly by replacing $(-Y_{\text{in}} , Y_{\text{out}})$ with $(Y_{\text{out}} , -Y_{\text{in}})$ in the construction through out this section.
\end{rmk}

\subsection{Invariance of the relative invariants} 
\label{sec INVofBF}

In this subsection, we will show that the morphism $\underline{\textnormal{bf}}^{A} = \underline{\textnormal{bf}}^A(X, \hat{\frak s}, A_0,  g,  \hat{o}, [\vec{\eta}]; S^1)$ and $\underline{\textnormal{bf}}^{R} = \underline{\textnormal{bf}}^{R}(X, \hat{\mathfrak s}, A_0,  g, \hat{o}, [\vec{\eta}]; S^1)$ depends only on $A_0, g, \hat{o}, [\vec{\eta}]$.  We have to check that they are independent of the choices of 

\begin{enumerate}[(i)]
  
  \item
  cutting function $\bar{g}$, cutting value $\theta$,  harmonic 1-forms $\{ h_{j} \}_{j=1}^{b_1}$ representing generators of $\im (H^1(Y; \Z) \rightarrow H^1(Y; \R))$, 
  
  \item
  Riemannian metric $\hat{g}$, connection $\hat{A}_0$ on $X$ with $\hat{g}|_{Y} = g, \hat{A}_0|_{Y} = A_0$, 
  
  \item
  perturbation $f : Coul(Y) \rightarrow \R$.

\end{enumerate}

 Moreover when $c_1(\frak{s})$ is torsion, we will show that $\underline{\textnormal{BF}}^{A}(X, \hat{\frak{s}}, \hat{o}, [\vec{\eta}]; S^1)$ and $\underline{\textnormal{BF}}^{R}(X, \hat{\mathfrak s}, \hat{o}, [\vec{\eta}]; S^1)$ are independent of $A_0, g$ too. 

\vspace{2mm}

Choose two cutting functions $\bar{g}$, $\bar{g}'$, cutting values $\theta, \theta'$ and sets of harmonic $1$-forms $\{ h_j \}_{j=1}^{b_1}$, $\{ h_{j}' \}_{j=1}^{b_1}$ representing generators of 
\[
2\pi i \im (H^1(Y; \Z) \rightarrow H^1(Y;\R)).
\]
We get two inductive systems
\[
       \begin{split}
       & \underline{\textnormal{swf}}^{A}(Y, \{ h_j \}_j, \bar{g}, \theta) = (I_1 \rightarrow I_2 \rightarrow \cdots), \\
       & \underline{\textnormal{swf}}^{A}(Y, \{ h_j' \}_j, \bar{g}', \theta') = ( \tilde{I}_1 \rightarrow \tilde{I}_2 \rightarrow \cdots)
       \end{split}
\]
in ${\frak C}$. 
Here $I_m$, $\tilde{I}_m$ are the desuspension of the Conley indices 
\[
I_{S^1}( \varphi^n,  \inv (J_{m}^{n,+})  ),  \
I_{S^1}( \varphi^n,  \inv (\tilde{J}_{m}^{n, +})  )
\]
for $n \gg m$  by $V_{\lambda_n}^{0}$, and  $J_{m}^{n, +}$, $\tilde{J}_{m}^{n, +}$ are the bounded sets  in $Str(\tilde{R})$ defined by using $(\{ h_j \}_j, \bar{g}, \theta)$, $(\{ h_j' \}_j, \bar{g}', \theta')$.  

Choosing integers $m_j \ll \tilde{m}_j \ll m_{j+1}$,  we can assume that $\inv( J^{n,+}_{m_j})$ is an attractor in $\inv( \tilde{J}^{n, +}_{\tilde{m}_j} )$ and we have the attractor map
\[
        I_{S^1}( \inv(J_{m_j}^{n, +} ))  \rightarrow I_{S^1}( \inv(  \tilde{J}^{n,+}_{ \tilde{m}_j }  )  )
\]
which induces a morphism
\[
         I_{m_j} \rightarrow  \tilde{I}_{\tilde{m}_j}. 
\]
Similarly we have a morphism
\[
      \tilde{I}_{ \tilde{m}_j} \rightarrow I_{m_{j+1}}.
\]
These morphisms induce an isomorphism between $\underline{\textnormal{swf}}^{A}(Y, \{ h_j \}_j, \bar{g}, \theta)$ and $\underline{\textnormal{swf}}^{A}(Y, \{ h_j' \}_j, \bar{g}', \theta')$.   The isomorphism between 
\[
\underline{\textnormal{swf}}^{R}(Y, \{ h_j \}_j, \bar{g}, \theta)
\ \text{and} \ 
\underline{\textnormal{swf}}^{R}(Y, \{ h_j' \}_j, \bar{g}, \theta')
\]
is obtained similarly. 
The morphisms in (\ref{morphism 1}) inducing the relative invariants $\underline{\textnormal{bf}}^{A}, \underline{\textnormal{bf}}^{R}$ are compatible with the attractor maps and repeller maps as in stated in Lemma \ref{morphism compatible}. It means that $\underline{\textnormal{bf}}^{A}$, $\underline{\textnormal{bf}}^{R}$ are independent of the choices of $\{ h_ j \}_j$, $\bar{g}$, $\theta$ up to the canonical isomorphisms.

\vspace{2mm}

Choose connections $\hat{A}_0, \hat{A}_0'$ on $X$ with $\hat{A}_0|_{Y} = \hat{A}_0'|_{Y} = A_0$ and Riemannian metrics $\hat{g}, \hat{g}'$ on $X$ with $\hat{g}|_{Y} = \hat{g}'|_{Y} = g$.  Then the homotopies
\[
     \hat{A}_0(s) = (1-s) \hat{A}_0 + s \hat{A}_0', \
     \hat{g} (s) = (1 - s) \hat{g} + s \hat{g}'
\]
naturally induce the homotopy between the maps $v$, $v'$ defined in (\ref{eq map v}) associated with $(\hat{g}_0, \hat{A}_0), (\hat{g}', \hat{A}_0')$.   Hence $\underline{\textnormal{bf}}^{A}, \underline{\textnormal{bf}}^{R}$ are independent of $\hat{A}_0, \hat{g}$.   
\vspace{2mm}

Take sequences $\lambda_n, \lambda_n', \mu_n, \mu_n'$ with $-\lambda_n, -\lambda_n', \mu_n, \mu_n' \rightarrow \infty$.
Then we get objects 
\[ 
  I_{m_0}^{n, -}(-Y_{\text{in}}),  I_{m_1}^{n, +}(Y_{ \text{out} }),   \ 
   \tilde{I}_{m_0}^{n, -}(-Y_{\text{in}}), \tilde{I}_{m_1}^{n,+}(Y_{\text{out}}).
\]
We have canonical isomorphisms
\[
              I_{m_0}^{n, -}(-Y_{\text{in}}) \cong \tilde{I}_{m_0}^{n, -}( - Y_{\text{in}} ), \ 
              I_{m_1}^{n, +}(Y_{\text{out}}) \cong \tilde{I}_{m_1}^{n, +}(Y_{\text{out}})
\]
for  $n$ large relative to $m_0$, $m_1$.   The morphisms $\psi_{m_0, m_1}^{n}$ are compatible with these isomorphisms as stated in Lemma \ref{morphism compatible}.  Therefore  $\underline{\textnormal{bf}}^{A}$, $\underline{\textnormal{bf}}^{R}$ are independent of $\lambda_n, \mu_n$ up to canonical isomorphisms.

\vspace{2mm}

Let us consider the invariance of $\underline{\textnormal{bf}}^{A}, \underline{\textnormal{bf}}^{R}$ with respect to the perturbation $f$. 
Take two perturbations $f_1, f_2 : Coul(Y) \rightarrow \R$. 
Then we obtain two inductive systems
\[
    \begin{split}
      & \underline{\textnormal{swf}}^{A}(Y, f_1) = ( I_1 \rightarrow I_2 \rightarrow \cdots ) \\
      & \underline{\textnormal{swf}}^{A}(Y, f_2) = ( \tilde{I}_1 \rightarrow \tilde{I}_2 \rightarrow \cdots )
    \end{split}
\]
in the category $\frak{C}$, which are isomorphic to each other.  Let us recall how to get the isomorphism briefly. (See  Section 6.3 of \cite{KLS1} for the details.)     The perturbations $f_1, f_2$ define the functionals ${\mathcal L}_1$, ${\mathcal L}_2$, which induce the flows
\[
     \varphi^n ( {\mathcal L}_1), \varphi^n( {\mathcal L}_2 ) : 
     V_{\lambda_n}^{\mu_n} \times \R  \rightarrow V_{\lambda_n}^{\mu_n}.
\]
The objects $I_m, \tilde{I}_m$ are the  desuspensions by $V_{\lambda_n}^{0}$ of the Conley indices 
\[
     I_{S^1}( \varphi^{n} ({\mathcal L}_1),  \inv( J_{m}^{n, +} ) ), \
    I_{S^1} ( \varphi^{n}(  {\mathcal L}_2),  \inv( \tilde{J}_{m}^{n, +} ) ).
\]
Choose integers $k_m, \tilde{k}_m$ with $0 \ll k_m \ll \tilde{k}_m \ll k_{m+1}$. Then  we have
\[
\begin{split}
           J_{k_m}^{+} 
         &  \subset p_{\mathcal H}^{-1}(  [-e_m+1, e_m-1]^{b_1} ) \cap Str(\tilde{R})  \\
      &     \subset p_{\mathcal H}^{-1}( [-e_m, e_m]^{b_1}) \cap Str(\tilde{R}) \\
      &     \subset \tilde{J}_{\tilde{k}_m}^{+}
  \end{split}
\]
for some large positive number $e_m$.  We have a map
\[
        \bar{i}^n_{m}   :  I_{S^1}( \varphi^{n}( {\mathcal L}_1  ), \inv( J^{n, +}_{k_m} )  )   \rightarrow 
                                I_{S^1}(  \varphi^{n} ( {\mathcal L}_2), \inv (  \tilde{J}_{\tilde{k}_m}^{n, +}  )  ),
\]
which induces the isomorphism between $\underline{\textnormal{swf}}^{A}(Y, f_1)$ and $\underline{\textnormal{swf}}^{A}(Y, f_2)$.  The map $\bar{i}_m^{n}$ is the composition $\rho_1 \circ \rho_2$ of 
\[
     \rho_1 : I_{S^1}( \varphi^n(  {\mathcal L}^0_{e_m} ), \inv( \tilde{J}^{n}_{ \tilde{k}_m }  )  )   \rightarrow 
                   I_{S^1}(\varphi^n (  {\mathcal L}_2 ), \inv ( \tilde{J}_{ \tilde{k}_m }^{n,+} )  ) 
\]
and
\[
          \rho_2 : I_{S^1}(  \varphi^n( {\mathcal L}_1 ), \inv( J_{k_m}^{n, +}  )  )  \rightarrow
                       I_{S^1}(  \varphi^n( \mathcal{L}_{e_m}^0  ), \inv( \tilde{J}_{\tilde{k}_m}^{n, +} )  ).
\]
Here ${\mathcal L}_{e_m}^0$ is a functional on $Coul(Y)$ such that 
\[
   \begin{split}
       &   {\mathcal L}^0_{e_m} = {\mathcal L}_1  \   \text{on $p_{\mathcal H}^{-1}( [-e_m+1, e_m - 1]^{b_1} )$,  }  \\
       &   {\mathcal L}^0_{e_m} = { \mathcal L }_2  \ \text{on $p_{\mathcal H}^{-1}( \R^{b_1} \setminus [ -e_m, e_m ]^{b_1}  )$  }.
   \end{split}
\]
The map $\rho_1$ is the homotopy equivalence induced by a  homotopy 
\[
\{ \varphi( \mathcal{L}^{s}_{e_m} ) \}_{0 \leq s \leq 1},
\]
where ${\mathcal L}_{e_m}^{s} = s {\mathcal L}_{1} + (1-s) {\mathcal L}_{e_m}^{0}$.   Note that 
\[
\inv(J_{k_m}^{n, +}, \varphi^n( {\mathcal L}_{e_m}^0 )) ( = \inv( J_{k_m}^{n, +}, \varphi^n( \mathcal{L}_1 )  ) )
\]
is  an attractor in $\inv(\tilde{J}_{\tilde{k}_m}^{n, +}, \varphi^n( {\mathcal L}_{e_m}^0  ))$. The map $\rho_2$ is the attractor map. 

Similarly the isomorphism between  $\underline{\textnormal{swf}}^{R}(Y, f_1)$ and $\underline{\textnormal{swf}}^{R}(Y, f_2)$ is induced by the composition of the repeller map and the homotopy equivalence induced by the homotopy of the flows. 

To prove the invariance of $\underline{\textnormal{bf}}^{A}, \underline{\textnormal{bf}}^{R}$ with respect to perturbation $f$, we need to show that the morphisms (\ref{morphism 1})  are compatible with the attractor maps, the repeller maps and the homotopy equivalence induced by the homotopy of the flows. 
The compatibility with the attractor maps and the repeller maps is already stated in Lemma \ref{morphism compatible}.  We will show the compatibility with the homotopy equivalence induced by the homotopy of the flows. 

Take perturbations $f_0, f_1 :  Coul(-Y_{in}) \coprod Coul(Y_{out}) \rightarrow \R$.    Let us consider the flow
\[
    \widetilde{\varphi}_n : V_{\lambda_n}^{\mu_n} \times [0, 1] \times \R \rightarrow V_{\lambda_n}^{\mu_n} \times [0, 1]
\]
on $V_{\lambda_n}^{\mu_n} \times [0, 1]$, induced by the homotopy 
\begin{equation}  \label{eq L^s}
\begin{split}
 &   \mathcal{L}_{Y_{\text{in}}, e_{m_0}}^{s} \coprod \mathcal{L}_{ Y_{\text{out}}, e_{m_1}  }^{s} :  \\
 &  \quad  Coul(-Y_{\text{in}}) \coprod Coul(Y_{\text{out}}) \rightarrow \R 
    \quad 
    (0 \leq s \leq 1). 
    \end{split}
\end{equation}
We also have the Seiberg-Witten map on $X$ induced by the homotopy:
\[
     Coul^{CC}(X) \times [0, 1] \rightarrow 
     L^2_{k-1}( i\Omega^+(X) \oplus S^-_X) \oplus V_{-\infty}^{\mu_n} \times [0, 1].
\]
Using the flow and the Seiberg-Witten map, for a small positive number $\epsilon > 0$, we define 
\[
    \widetilde{K}_1 = \widetilde{K}_1(n, m_0,  \epsilon), \  \widetilde{K}_2 =  \widetilde{K}_2(n, m_0,  \epsilon) 
     \subset
     B(V_{\lambda_n}^{\mu_n}, \widetilde{R}) \times [0, 1]
\]
as in (\ref{eq K_1 K_2}). As before we can show that $(\widetilde{K}_1, \widetilde{K}_2)$ is a pre-index pair and can find an index pair $(\widetilde{N}, \widetilde{L})$ such that
\[
     \widetilde{K}_1(n, m_0, \epsilon) \subset \widetilde{N}, \
     \widetilde{K}_2(n, m_0, \epsilon) \subset \widetilde{L}.
\]
For $s \in [0,1]$, put
\[
  \begin{split}
   & K_{1, s}(n, m_0, \epsilon) := \widetilde{K}_1(n, m_0, \epsilon) \cap  (V_{\lambda_n}^{\mu_n} \times \{ s \}), \\
   & K_{2, s}(n, m_0, \epsilon) := \widetilde{K}_2(n, m_0, \epsilon) \cap  (V_{\lambda_n}^{\mu_n} \times \{ s \}), \\
   &  N_s := \widetilde{N} \cap (V_{\lambda_n}^{\mu_n} \times \{ s \}), \\
   &  L_s := \widetilde{L} \cap (V_{\lambda_n}^{\mu_n} \times \{ s \}).
  \end{split}
\]
We get the map
\[
  \begin{split}
  &  v_{s} : B(W_{n, \beta}, R) / S(W_{n, \beta}, R)  \\
    &\quad \rightarrow 
     (B(U_n, \epsilon) / S(U_n, \epsilon)) \wedge ( K_{1, s}(n, m_0, \epsilon)/ K_{2, s}(n, m_0, \epsilon) ) \\
    & \quad \hookrightarrow 
     ( B(U_n, \epsilon) / S(U_n, \epsilon) ) \wedge  (N_s / L_s).
  \end{split}
\]
The maps $v_0, v_1$ induce morphisms
\[
  \begin{split}
   & \psi_0 : \Sigma^{ -(V_{X}^+ \oplus V_{\text{in}}) } T \rightarrow I_{m_0}^{n, -}(-Y_{\text{in}})_0 \wedge I_{m_1}^{n, +}(Y_{\text{out}})_0, \\
   & \psi_1 : \Sigma^{ -(V_{X}^+ \oplus V_{\text{in}}) } T \rightarrow I_{m_0}^{n, -}(-Y_{\text{in}})_1 \wedge I_{m_1}^{n, +}(Y_{\text{out}})_1
  \end{split}
\]
for $0 \ll m_0 \ll m_1 \ll n$ as before. 
We have to check that the following diagram is commutative: 
\begin{equation} \label{diagram psi_0 psi_1}
    \xymatrix{
       \Sigma^{ -(V_{X}^+ \oplus V_{\text{in}}) } T \ar[r]^(0.37){\psi_0} \ar[rd]_{\psi_1} &
                    I_{m_0}^{n, -}(-Y_{\text{in}})_0 \wedge I_{m_1}^{n, +}(Y_{\text{out}})_0 \ar[d]^{\cong} \\
       &  I_{m_0}^{n, -}(-Y_{\text{in}})_1 \wedge I_{m_1}^{n, +}(Y_{\text{out}})_1
    }
\end{equation}
Here $I_{m_0}^{n, -}(-Y_{\text{in}})_0 \wedge I_{m_1}^{n, +}(Y_{\text{out}})_0 \cong I_{m_0}^{n, -}(-Y_{\text{in}})_1 \wedge I_{m_1}^{n,+}(Y_{\text{out}})_1$ is the isomorphism induced by the homotopy (\ref{eq L^s}).  Consider the inclusion
\[
     i_s : N_s / L_s \hookrightarrow \widetilde{N} / \widetilde{L}
\]
for $s \in [0, 1]$.  By Theorem 6.7 and Corollary 6.8 of \cite{Salamon},   $i_s$ is a homotopy equivalence and  the following diagram is commutative up to homotopy:
\begin{equation} \label{diagram i_0 i_1}
       \xymatrix{
           N_0 / L_0 \ar[r]^{i_0} \ar[d]_{\cong}   &  \widetilde{N} / \widetilde{L} \\
           N_1 / L_1 \ar[ru]_{i_1}          &
        }
\end{equation}
Here $N_0 / L_0 \cong N_1 / L_1$ is the homotopy equivalence induced by the homotopy (\ref{eq L^s}).   
With the homotopy 
\[
     i_s \circ v_s : B(W_{n, \beta}, R) / S(W_{n, \beta}) \rightarrow 
                           (B(U_n, \epsilon) / S(U_n, \epsilon)) \wedge (\widetilde{N} / \widetilde{L})
\]
 between $i_0 \circ v_0$ and $i_1 \circ v_1$ and the commutativity of the diagram (\ref{diagram i_0 i_1}), we can see that the diagram (\ref{diagram psi_0 psi_1}) is commutative. 
The invariance of $\underline{\textnormal{bf}}^{A}, \underline{\textnormal{bf}}^{R}$ with respect to perturbation $f$ has been proved. 

\vspace{2mm}

Assume that $c_1({\frak s})$ is torsion.  We will prove that the normalized invariants $\underline{\textnormal{BF}}^{A}, \underline{\textnormal{BF}}^{R}$ are independent of Riemannian metric $g$ and base connection $A_0$ on $Y$.   Take Riemannian metrics $g, g'$  and connections $A_0, A_0'$ on $Y$.  Let us consider the homotopy
\[
      A_0 (s) = (1-s) A_0 + s A_0', \ 
      g(s) = (1-s)g + s g' \
      (s \in [0, 1]).
\]
Choose continuous families of Riemannian metrics $\hat{g}(s)$ and connections $\hat{A}_0(s)$ on $X$ with $\hat{g}(s)|_{Y} = g(s), \hat{A}_0(s)|_{Y} = A_0(s)$. 
Splitting the interval $[0, 1]$ into small intervals $[0, 1] = [0, t_1] \cup \cdots \cup [t_{N-1}, t_{N}]$,  the discussion is reduced to the case when $\lambda_n, \mu_n$ (for some fixed, large number $n$) are not eigenvalues of the Dirac operators $D_{s}$ on $Y$ associated to $g(s), A(s)$.  In this case,  the dimension of $W_{n, \beta}(s)$ is constant, where
\[
       W_{n, \beta}(s) :=  
       (\tilde{L}_s, p^{\mu_n}_{\infty})^{-1}(U_n \times V_{\lambda_n}^{\mu_n}(s)) \cap \mathcal{W}_{X, \beta}(s).         
\]
Then we can mimic  the discussion  about the invariance with respect to perturbation $f$ to get a homotopy $v_s$ between $v_0$ and $v_1$ which are the maps in (\ref{eq map v}) associated $(\hat{g}, \hat{A}_0), ( \hat{g}', \hat{A}_0')$.  Therefore the morphisms $\psi_{m_0, m_1}^{n}$ associated with $(\hat{g}_0, \hat{A}_0)$ and $(\hat{g}_1, \hat{A}_1)$ are the same. 
Note that the objects  $(V_{\lambda_n}^{0}(s) \oplus \C^{n(Y, g_s, A_s)})^+$ of $\frak{C}$  for $s = 0 ,1 $ are isomorphic to each other.  Taking the desuspension by $V_{\lambda_n}^{0}(s) \oplus \C^{n(Y, g_s, A_s)}$, we conclude  that $\underline{\textnormal{BF}}^{A}, \underline{\textnormal{BF}}^{R}$ are independent of $g, A_0$ up to canonical isomorphisms.


\section{The gluing theorem}
\subsection{Statement and setup of the gluing theorem} \label{sec gluingsetup}

In this section, let  $X_{0} \colon Y_{0}\rightarrow Y_{2}$ and $X_{1} \colon Y_{1}\rightarrow -Y_{2}$ be connected, oriented cobordisms with the
following properties:
\begin{itemize}
\item $Y_{2}$ is connected;
\item $Y_{0},Y_{1}$ may not be connected but $b_{1}(Y_{0})=b_{1}(Y_{1})=0$.
\end{itemize}
By gluing the two cobordisms along $Y_{2}$, we obtain a cobordism $X \colon Y_{0}\cup Y_{1}\rightarrow \emptyset$. As in Section~\ref{sec 4mfd}, we choose the following data when defining the relative Bauer--Furuta invariants:\begin{itemize}
\item A spin$^{c}$ structure $\hat{\mathfrak{s}}$ on $X$.
\item A Riemannian metric $\hat{g}$ on $X$, we require it equals the product metric near $Y_{i}$.
\item A base connection $\hat{A}^{0}$ on $X$;
\item A base point $\hat{o}\in Y_{2}$ and a based path data $[\vec{\eta}_{i}]$ on $X_{i}$ for $i=0,1$. The path from $\hat{o}$
to $Y_{2}$ is chosen to be the constant path. By patching $[\vec{\eta}_{1}]$ and $[\vec{\eta}_{2}]$ together in the obvious
way, we get a based path data $[\vec{\eta}]$ on $X$;
\item Denote the restriction of $\hat{\mathfrak{s}}$ (resp. $\hat{g}$ and $\hat{A}^{0}$ ) to $X_{i}$ by $\hat{\mathfrak{s}}_{i}$
(resp. $\hat{g}_{i}$ and $\hat{A}^{0}_{i}$) and the restriction to $Y_{j}$ by $\mathfrak{s}_{j}$ (resp. $g_{j}$ and $A^{0}_{j}$).
\end{itemize}
With the above data chosen, we obtain the invariants $$\underline{\textnormal{bf}}^{A}(X_{0},\hat{\mathfrak{s}}_{0},\hat{A}^{0}_{0},\hat{g}_{0},\hat{o},[\vec{\eta}_{0}];S^{1}),$$
$$\underline{\textnormal{bf}}^{R}(X_{1},\hat{\mathfrak{s}}_{1},\hat{A}^{1}_{0},\hat{g}_{1},\hat{o},[\vec{\eta}_{1}];S^{1}),$$
and $$\textnormal{BF}(X,  \hat{\mathfrak{s}}, \hat{o},[\vec{\eta}]).$$ For shorthand, we write them as $\underline{\textnormal{bf}}^{A}(X_{0})$,
$\underline{\textnormal{bf}}^{R}(X_{1})$ and $\textnormal{BF}(X)$ respectively throughout this section.


\begin{thm}[The Gluing Theorem]  \label{thm gluing BF inv}
With the above setup, if we assume further that 
\begin{equation}\label{homology condition}
\im(H^{1}(X_{0};\mathbb{R})\rightarrow H^{1}(Y_{2};\mathbb{R}))\subset \im(H^{1}(X_{1};\mathbb{R})\rightarrow H^{1}(Y_{2};\mathbb{R})),
\end{equation}
then, under the natural identification between domains and targets, we have
\begin{align*} \label{eq gluingformula}
\textnormal{BF}(X)|_{\pic(X,Y_{2})}=\pmb{\tilde{\epsilon}} (\underline{\textnormal{bf}}^{A}(X_{0}),\underline{\textnormal{bf}}^{R}(X_{1})),
\end{align*}
where $\pmb{\tilde{\epsilon}}(\cdot,\cdot)$ is the Spanier-Whitehead
duality operation defined in Section~\ref{section spanierwhitehead}.
\end{thm}
\begin{cor}
When the map $H^{1}(X_{0};\mathbb{R})\rightarrow H^{1}( Y_2 ; \mathbb{R})$ is trivial, we recover the full Bauer--Furuta invariant
\begin{align*} 
\textnormal{BF}(X)=\pmb{\tilde{\epsilon}}(\underline{\textnormal{bf}}^{A}(X_{0}), \underline{\textnormal{bf}}^{R}(X_{1})).
\end{align*}
\end{cor}
\begin{cor}
When $\mathfrak{s}_{2}$ is torsion and (\ref{homology condition}) is satisfied, we also have analogous result for the normalized version
$$
\textnormal{BF}(X)|_{ \pic(X,Y_2)}=\pmb{\tilde{\epsilon}} (\underline{\textnormal{BF}}^{A}(X_{0}),\underline{\textnormal{BF}}^{R}(X_{1})).
$$
\end{cor}

We begin  by setting up some notations. Let $\iota_{i} \colon Y_{2}\rightarrow X_{i}$ be the inclusion map. We pick a set of loops 
\[
\{\alpha^{0}_{1},\cdots ,\alpha^{0}_{b^{0}_{1,\alpha}}\}, \{\alpha^{1}_{1},\cdots ,\alpha^{1}_{b^{1}_{1,\alpha}}\},
\{\beta_{1},\cdots ,\beta_{b_{1,\beta}}\}
\]
with the following properties:
\begin{itemize}
\item For $i=0,1$, the set $\{\alpha^{i}_{1},\cdots ,\alpha^{i}_{b^{i}_{1,\alpha}}\}$ is contained in the interior of $X_{i}$
and represents a basis of cokernel of the induced map 
\[
(\iota_{i})_* \colon H_{1}(Y_{2};\mathbb{R})\rightarrow H_{1}(X_{i};\mathbb{R}).
\]
\item $\{\beta_{1},\cdots ,\beta_{b_{1,\beta}}\}\subset Y_{2}$ represents a basis for a subspace complementary to the kernel
of $(\iota_{0})_* \colon H_{1}(Y_{2};\mathbb{R})\rightarrow H_{1}(X_{0};\mathbb{R}).$
\end{itemize}
Under the assumption (\ref{homology condition}), the above properties further imply the following two properties:
\begin{itemize}
\item The set 
\begin{equation}\label{eq: decomposition of basis}
    \{\alpha^{0}_{1},\cdots ,\alpha^{0}_{b^{0}_{1,\alpha}}\}\cup \{\alpha^{1}_{1},\cdots ,\alpha^{1}_{b^{1}_{1,\alpha}}\}\cup
\{\beta_{1},\cdots ,\beta_{b_{1,\beta}}\}\end{equation} represent a basis of $H_1 (X;\mathbb{R})$;
\item $\{\alpha^{0}_{1},\cdots ,\alpha^{0}_{b^{0}_{1,\alpha}}\}\cup \{\alpha^{1}_{1},\cdots ,\alpha^{1}_{b^{1}_{1,\alpha}}\}$
represent a basis of $H_{1}(X,Y_{2};\mathbb{R})$.
\end{itemize}

As before, we use     $\mathcal{G}^{h, \hat{o}}_{X_i}$     to denote the group of harmonic gauge transformations
$u$ on $X_{i}$ such that $u(\hat{o})=1$ and $u^{-1}du\in i\Omega^{1}_{CC}(X_{i})$, and let  $\mathcal{G}^{h,\hat{o}}_{X_{i},\partial X_{i}}$   be the subgroup of  $\mathcal{G}^{h, \hat{o}}_{X_i}$ corresponding to $\ker ( H^1(X_i; \Z) \rightarrow H^1(\partial X_i; \Z) )$  . 
Note that $\mathcal{G}^{h,\hat{o}}_{X_{i},\partial X_{i}} \cong H^{1}(X_{i},Y_{2};\mathbb{Z})$.

As in Section~\ref{sec 4mfd}, 
for $i=0,1$, we consider the bundles $$\mathcal{W}_{X_{i}}=Coul^{CC}(X_{i})/\mathcal{G}^{h,\hat{o}}_{X_{i},\partial X_{i}},$$
over $\pic(X_{i},\partial X_{i})$ and the subbundle 
 $$\mathcal{W}_{X_{0},\beta}:=\{x\in \mathcal{W}_{X_{0}}\mid \hat{p}_{\beta}(x)=0\},$$
 where the projection $\hat{p}_{\beta} \colon Coul^{CC}(X_0) \rightarrow \mathbb{R}^{b_{1,\beta}}$ is given by 
\[
    \hat{p}_{\beta}(\hat{a},\hat{\phi})
      = (-i\int_{\beta_{1}}\mathbf{t}\hat{a}, \, \ldots  \, ,-i\int_{\beta_{b_{1,\beta}}}\mathbf{t}\hat{a}).
\]
\begin{rmk}Recall that $\underline{\textnormal{bf}}^{R}(X_{1})$ is constructed from the bundle $\mathcal{W}_{X_{1}}$, while $\underline{\textnormal{bf}}^{A}(X_{0})$ is constructed from a smaller bundle $\mathcal{W}_{X_{0},\beta}$. As a consequence, the decomposition (\ref{eq: decomposition of basis}) is essential in proving the pairing theorem involving the A-invariant for $X_{0}$ and R-invariant for $X_{1}$. Since existence of this asymmetric decomposition is implied by Condition (\ref{homology condition}), one can not switch $X_{0}$ and $X_{1}$ in this condition (without changing types of the relative invariants). 
\end{rmk}

 We have a basic boundedness result for glued trajectories:

\begin{pro}\label{bounded for gluing} There exists a universal constant $R_{3}$ with the following significance: For 
 any tuple $(\tilde{x}_{0},\tilde{x}_{1},\gamma_{0},\gamma_{1},\gamma_{2} , T)$ satisfying the following conditions
\begin{itemize}
\item $(\tilde{x}_{0},\tilde{x}_{1})\in \mathcal{W}_{X_{0},\beta}\times \mathcal{W}_{X_{1}}$ satisfies $\widetilde{SW}(\tilde{x}_{j})=0$;
\item $\gamma_{i} \colon (-\infty,0]\rightarrow Coul(Y_{i})$ $(i=0,1)$ and $\gamma_{2}\colon[-T,T]\rightarrow Coul(Y_{2})$ are finite
type Seiberg-Witten trajectories;
\item $\tilde{r}_{0}(\tilde{x}_{0})=\gamma_{0}(0),\ \tilde{r}_{2}(\tilde{x}_{0})=\gamma_{2}(-T),\ \tilde{r}_{2}(\tilde{x}_{1})=\gamma_{2}(T)$
and $\tilde{r}_{1}(\tilde{x}_{1})=\gamma_{1}(0)$, where $\tilde{r}_{j}$ denotes the twisted restriction map to $Coul(Y_{j})$;
\end{itemize}
one has $\|\tilde{x}_{i}\|_{F}\leq R_{3}$ for $i=0,1$ and $\gamma_{j}\subset Str_{Y_{j}}(R_{3})$ for $j=0,1,2$.
\end{pro}

\begin{proof} Suppose there exists a sequence not satisfying such uniform bounds.
We also assume that $T\rightarrow +\infty$ as the case when $T$ is uniformly bounded is trivial. From the condition $\hat{p}_{\beta}(x_0)=0 $, the norm of $\gamma_{0}$
and the norm $\|\tilde{x}_{0}\|_{F}$ is controlled by Theorem~\ref{boundedness for X-trajectory}. Notice
that the solutions converge to a broken trajectory on the $Y_{2}$-neck, which is contained in $Str_{Y_{2}}(R)$
for some universal constant $R$ by \cite[Theorem 3.2]{KLS1}. As in the construction, of $\swf(Y_2)$, we consider a sequence of bounded subset $\{ J^+_m (Y_2) \}$ of $Str_{Y_{2}}(R)$ (cf. \cite[Definition~5.3]{KLS1}).
Since $\|\tilde{x}_{0}\|_{F}$ is uniformly bounded, $\tilde{r}_{2}(\tilde{x}_{0})$ is contained in $J_{m}^{+}(Y_2)$
for some fixed $m$. From the fact that $J_{m}^{+}(Y_2)$ is an attractor with respect to the Seiberg--Witten flow,
we see that the whole broken trajectory is contained in $J^{+}_{m}(Y_2)$. In particular, $\tilde{r}_{2}(\tilde{x}_{1})$ also belongs to  $J_{m}^{+}(Y_2)$. We then apply  Theorem~\ref{boundedness for X-trajectory} again on $X_1$ to control $\|\tilde{x}_{1}\|_{F}$
and the norm of $\gamma_{1}$.
\end{proof}

Following Section~\ref{sec bfconstuct}, we will start to consider finite-dimensional approximation of the Seiberg--Witten map on both $X_0$ and $X_1$. Let us fix  an increasing sequence
of positive real numbers $\{\mu_{n}\}$ such that $\mu_n \rightarrow \infty$. 
For $i=0,1,2$, let $V^i_n \subset Coul(Y_i)$ be the span of eigenspaces with respect to $(*d,\slashed{D})$ with eigenvalues in the interval $[-\mu_n , \mu_n] $.
For $i=0,1$, we choose appropriate finite-dimensional subspaces $U^{i}_{n}\subset L^{2}_{k-1/2}(i\Omega^{+}_{2}(X_{i})\oplus \Gamma(S^{-}_{X_{i}}))$. The preimages of $U^i_n \times V^i_n \times V^2_n$ under $(\tilde{L} , p^{\mu_n}_{-\infty} \circ \tilde{r}) $ give rise to finite-dimensional subbundles $W^{0}_{n,\beta}\subset \mathcal{W}_{X_{0},\beta}$ and
$W^{1}_{n}\subset \mathcal{W}_{X_{1}}$.

We now state the boundedness result for approximated solutions. 
\begin{pro}\label{bounded for gluing, approximated}
For any $R>0$ and $L \ge 0 $ and any bounded subsets $S_{i}$ of $Coul(Y_{i})$ $(i=0,1,2)$, there exist constants $\epsilon,N,\bar{T}>0$
with the following significance: For any tuple $(\tilde{x}_{0},\tilde{x}_{1},\gamma_{0},\gamma_{1},\gamma_{2},n,T,T')$ satisfying
\begin{itemize}
\item $n >N $, $ T' > \bar{T}, $ and $T \le L$,
\item $(\tilde{x}_{0},\tilde{x}_{1})\in B(W^{0}_{n,\beta},R)\times B(W^{1}_{n},R)$ such that $\|\widetilde{SW}_{n}(\tilde{x}_{j})\|_{L^{2}_{k-1/2}}<\epsilon$
$(j=0,1)$,\item $\gamma_{i} \colon (-T',0]\rightarrow V^{i}_{n}\cap S_{i}$ $(i=0,1)$ and $\gamma_{2} \colon [-T,T]\rightarrow V^{2}_{n}\cap S_{2}$
are finite type approximated Seiberg-Witten trajectories,
\item $p^{\mu_n}_{-\infty} \circ\tilde{r}_{0}(\tilde{x}_{0})=\gamma_{0}(0),\ p^{\mu_n}_{-\infty} \circ\tilde{r}_{2}(\tilde{x}_{0})=\gamma_{2}(-T),\
 p^{\mu_n}_{-\infty} \circ\tilde{r}_{2}(\tilde{x}_{1})=\gamma_{2}(T)$
and $p^{\mu_n}_{-\infty} \circ\tilde{r}_{1}(\tilde{x}_{1})=\gamma_{1}(0)$,\end{itemize}
one has the following estimate
\begin{itemize}
\item $\|\tilde{x}_{i}\|_{F}\leq R_{3}+1$ for $i=0,1$;
\item $\gamma_{2}\subset Str_{Y_{2}}(R_{3}+1)$;
\item $\gamma_{i}|_{[-(T'- \bar{T}),0]}\subset  B_{Y_{i}}(R_{3}+1)$ for $i=0,1$.
\end{itemize}
Here $R_{3}$ is the constant from Proposition~\ref{bounded for gluing}.
\end{pro}

\begin{proof} The proof is analogous to that of Proposition~\ref{type A boundedness} and Proposition~\ref{boundedness on cylinder for approximated solutions} where one applies Proposition~\ref{bounded for gluing} instead.
\end{proof}


Recall that, for $i=0,1$, the manifold $Y_i$ is a rational homology sphere and one can find a sufficiently large ball $B_{Y_i} (\tilde{R_i})$ in the Coulomb slice containing all finite type Seiberg-Witten trajectories (cf.\cite{Manolescu1}). On $Y_2$, an unbounded subset $Str_{Y_2}({\tilde{R_2}})$ contains all finite type Seiberg-Witten trajectories when $\tilde{R_2}$ is sufficiently large. With a choice of cutting functions, we obtain an increasing
sequence of bounded sets $\{ J_{m_{}}^{+}(Y_{2}) \} $ contained in $Str_{Y_{2}}({\tilde{R_2}}) $.
Note that we can identify  $J^{n,-}_{m}(-Y_{2})=J^{n,+}_{m}(Y_{2})$.

Throughout the rest of the section, we will fix the following parameters in order of dependency carefully.
\begin{enumerate}[(i)]
\item Pick $\hat{R}_{0}> R_3$ such that any finite type $X_{0}$-trajectories  $(x,\gamma)$ with $x\in \mathcal{W}_{X_{0},\beta}$ satisfies $\|x\|_{F}\leq
\hat{R}_{0}$ (cf.  Theorem~\ref{boundedness for X-trajectory}).

\item Pick $\tilde{R}_{0},\tilde{R}_{2}> R_3+2$ such that $\tilde{r}(B(\mathcal{W}_{X_{0},\beta},\hat{R}_{0}))\subset B_{Y_{0}}(\tilde{R_{0}})\times Str_{Y_{2}}(\tilde{R}_{2})$ and  also $B_{Y_{0}}(\tilde{R}_{0}-1)\times Str_{Y_{2}}(\tilde{R}_{2}-1)$ contains all finite type Seiberg-Witten trajectories.

\item Choose a positive integer $m$ such that 
\[
\tilde{r}_{2}(B(\mathcal{W}_{X_{0},\beta},\hat{R}_{0}))\subset J^{+}_{m-1}(Y_{2}).
\]

\item Pick $\hat{R}_{1}>R_3+1$ such that any finite type $X_{1}$-trajectory $(x,\gamma)$ with $\tilde{r}_{2}(x)\in  J^{+}_{m}(Y_{2})$, one has $\|x_{}\|_{F}<\hat{R}_{1}$.

\item Choose a positive number $\tilde{R}_{1}$ such that $\tilde{r}_{2}(B(\mathcal{W}_{X_{1}},\hat{R}_{1}))\subset B_{Y_{1}}(\tilde{R}_{1})$
and $B_{Y_{1}}(\tilde{R}_{1}-1)$ contains all finite type Seiberg-Witten trajectory on $Y_{1}$.

\setcounter{choicegluing}{\value{enumi}}
\end{enumerate}

\subsection{Deformation of the duality pairing}
\label{sec deform1st}

In this section, we will focus on describing  the right hand side of the gluing theorem 
\[
\pmb{\tilde{\epsilon}} (\underline{\textnormal{bf}}^{A}(X_{0}),\underline{\textnormal{bf}}^{R}(X_{1}))
\]
and its deformation.   
As in Section~\ref{sec bfconstuct}, we write down the following subsets in order to define 
$\underline{\textnormal{bf}}^{A}(X_{0}) $ and $\underline{\textnormal{bf}}^{R}(X_{1})$:
\begin{align*}
K_{0} =&p^{\mu_{n}}_{-\infty}\circ\tilde{r}(\widetilde{SW}_{n}^{-1}(B(U^{0}_{n},\epsilon))\cap B(W^{0}_{n,\beta},\hat{R}_{0})) , \\
S_{0} =&p^{\mu_{n}}_{-\infty}\circ\tilde{r}(\widetilde{SW}_{n}^{-1}(B(U^{0}_{n},\epsilon))\cap S(W^{0}_{n,\beta},\hat{R}_{0})), \\
K_{1}=&p^{\mu_{n}}_{-\infty}\circ\tilde{r}(\widetilde{SW}_{n}^{-1}(B(U^{1}_{n},\epsilon))\cap B(W^{1}_{n},\hat{R}_{1}))\cap
(V^{1}_{n}\times
J^{n,-}_{m}(-Y_{2})),\\
S_{1}=& \{ p^{\mu_{n}}_{-\infty}\circ\tilde{r}(\widetilde{SW}_{n}^{-1}(B(U^{1}_{n},\epsilon))\cap S(W^{1}_{n},\hat{R}_{1}))
            \cap (V^{1}_{n}\times J^{n,-}_{m}(-Y_{2})) \}\\
        &  \cup \{ K_{1} \cap (V^{1}_{n}\times \partial J^{n,-}_{m}(Y_{2})) \}.
\end{align*}
Note that some of the subsets are simpler because $b_1( Y_0) = b_1(Y_1) = 0$.
The parameters $(\hat{R}_{0},\hat{R}_{1},\tilde{R}_{0},\tilde{R}_{1},\tilde{R}_{2},m)$ are selected earlier.
Subsequently, we will also fix a large number $L_{0}$ with the following property and then proceed to $n$ and $\epsilon$.
\begin{enumerate}[(i)]
\setcounter{enumi}{\value{choicegluing}}
\item Choose a positive number $L_0$ such that, for any large $n$ and small $\epsilon$, one has
\begin{enumerate}
\item  $(K_{0},S_{0})$ and $(K_{1},S_{1})$ are $L_{0}$-tame pre-index pairs. This follows from Proposition~\ref{prop 4dimpreindex} red by applying it to $X_0$ and $X_1$;
\item The pair $(K^{0}_{},S^{0}_{})$, as defined below
\begin{align*}
K^{0}&=\{(y_{0},y_{1})\mid (y_{0},y)\times (y_{1},y)\in K_{0} \times K_{1} \text{ for some }y\}
\\
S^{0}&=\{(y_{0},y_{1})\mid (y_{0},y)\times (y_{1},y)\in S_{0}\times K_{1}\cup K_{0}\times S_{1}  \text{ for some }y\},
\end{align*}
is an $L_{0}$-tame pre-index pair for $B(V^{0}_{n},\tilde{R}_{0})\times B(V^{1}_{n},\tilde{R}_{1})$. This follows from Proposition
\ref{bounded for gluing, approximated} with $L=0$.
\item \label{item (vi)c} Pick a slightly smaller closed subset $J'_m\subset \operatorname{int}(J_{m}^{+}(Y_{2}))$ such that for any approximated trajectory
$ \gamma \colon [-L_{0},L_{0}]\rightarrow  B(V^{0}_{n},\tilde{R}_{0})\times B(V^{1}_{n},\tilde{R}_{1})\times J^{n,+}_{m}(Y_{2})$,
one has $\gamma(0)\in B(V^{0}_{n},\tilde{R}_{0}-1)\times B(V^{1}_{n},\tilde{R}_{1}-1)\times J'_m$ (cf. \cite[Lemma 5.5]{KLS1}).

\item $L_0 > 4T_m (j) $ where $T_m(j)$ is the constant which appeared in Lemma~\ref{lem J 2T int} applying to the manifold $Y_j$. 
\end{enumerate}

\item Finally, we pick a large positive integer $n$ and a small positive real number $\epsilon$
so that
\begin{enumerate}
\item The above assertions for $L_{0}$  hold;
\item Proposition~\ref{bounded for gluing, approximated} holds for $L=3L_{0}$, $R=\max(\hat{R}_{0},\hat{R}_{1})$, $S_{0}=B_{Y_{0}}(\tilde{R}_{0})$ , $S_{1}=B_{Y_{1}}(\tilde{R}_{1})$  and $S_{2}=J^{+}_{m}(Y_{2})$.
\end{enumerate}
\end{enumerate}

With all the above parameters fixed, we have canonical maps to Conley indices
\begin{align*}
\iota_{0} \colon K_{0}/S_{0}&\rightarrow I(B(V^{0}_{n},\tilde{R}_{0}))\wedge I(J^{n,+}_{m}(Y_{2})), \\
\iota_{1} \colon K_{1}/S_{1}&\rightarrow I(B(V^{1}_{n},\tilde{R}_{1}))\wedge I(J^{n,-}_{m}(-Y_{2})).
\end{align*}
For simplicity, we will write $A_{j}=B(V^{j}_{n},\tilde{R}_{j})$ and $A'_{j}=B(V^{j}_{n},\tilde{R}_{j}-1)$ for $j=0,1$.
We also let $A_2$ denote
$J^{n,+}_{m}(Y_{2})$ and let $A'_{2}$ be  a closed subset satisfying
$$
(Str_{Y_{2}}(\tilde{R}_{2}-1)\cap J^{n,+}_{m-1}(Y_{2}))\cup (J'_m \cap V^2_n)\subset \operatorname{int}(A'_{2})\subset A'_{2}\subset \operatorname{int}(A_{2}).
$$
By our choice of $L_0$ and Proposition~\ref{prop mfdnhbdswf}, there exists a manifold isolating block $\tilde{N}_{j}$
satisfying
 \begin{equation}\label{isolating block contains invariant set}
 A^{[ -L_{0},L_{0} ]}_{j}\subset \operatorname{int}(\tilde{N}_{j})\subset \tilde{N}_{j}\subset A'_{j}.
 \end{equation}
Let $\varphi^{j}$ be the approximated Seiberg--Witten flow on $A_j$.
Denote by $\tilde{N}^{-}_{j}$ (resp. $\tilde{N}^{+}_{j}$) be the submanifold of $\partial \tilde{N}_{j}$ where $\varphi^{j}$
points outward (resp. inward).

By the choice of $L_0$ and Lemma~\ref{flow map from tame index pair} and Theorem~\ref{from pre-index to index refined}, we can express the smash product of canonical maps 
\begin{align*}
\iota_{0}\wedge \iota_{1} \colon K_{0}/S_{0}\wedge K_{1}/S_{1}\rightarrow \tilde{N}_{0}/\tilde{N}^{-}_{0}\wedge \tilde{N}_{2}/\tilde{N}^{-}_{2}\wedge \tilde{N}_{1}/\tilde{N}^{-}_{1} \wedge
\tilde{N}_{2}/\tilde{N}^{+}_{2}
\end{align*}
as a map sending $(y_{0},y_{2},y_{1}, y'_2)$ to $$(\varphi^{0}_{}(y_{0},3L_{0}),\varphi^{2}_{}(y_{2},3L_{0}),\varphi^{1}_{}(y_{1},3L_0),\varphi^{2}_{}(y_{2},-3L_{0}))$$
when the following conditions are all satisfied
\begin{equation}\label{condition on Y01}
\varphi^{j}(y_{j},[0,3L_{0}])\subset A_{j} \text{ and } \varphi^{j}(y_{j},[ L_{0},  3L_{0}])\subset \tilde{N}_{j}\setminus
\tilde{N}^{-}_{j} \text{ for }j=0,1;
\end{equation}
\begin{equation}\label{condition1 on Y2}
\varphi^{2}(y_{2},[0,3L_{0}])\subset A_{2} \text{ and } \varphi^{2}(y'_{2},[-3L_{0},0])\subset A_{2};
\end{equation}
\begin{equation}\label{condition2 on Y2}
\varphi^{2}(y_{2},[ L_{0},  3L_{0}])\subset \tilde{N}_{2}\setminus \tilde{N}^{-}_{2}\text{ and }\varphi^{2}(y'_{2},[-   L_{0}, -3L_{0}])\subset
\tilde{N}_{2}\setminus \tilde{N}^{+}_{2}.
\end{equation}
Otherwise, it  will be sent to the base point. We will first try to simplify some of the above conditions. For brevity, we sometimes omit mentioning a part of map sending to the basepoint.

\begin{lem}\label{almost invariant set}
There exists a positive constant $\bar{\epsilon}_{0}$ such that one can find
a closed subset $B_0 \subset \operatorname{int}(\tilde{N}_{2})$ with the following property: For any $(y_2 , y'_2) $ satisfying (\ref{condition1 on Y2}) and 
$$
\|\varphi^{j}(y_{2},3L_{0})-\varphi(y'_{2},-3L_{0})\|\leq 5\bar{\epsilon}_0,
$$
one has
\begin{equation*}
\varphi^{2}(y_{2},[  L_{0},  3L_{0}])\subset B_0\text{ and }\varphi^{2}(y'_{2},[- L_{0}, -3L_{0}])\subset B_0.
\end{equation*}
In particular, $(y_2 ,y'_2 )$ will satisfy (\ref{condition2 on Y2}).
\end{lem}
\begin{proof}
From (\ref{isolating block contains invariant set}), we see that one can choose $B_0 = A^{[  -L_{0},L_{0} ]}_{2}$ if we consider the case $\bar{\epsilon}_0=0$. For positive $\bar{\epsilon}_0 $, we pick $B_0 $ to be a slightly larger closed subset containing $ A^{[ -L_{0},  L_{0}]}_{2} $ and then apply  a continuity argument.
\end{proof}

To deform our maps, we also consider a variation of the above lemma.

\begin{lem} \label{lem deformL01st}
There exists a positive constant $\bar{\epsilon}_{1}$ such that for any $L\in
[0,L_{0}]$ and any $(y_{0},y_{2},y'_{2},y_{1})\in K_{0}\times K_{1}$ satisfying  (\ref{condition on Y01}) and 
\begin{align}
\varphi^{2}(y_{2},[0,3L])\subset A_{2} \text{ and } \varphi^{2}(y'_{2},[-3L,0])\subset A_{2}, \label{condition4 on Y2} \\
\|\varphi^{2}(y_{2},3L)-\varphi^{2}(y'_{2},-3L)\|\leq \bar{\epsilon}_1, \label{condition5 on Y2}
\end{align}
we have
$$
\varphi^{2}(y_{2},[0,3L])\subset A'_{2} \text{ and } \varphi^{2}(y'_{2},[-3L,0])\subset A'_{2}.
$$
\end{lem}
\begin{proof}
We first consider the case $\bar{\epsilon}_1=0$. Then, by Proposition \ref{bounded for gluing, approximated} and our choice of $(n,\epsilon) $, we have
$\varphi^{2}(y_{2},[0,6L]) \subset Str_{Y_{2}}(\tilde{R}_{2}-1)$. From our choice, we also have $y_{2} \in
J^{n,+}_{m-1}(Y_2) \subset V^{2}_{n}$. Since $J^{n,+}_{m-1}(Y_2)$ is an attractor in $J^{n,+}_{m}(Y_2)$, we have $\varphi^{2}(y_{2},[0,6L]) \subset
J^{n,+}_{m-1}(Y_2)$. Thus
\[
   \varphi^{2}(y_{2},[0,6L]) \subset (Str_{Y_{2}}(\tilde{R}_{2}-1)\cap J^{n,+}_{m-1}(Y_{2})) 
                                          \subset \operatorname{int}(A'_{2}).
\]
The general case follows from a continuity argument.
\end{proof}

We will also consider the following subsets enlarging $(K^0 , S^0)$
\begin{align*}
K^{\bar{\epsilon}}:=\{(y_{0},y_{1})\mid &(y_{0},y_{2})\times (y_{1},y_{2}')\in K_{0}\times K_{1}\\
&\text{ for some $y_2, y_2'$ with } \|y_{2}-y'_{2}\| \leq    \bar{\epsilon}\},
\\
S^{\bar{\epsilon}}:=\{(y_{0},y_{1})\mid &(y_{0},y_{2})\times (y_{1},y_{2}')\in (S_{0}\times K_{1})\cup (K_{0}\times S_{1})\\
&\text{ for some $y_2, y_2'$ with } \|y_{2}-y'_{2}\|\leq \bar{\epsilon}\}.
\end{align*}
Since $(K^0 , S^0)$ is an $L_0$-tame pre-index pair, the following can be obtained by a continuity argument.

\begin{lem} \label{lem epsilonL0tame} There exists a positive constant $\bar{\epsilon}_{2}$ such that the pair $(K^{\bar{\epsilon}},S^{\bar{\epsilon}})$ is an $L_{0}$-tame pre-index pair for any $0 \le \bar{\epsilon} \le \bar{\epsilon}_2 $.
\end{lem}
 
For a vector space or a vector bundle, denote by $B^{+}(V,R)$ the sphere $B(V,R)/S(V,R)$.
 Recall that the Spanier-Whitehead duality map (see Section~\ref{section dualswf})
$$
\pmb{\tilde{\epsilon}} \colon \tilde{N}_{2}/\tilde{N}^{-}_{2}\wedge \tilde{N}_{2}/\tilde{N}^{+}_{2}\rightarrow B^{+}(V^{2}_{n},\bar{\epsilon})
$$
can be given by 
\begin{align*}\pmb{\tilde{\epsilon}}(y_{2},y'_{2}) = 
\begin{cases} \eta_{-}(y_{2})-\eta_{+}(y'_{2}) &\text{if } \|\eta_{-}(y_{2})-\eta_{+}(y'_{2})\|\leq \bar{\epsilon},  \\  * &\text{otherwise}.\end{cases}
\end{align*}
Here we pick $\bar{\epsilon} < \min\{\bar{\epsilon}_0 , \bar{\epsilon}_1 , \bar{\epsilon}_2\} $ and  $\eta_{\pm} \colon\tilde{N}_{2}\rightarrow \tilde{N}_{2}$ are  homotopy equivalences which are identity on $B_0 \subset \operatorname{int}(\tilde{N}_{2})$ 
and satisfy $\|\eta_{\pm}(x)-x\|\leq 2\bar{\epsilon}$ for any $x$.  Here $B_0$ is the closed set in Lemma \ref{almost invariant set}.

Consequently we can write down the composition of $\iota_{0}\wedge \iota_{1} $ and $\pmb{\epsilon}  $ as a map
$$\pmb{\tilde{\epsilon}}(\iota_{0},\iota_{1}) \colon K_{0}/S_{0}\wedge
K_{1}/S_{1}\rightarrow \tilde{N}_{0}/\tilde{N}^{-}_{0}\wedge \tilde{N}_{1}/\tilde{N}^{-}_{1}\wedge B^{+}(V^{2}_{n},\bar{\epsilon})$$
 given by
 \begin{equation}\label{gluing with long neck}
 \begin{split}
& (y_{0},y_{2},y_1 , y'_{2})\mapsto  \\
& \qquad (\varphi^{0}(y_{0},3L_{0}),\varphi^{1}(y_{1},3L_{0}),\varphi^{2}(y_{2},3L_{0})-\varphi^{2}(y'_{2},-3L_{0}))
 \end{split}
  \end{equation}
  if  (\ref{condition on Y01}) and (\ref{condition1 on Y2}) and
\begin{equation}\label{condition3 on Y2}
\|\varphi^{2}(y_{2},3L_{0})-\varphi^{2}(y'_{2},-3L_{0})\|\leq \bar{\epsilon}
\end{equation}
are satisfied. This follows from Lemma~\ref{almost invariant set} and our choice of $\bar{\epsilon} $ and $\eta_{\pm}$.

We now begin to deform the map $\pmb{\tilde{\epsilon}}(\iota_{0},\iota_{1}) $. 

\subsubsection*{Step 1} We will deform the map so that $L_0 $ in the last term of (\ref{gluing with long neck}) goes from $L_0$ to $0$. To achieve this, we consider a family of maps 
 \begin{align*}
 K_{0}/S_{0}\wedge
K_{1}/S_{1}&\rightarrow \tilde{N}_{0}/\tilde{N}^{-}_{0}\wedge \tilde{N}_{1}/\tilde{N}^{-}_{1}\wedge B^{+}(V^{2}_{n},\bar{\epsilon})\\
 (y_{0},y_{2},y_1 ,y'_{2})&\mapsto (\varphi^{0}(y_{0},3L_{0}),\varphi^{1}(y_{1},3L_{0}),\varphi^{0}(y_{2},3L)-\varphi^{2}(y'_{2},-3L))
  \end{align*}
if   (\ref{condition on Y01}) together with 
the conditions \begin{equation*}
\begin{split} \varphi^{2}(y_{2},[0,3L])\subset A_{2} \text{, } \varphi^{2}(y'_{2},[-3L,0])\subset A_{2} \\
\|\varphi^{2}(y_{2},3L)-\varphi^{2}(y'_{2},-3L)\|\leq \bar{\epsilon}\end{split}
\end{equation*}
are all satisfied. Lemma~\ref{lem deformL01st} guarantees that this is a continuous family. Thus, $\pmb{\tilde{\epsilon}}(\iota_{0},\iota_{1})$ is homotopic to the map ${\pmb{\tilde{\epsilon}}}_0(\iota_{0},\iota_{1})$  at $L=0$, which is given by
\begin{equation}\label{map with short neck}
 (y_{0},y_{2},y'_{2},y_{1})\mapsto (\varphi^{0}(y_{0},3L_{0}),\varphi^{1}(y_{1},3L_{0}),y_{2}-y'_{2})
  \end{equation}
if (\ref{condition on Y01}) and $\|y_{2}-y'_{2}\|\leq \bar{\epsilon}$
are satisfied.

\subsubsection*{Step 2}
By Lemma~\ref{lem epsilonL0tame}, $(K^{\bar{\epsilon}},S^{\bar{\epsilon}})$
is an $L_{0}$-tame pre-index pair and we have a canonical map
$$\iota^{\bar{\epsilon}} \colon K^{\bar{\epsilon}}/S^{\bar{\epsilon}}\rightarrow I(B(V^{0}_{n},\tilde{R}_{0}))\wedge I(B(V^{1}_{n},\tilde{R}_{1})).$$
It is not hard to check that the map given by
\begin{align}
K_{0}/S_{0}\wedge K_{1}/S_1 &\rightarrow I(B(V^{0}_{n},\tilde{R}_{0}))\wedge I(B(V^{1}_{n},\tilde{R}_{1}))\wedge (V^{2}_{n})^{+}.
\label{eq mappreindex2nd} \\
(y_{0},y_{2}, y_1, y'_{2}) &\mapsto  \begin{cases}(\iota^{\bar{\epsilon}}(y_{0},y_{1}),y_{2}-y'_{2}) &\text{if } \|y_{2}-y'_{2}\|\leq \bar{\epsilon},
 \\ * & \text{otherwise} \end{cases} \nonumber
\end{align}
is well-defined and continuous. 
From Lemma~\ref{flow map from tame index pair}, we can represent $\iota^{\bar{\epsilon}} $ by a map
\begin{align*}
K^{\bar{\epsilon}}/S^{\bar{\epsilon}} &\rightarrow \tilde{N}_{0}/\tilde{N}^{-}_{0}\wedge \tilde{N}_{1}/\tilde{N}^{-}_{1} \\
(y_0 , y_1) &\mapsto (\varphi^{0}(y_{0},3L_{0}),\varphi^{1}(y_{1},3L_{0})),
\end{align*}
if (\ref{condition on Y01}) is satisfied. Consequently, we are able to replace the first two components of the map (\ref{map with short neck}) by  $\iota^{\bar{\epsilon}} $  and  the map ${\pmb{\tilde{\epsilon}}}_0(\iota_{1},\iota_{2}) $  by (\ref{eq mappreindex2nd}).

Finally, recall that the relative Bauer--Furuta invariant $\underline{\textnormal{bf}}^{A}(X_{0}) $ is obtained from composition of a map
\begin{align*}
B^{+}(W^{0}_{n,\beta},\hat{R}_{0})  \rightarrow  B^{+}(U^{0}_{n},\epsilon)\wedge K_0 / S_0\end{align*}
and the canonical map $\iota_0 $. The invariant $\underline{\textnormal{bf}}^{R}(X_{1}) $ is obtained similarly.
Then, $\pmb{\tilde{\epsilon}}(\underline{\textnormal{bf}}^{A}(X_{0}), \underline{\textnormal{bf}}^{R}(X_{1}))$
is given by applying Spanier--Whitehead dual map to their smash product.
From previous steps, we can conclude the following result

\begin{pro}\label{deformed pairing}
The morphism
$\pmb{\tilde{\epsilon}}(\underline{\textnormal{bf}}^{A}(X_{0}), \underline{\textnormal{bf}}^{R}(X_{1}))$ can be represented by suitable desuspension of the map
\begin{align*}
& B^{+}(W^{0}_{n,\beta},\hat{R}_{0}) \wedge B^{+}(W^{1}_{n},\hat{R}_{1}) \\
&\qquad \rightarrow  B^{+}(U^{0}_{n},\epsilon)\wedge B^{+}(U^{1}_{n},\epsilon)\wedge
B^{+}(V^{2}_{n},\bar{\epsilon})\wedge I^{n}(-Y_{0})\wedge I^{n}(-Y_{1})
\end{align*}
defined by
\[(\tilde{x}_{0},\tilde{x}_{1})\mapsto (\widetilde{SW}_{n}(\tilde{x}_{0}),\widetilde{SW}_{n}(\tilde{x}_{1}),r_{2}(\tilde{x}_{0})-r_{2}(\tilde{x}_{1}),{\iota}^{\bar{\epsilon}}(r_{0}(\tilde{x}_{0}),r_{1}(\tilde{x}_{1}))
\]
if $\|\widetilde{SW}_{n}(\tilde{x}_{i})\|\leq \epsilon$ and $\|r_{2}(\tilde{x}_{0})-r_{2}(\tilde{x}_{1})\|\leq \bar{\epsilon}$
and sending $(\tilde{x}_{0},\tilde{x}_{1})$ to the base point otherwise.
Here $I^{n}(-Y_{i})$ denotes $I(B(V^{i}_{n},\tilde{R}_{i}))$ for $i=0,1$.
\end{pro}

\subsection{Stably c-homotopic pairs } \label{subsec stablyc}
In this subsection, we recall notions of stably c-homotopy and SWC triples which were originally introduced by Manolescu \cite{Manolescu2}.
These provide a convenient framework when deforming stable homotopy maps coming from construction of Bauer--Furuta invariants.
 Although most
of the definitions are covered in \cite{Manolescu2}, we rephrase  them in a slightly more general setting which is easier
to apply in our situation. We also give some details for completeness and concreteness.

Let $p_{i} \colon E_{i}\rightarrow B\ (i=1,2)$ be Hilbert bundles over some compact space $B$. We denote by $\|\cdot\|_{i}$ the
fiber-direction norm of ${E}_{i}$. Let $\bar{E}_{1}$ be the fiberwise completion of ${E}_{1}$ using a weaker
norm, which we denote by $|\cdot |_{1}$. 
We also assume that for any bounded sequence $\{x_{n}\}$ in $E_{1}$, there exist $x_{\infty}\in
E_{1}$ such that after passing to a subsequence, we have\begin{itemize}
\item  $\{x_{n}\}$ converge to $x_{\infty}$ weakly in $E_{1}$. 
\item  $\{x_{n}\}$ converge to $x_{\infty}$ strongly in $\bar{E}_{1}$.
\end{itemize}

\begin{defi}\label{admissible pairs}
A pair $l,c \colon E_{1}\rightarrow E_{2}$ of bounded continuous  bundle maps is called an admissible pair if it
satisfies the following conditions:
\begin{itemize}
\item $l$ is a fiberwise linear map;
\item $c$ extends to a continuous map $\bar{c} \colon \bar{E}_{1}\rightarrow E_{2}$.
\end{itemize}

\end{defi}

At this point, we will focus on the context of the gluing theorem as in Section~\ref{sec gluingsetup}. 
Let $V=Coul(Y_{0})\times Coul(Y_{1})$ with $b_{1}(Y_{0}) = b_1( Y_{1})=0$.
As before, denote by $V^{\mu}_{\lambda}$ the subspace spanned by the eigenvectors
of $(*d,\slashed{D})$ with eigenvalue in $(\lambda , \mu]$ and denote the projection $V\rightarrow V_{\lambda}^{\mu}$ by $p_{\lambda}^{\mu}$.
Motivated by the Seiberg-Witten map on 4-manifolds with boundary, we give the following definition.

\begin{defi}\label{SWC triple}
Let $(l,c)$ be an admissible pair from $E_{1}$ to $E_{2}$ and let $r \colon E_{1}\rightarrow V$ be a continuous map which is linear
on each fiber. We call $(l,c,r)$ an \emph{SWC-triple} (which stands for Seiberg--Witten--Conley) if the following conditions are satisfied:
\begin{enumerate}

\item The map $l\oplus (p^{0}_{-\infty}\circ r) \colon E_{1}\rightarrow E_{2}\oplus V^{0}_{-\infty}$ is fiberwise Fredholm.
\item There exists $M'>0$ such that for any pair of $x\in E_{1}$ satisfying $(l+c)(x)=0$ and a half-trajectory of finite type $\gamma \colon (-\infty,0]\rightarrow
V$ with $r(x)=\gamma(0)$, we have $\|x\|_{1}<M'$ and $\|\gamma(t)\|<M'$ for any $t\geq 0$.
\end{enumerate}
\end{defi}

Two SWC-triples $(l_{i},c_{i},r_{i})\ (i=0,1)$ (with the same domain and targets) are called \emph{c-homotopic} if
there is a homotopy between them through a continuous family of SWC triples with a uniform constant $M'$. 

Two SWC-triples  $(l_{i},c_{i},r_{i})\ (i=0,1)$ (with possibly different domain and targets) are called \emph{stably c-homotopic}
if there exist Hilbert bundles $E_{3},E_{4}$ such that $((l_{1}\oplus \operatorname{id}_{E_{3}},c_{1}\oplus
0_{E_{3}}),r_{1}\oplus 0_{E_{3}})$ is c-homotopic to  $((l_{2}\oplus \operatorname{id}_{E_{4}},c_{2}\oplus 0_{E_{4}}),r_{2}\oplus
0_{E_{4}})$ .

For any SWC triple $(l,c,r)$, we can define a relative Bauer--Furuta type invariant as a pointed stable homotopy class
$$
   BF(l,c,r)\in \{\Sigma^{n \mathbb{C} }\mathbf{T}(\text{ind}(l,p^{0}_{-\infty}\circ r)), \operatorname{SWF}(-Y_{0})\wedge \operatorname{SWF}(-Y_{1})\}_{},
$$
where 
\[
n=n(Y_{0},\mathfrak{s}_{Y_{0}},g_{Y_{0}})+n(Y_{1},\mathfrak{s}_{Y_{1}},g_{Y_{1}})
\]
by  so called ``SWC-construction'' analogous to the  construction in Section~\ref{sec 4mfd} described below.

Let us pick a trivialization $E_2 \cong F_2 \times B$ with a projection $q \colon E_{2}\rightarrow F_{2}$, an increasing sequence of real numbers $\lambda_{n}\rightarrow \infty$ and a sequence of increasing finite-dimensional subspaces $\{F^{n}_{2}\}$ of $F_{2}$ such that the projections $p_{n} \colon F_{2}\rightarrow F^{n}_{2}$ converge pointwisely
 to the identity map and $q^{-1}(F^{n}_{2})\times V^{\lambda_{n}}_{-\lambda_{n}}\subset E_{2}\times V^{\lambda_{n}}_{-\infty}$
is transverse to the image of $(l,p^{\lambda_{n}}_{-\infty}\circ r)$ on each fiber. Let $E_{1}^{n}$ be the preimage $(l,p^{\lambda_{n}}_{-\infty}\circ
r)^{-1}(q^{-1}(F^{n}_{2})\times V^{\lambda_{n}}_{-\lambda_{n}})$ which is a finite rank subbundle. 

Consider an approximated map
$$f_{n}=p_{n}\circ q\circ (l+c) \colon E^{n}_{1}\rightarrow F^{n}_{2}.$$
From the definition of the SWC triple, 
one can deduce the following result in the same manner as the construction of relative invariants for Seiberg--Witten maps:  
for any $R',R\gg 0$ satisfying $r(B(E_{1},R))\subset
B(V,R')$, there exist $N,\epsilon_{0}$ such that for any $n\geq N$ and $\epsilon<\epsilon_{0}$, the pair of subsets $$(p^{\lambda_{n}}_{-\infty}\circ
r(f_{n}^{-1}(B(F^{n}_{2},\epsilon))\cap B(E_{1},R)),p^{\lambda_{n}}_{-\infty}\circ r(f_{n}^{-1}(B(F^{n}_{2},\epsilon))\cap
S(E_{1},R))))$$
is a pre-index pair in the isolating neighborhood $B(V^{\lambda_{n}}_{-\lambda_{n}},R')$.

From this, we can find an index pair $(N,L)$ containing the above pre-index pair, which allows us to define an induced map
$$B(E^{n}_{1},R)/S(E^{n}_{1},R)\rightarrow B(F^{n}_{2},\epsilon)/S(F^{n}_{2},\epsilon)\wedge N/L.$$ After desuspension, we
obtain a stable map 
$$
       h \colon \Sigma^{n\C}  \mathbf{T}(\operatorname{ind}(l,p^{0}_{-\infty}\circ r))   \rightarrow  
       \operatorname{SWF}(-Y_{0})  \wedge \operatorname{SWF}(-Y_{1}).
$$ By standard homotopy arguments, the stable homotopy class $[h]$ does not depend on auxiliary choices.
As a result, we define the stable homotopy class $[h]$ to be the relative invariant $BF(l,c,r)$ for this SWC triple. 

It is straightforward to prove that two stably c-homotopic SWC triples give the same stable homotopy class.
This is the main point of introducing SWC construction.
We end with a very useful lemma which is  a generalization of   Observation 1  in \cite[Section 4.1]{Manolescu2}
and allows us to move between maps and conditions on the domain.

\begin{lem}\label{moving map to  domain2}
Let $(l,c)$ be an admissible pair from $E_{1}$ to $E_{2}$ and let $r \colon E_{1}\rightarrow V$ be a continuous map which is linear
on each fiber. Suppose that we have a surjective bundle map $g \colon E_{1}\rightarrow E_{3}$. Then the triple $(l\oplus g,c\oplus 0_{E_{e}},r)$
is an SWC triple if and only if the triple $(l|_{\ker g},c|_{\ker g},r|_{\ker g})$ is an SWC triple. In the case that such two triples are SWC triples, they are stably c-homotopic to each other.\end{lem}

\subsection{Deformation of the Seiberg-Witten map}
\label{sec deform2nd}

Throughout this section, we will denote by  
$$G =H^{1}(X,Y_{2};\mathbb{Z})\cong H^{1}(X_{0},Y_{2};\mathbb{R})\times H^{1}(X_{1},Y_{2};\mathbb{Z})$$
and fix such an identification . Furthermore, we introduce the notation 
\[
\begin{split}
& \Omega^{1}(X_{1},Y_{1},\alpha^{1}):=\\
& \qquad 
\left\{
\hat{a}\in \Omega^{1}(X_{1}) \left| d^{*}\mathbf{t}_{Y_{1}}(\hat{a})=0,\ \int_{Y^{j}_{1}}(*\hat{a})=0,\
\int_{\alpha^{1}_{k}} \hat{a}=0,\ \forall j,k
\right.
\right\}
\end{split}
\]
and define $\Omega^{1}(X_{0},Y_{0},\alpha^{0}\cup
\beta)$ and $\Omega^{1}(X,Y_{0}\cup Y_{1},\alpha^{0}\cup\alpha^{1}\cup  \beta)$ similarly. Let us also denote all the relevant Hilbert spaces

\begin{itemize}
\item $
V_{X_{0}}:=L^{2}_{k+1/2}(i\Omega^{1}(X_{0},Y_{0},\alpha^{0}\cup \beta)\oplus \Gamma(S^{+}_{X_{0}}));
$
\item $
V_{X_{1}}:=L^{2}_{k+1/2}(i\Omega^{1}(X_{1},Y_{1},\alpha^{1})\oplus \Gamma(S^{+}_{X_{1}}));
$
\item $V_{X}:=L^{2}_{k+1/2}(i\Omega^{1}(X,Y_{0}\cup Y_{1},\alpha^{0}\cup\alpha^{1}\cup  \beta)\oplus \Gamma(S^{+}_{X}));
$
\item $V:=Coul(Y_{0})\times Coul(Y_{1});$ 
\item $U_{X_{i}}:=L^{2}_{k-1/2}({i\Omega^{0}(X_{i})\oplus i\Omega^{2}_{+}(X_{i})\oplus \Gamma(S^{-}_{X_{i}})})\text{ for }i=0,1;$
\item $U_{X}:=L^{2}_{k-1/2}({i\Omega_{0}^{0}(X)\oplus i\Omega^{2}_{+}(X)\oplus \Gamma(S^{-}_{X})})$;
\item $H^{1}(X_{\bullet},Y_{2};\mathbb{R})$, where $X_{\bullet}$ stands for $X_{0},X_{1}$ or $X$. 
\end{itemize}
Here $\Omega^{0}_{0}(X)$ denotes the space of functions on $X$ which integrate to zero. 
Note that $G$ acts on all
these spaces  as following:
\begin{itemize}
\item On differential forms, the action is trivial.
\item On spinors, we use the identification 
\begin{equation}\label{harmonic gauge transformation}
G\cong \mathcal{G}^{h,\hat{o}}_{X,Y_{2}},
\end{equation}
where $\mathcal{G}^{h,o}_{X,Y_{2}}$ denotes the group of harmonic gauge transformations  $u$ on $X$ such that  $u^{-1}du\in
i\Omega_{CC}^{1}(X)$ and $u|_{Y_{2}}=
e^f$ with $f(\hat{o}) = 0 $. The action is by gauge transformation. Note that we will use the restriction
of $\mathcal{G}^{h,\hat{o}}_{X,Y_{2}}$ on $X_0 $ and $X_1$ instead of the harmonic gauge transformation satisfying boundary condition on $X_{0}$ or $X_1$.
\item On the homology $H^{1}(X_{\bullet},Y_{2};\mathbb{R})$, the action is given by \emph{negative} translation.
\end{itemize}

We consider Hilbert bundles 
\begin{align*}
\tilde{V}_{X}&=(V_{X}\times H^{1}(X,Y_{2};\mathbb{R}))/G, \\
\tilde{U}_{X}&=(U_{X}\times H^{1}(X,Y_{2};\mathbb{R}))/G
\end{align*}
over $\pic(X,Y_{2})$ and a pair of maps 
$$
l_{X},c_{X} \colon V_{X}\times H^{1}(X,Y_{2};\mathbb{R})\rightarrow L^{2}_{k-1/2}(i\Omega^{2}_{+}(X)\oplus \Gamma(S^{-}_{X}))\times H^{1}(X,Y_{2};\mathbb{R})
$$
given by 
$$
l_{X}(\hat{a},\phi,h):=(d^{+}\hat{a},\slashed{D}^{+}_{\hat{A}_{0}+i\tau(h)}\phi,h),\ c_{X}:=(F^{+}_{\hat{A}^{t}_{0}}-\rho^{-1}(\phi\phi^{*})_{0},\rho(\hat{a})\phi,h),
$$
where $\tau (h) $
is the unique harmonic 1-form $u$ on $X$ representing $h$ such that $\mathbf{t}_{Y_{2}}(\tau(h))$ is exact
and $\tau(h)\in
i\Omega_{CC}^{1}(X)$. It is straightforward to see that $l_{X}$ and $c_{X}$ are equivariant under the $G$-action. Thus, we can take the quotient and
obtain bundle maps 
$$
(d^* \oplus \tilde{l}_{X}), (0 \oplus \tilde{c}_{X}) \colon \tilde{V}_{X}\rightarrow \tilde{U}_{X}.
$$

Observe that the double Coulomb condition on $V_X$ is simplified to just $d^* (\hat{a}) =0 $.   
It then follows that  $(\tilde{l}_X|_{\ker d^*},\tilde{c}_X|_{\ker d^*}, (\tilde{r}_{0},\tilde{r}_{1})|_{\ker d^*})$ is an SWC-triple and $\operatorname{BF}(X)|_{\pic(X,Y_{2})}$
is precisely obtained from the SWC-construction of this triple, where $ \tilde{r}_{i}:\tilde{V}_{X}\rightarrow Coul(Y_{i})$ denotes the twisted restriction map as in Section~\ref{sec 4mfd}.

The goal of this section is to deform $\operatorname{BF}(X)|_{\pic(X,Y_{2})}$ to the map $\pmb{\tilde{\epsilon}}(\underline{\textnormal{bf}}^{A}(X_{0}), \underline{\textnormal{bf}}^{R}(X_{1}))$  represented as in Proposition~\ref{deformed pairing}. 
There will be several steps.

\subsubsection*{Step 1} 
We move the gauge fixing condition $d^* = 0 $ to  stably c-homotopic maps. Since
$$d^{*} \colon i\Omega^{1}(X,Y_{0}\cup Y_{1},\alpha^{0}\cup \alpha^{1}\cup\beta)\rightarrow i\Omega_{0}^{0}(X)$$
is surjective, we directly apply Lemma~\ref{moving map to  domain2} and obtain the following:
\begin{lem}
The relative Bauer--Furuta invariant $\operatorname{BF}(X)|_{\pic(X,Y_{2})}$ is obtained by the SWC construction on the triple
$(d^* \oplus \tilde{l}_{X},0 \oplus \tilde{c}_{X},(\tilde{r}_{0},\tilde{r}_{1}))$.
\end{lem}

\subsubsection*{Step 2}
We begin to glue configurations on $X_0$ and $X_1$ to obtain configurations on $X$. 
Let us consider a Sobolev space of configurations on the boundary
$$V^{k-m}_{Y_{2}}:=L_{k-m}^{2}(i\Omega^{1}(Y_{2})\oplus i\Omega^{0}(Y_{2})\oplus \Gamma(S_{Y_{2}})).$$
for $0\leq m\leq k$.

For any 1-form $\hat{b}$ on $X$, we can combine the Levi--Civita connection on  $\Lambda^{*}T^{*}(X_{i})$ and the spin$^{c}$
connection $\hat{A}_{0}|_{X_{i}}+\hat{b}$ to obtain a connection on $\Lambda^{*}T^{*}(X_{i})\oplus S_{X_{i}}$. We use $\nabla^{\hat{b}}$
to denote the corresponding covariant derivative.
Consider a map 
\begin{align*}
D^{(m)} \colon V_{X_{0}}\times V_{X_{1}}&\times H^{1}(X,Y_{2};\mathbb{R}) \rightarrow  V^{k-m}_{Y_{2}}\times H^{1}(X,Y_{2};\mathbb{R}) \\
(x_{0},x_{1},h) &\mapsto  ((\nabla^{\tau(h)|_{X_{0}}}_{\vec{n}})^{m}x_{0})|_{Y_{2}}-((\nabla^{\tau(h)|_{X_{1}}}_{\vec{n}})^{m}x_{1})|_{Y_{2}},h),
\end{align*}
where $\vec{n}$ is the outward normal direction of $Y_{2}\subset X_{0}$. Here, we apply standard bundle isomorphisms
$
T^{*}(X_{i})|_{Y_{2}}\cong T^{*}Y_{2}\oplus \underline{\mathbb{R}}\text{ and }S^{+}_{X_{i}}|_{Y_{2}}\cong S_{Y_{2}}.
$

It is clear that the map $D^{(m)}$ is equivariant under the action of $G$.  As a result, we can take the quotient and obtain a map
$$ 
\tilde{D}^{(m)} \colon \tilde{V}_{X_{0},X_{1}}\rightarrow \tilde{V}^{k-m}_{Y_{2}},
$$
where we set
\begin{align*}
\tilde{V}_{X_{0},X_{1}} &:=(V_{X_{0}}\times V_{X_{1}}\times H^{1}(X,Y_{2};\mathbb{R}))/G \\
\tilde{V}^{k-m}_{Y_{2}} &:=(V^{k-m}_{Y_{2}}\times H^{1}(X,Y_{2};\mathbb{R}))/G.
\end{align*}
We state the gluing result for these spaces, which is a variation of \cite[Lemma~3]{Manolescu2}. The proof is only local near $Y_{2}$ and can be adapted without change.

\begin{lem}\label{gluing sobolev space}
The bundle map 
$$(\tilde{D}^{(k)},\cdots,\tilde{D}^{(0)}): \tilde{V}_{X_{0},X_{1}}\rightarrow \mathop{\oplus}^{k}_{m=0}\tilde{V}^{k-m}_{Y_{2}}$$
is fiberwise surjective and the kernel can be identified with the bundle $\tilde{V}_{X}$.
\end{lem}

Analogous to the maps $d^* \oplus l_{X}$ and $0 \oplus c_{X}$, we define the map
\[l_{X_{0},X_{1}} \colon V_{X_{0}}\times V_{X_{1}}\times H^{1}(X,Y_{2};\mathbb{R}) \rightarrow U_{X_{0}}\times U_{X_{1}}\times
H^{1}(X,Y_{2};\mathbb{R}) \label{linear part of SW map}\]
\[\begin{split}
&((\hat{a}_{0},\phi_{0}),(\hat{a}_{1},\phi_{1}),h) \\
 &\mapsto ((d^{*}\hat{a}_{0},d^{+}\hat{a}_{0},\slashed{D}^{+}_{(\hat{A}_{0}+i\tau(h))|_{X_{0}}}\phi_{0}),(d^{*}\hat{a}_{1},d^{+}\hat{a}_{1},\slashed{D}^{+}_{(\hat{A}_{0}+i\tau(h))|_{X_{1}}}\phi_{1}),h)
\end{split}
\]
and the map 
\[c_{X_{0},X_{1}} \colon V_{X_{0}}\times V_{X_{1}}\times H^{1}(X,Y_{2};\mathbb{R}) \rightarrow U_{X_{0}}\times U_{X_{1}}\times
H^{1}(X,Y_{2};\mathbb{R})\]
\[
\begin{split}
&((\hat{a}_{0},\phi_{0}),(\hat{a}_{1},\phi_{1}),h) \\
&\mapsto ((0,F^{+}_{\hat{A}^{t}_{0}}|_{X_{0}}-\rho^{-1}(\phi_{0}\phi_{0}^{*})_{0},\rho(\hat{a}_{0})\phi_{0}), \\
& \qquad  \qquad 
(0,F^{+}_{\hat{A}^{t}_{0}}|_{X_{1}}-\rho^{-1}(\phi_{1}\phi_{1}^{*})_{1},\rho(\hat{a}_{1})\phi_{1}),h).
\end{split}
\]
Then, by taking quotient, we get bundle maps  
$$\tilde{l}_{X_{0},X_{1}},\tilde{c}_{X_{0},X_{1}} \colon \tilde{V}_{X_{0},X_{1}}\rightarrow \tilde{U}_{X_{0},X_{1}},$$
where $\tilde{U}_{X_{0},X_{1}}:=(U_{X_{0}}\times U_{X_{1}}\times H^{1}(X,Y_{2};\mathbb{R}))/G$. By gluing of Sobolev spaces, the bundle  $\tilde{U}_{X}$
can be identified as a subbundle of  $\tilde{U}_{X_{0},X_{1}}$. Let $\operatorname{pj}$ be the orthogonal projection to this subbundle.
The following result is then a consequence of  Lemma~\ref{gluing sobolev space} and Lemma~\ref{moving map to  domain2}.

\begin{lem} \label{lem deformation1} 
The triple 
\begin{equation}\label{deformation 1}
((\operatorname{pj}\circ \tilde{l}_{X_{0},X_{1}},\tilde{D}^{(k)},\cdots, \tilde{D}^{(0)}),(\operatorname{pj}\circ \tilde{c}_{X_{0},X_{1}},0,\cdots,0),(\tilde{r}_{0},\tilde{r}_{1}))
\end{equation}
 is an SWC-triple and is stably c-homotopic to  $(d^* \oplus \tilde{l}_{X}, 0 \oplus \tilde{c}_{X},(\tilde{r}_{0},\tilde{r}_{1}))$.
\end{lem}

\subsubsection*{Step 3}
Next, we will glue the Sobolev spaces of the target.  Let us consider a map 
\begin{align*}
& E^{(m)} \colon  U_{X_{0}}\times U_{X_{1}}\times H^{1}(X,Y_{2};\mathbb{R}) \rightarrow  V^{k-1-m}_{Y_{2}}\times H^{1}(X,Y_{2};\mathbb{R}) \\
& E^{(m)}(y_{0},y_{1},h) = (((\nabla^{\tau(h)|_{X_{0}}}_{\vec{n}})^{m}y_{0})|_{Y_{2}}-((\nabla^{\tau(h)|_{X_{1}}}_{\vec{n}})^{m}y_{1})|_{Y_{2}},h),
\end{align*}
where we also apply standard bundle isomorphisms 
\[
\Lambda^{2}_{+}(X_{i})|_{Y_{2}}\cong T^{*}Y_{2},\ S^{-}_{X_{i}}|_{Y_{2}}\cong S_{Y_{2}}.
\]
By taking quotient with respect to the action of $G$, we obtain bundle maps 
$$
\tilde{E}^{(m)} \colon \tilde{U}_{X_{0},X_{1}}\rightarrow  \tilde{V}^{k-1-m}_{Y_{2}}.
$$

\begin{pro}\label{step 2}
The triple
\begin{equation}\label{deformation 2}
\begin{split}
((\operatorname{pj}\circ \tilde{l}_{X_{0},X_{1}},\ \tilde{E}^{(k-1)}\circ \tilde{l}_{X_{0},X_{1}},\ \cdots\ ,\tilde{E}^{(0)}\circ
\tilde{l}_{X_{0},X_{1}},\ \tilde{D}^{(0)}),\\
(\operatorname{pj}\circ \tilde{c}_{X_{0},X_{1}},\ \tilde{E}^{(k-1)}\circ \tilde{c}_{X_{0},X_{1}},\ \cdots\ ,\tilde{E}^{(0)}\circ
\tilde{c}_{X_{0},X_{1}},\ 0),(\tilde{r}_{0},\tilde{r}_{1}))
\end{split}
\end{equation}
is an SWC-triple and is c-homotopic to the triple (\ref{deformation 1}).
\end{pro}

\begin{proof}

We simply consider a linear $c$-homotopy between them as follows:
For $1\leq m\leq k$ and $0 \le t \le 1 $, define a map 
$$\tilde{D}^{(m)}_{t}=(1-t)\cdot \tilde{D}^{(m)}+t\cdot \tilde{E}^{(m-1)}\circ \tilde{l}_{X_{0},X_{1}} $$
and the following maps from $\tilde{V}_{X_{0},X_{1}}$ to $\tilde{U}_{X}\oplus \left( \mathop{\oplus}^{k}_{m=0}\tilde{V}^{k-m}_{Y_{2}}\right) $
\begin{align*}
l_{t} &:=(\operatorname{pj}\circ \tilde{l}_{X_{0},X_{1}},\ \tilde{D}_{t}^{(k)},\cdots, \tilde{D}_{t}^{(1)},\tilde{D}^{(0)}), \\
c_{t} &:=(\operatorname{pj}\circ \tilde{c}_{X_{0},X_{1}},\  t\cdot\tilde{E}^{(k-1)}\circ \tilde{c}_{X_{0},X_{1}},\ \cdots\
,t\cdot\tilde{E}^{(0)}\circ \tilde{c}_{X_{0},X_{1}},\ 0).
\end{align*}
This will give a $c$-homotopy as a result of the following lemma. 
\end{proof}

\begin{lem}\label{deformation 2 Fredholm}
For any $0 \le t \le 1$, the map 
$$
(l_{t},p^{0}_{-\infty}\circ(\tilde{r}_{0},\tilde{r}_{1})) \colon \tilde{V}_{X_{0},X_{1}}\rightarrow \tilde{U}_{X}\oplus (\mathop{\oplus}^{k}_{m=0}\tilde{V}^{k-m}_{Y_{2}})\oplus
V^{0}_{-\infty}(-Y_{0}\cup -Y_{1})
$$
is fiberwise Fredholm. Moreover, the zero set $(l_{t}+c_{t})^{-1}(0)\subset \tilde{V}_{X_{0},X_{1}}$ is independent of $t$ and can be described
as
\[
\begin{split}
& \{[(\hat{a},\phi,h)]\in \tilde{V}_{X}    \\
& \qquad 
\mid d^{*}\hat{a}=0\text{ and }(\hat{A}_{0}+i\tau(h)+\hat{a},\phi)
\text{ is a Seiberg-Witten solution}\}.
\end{split}
\]

\end{lem}
\begin{proof}
The key observation is that $E^{(m)}\circ l_{X_{1},X_{2}}- \tilde{D}^{(m+1)}$ contains at most $m$-th derivative in the normal
direction. Then, one can prove inductively that
$$(\tilde{D}_t^{(k)},\cdots,\tilde{D}_t^{(1)},\tilde{D}^{(0)})(x_{0},x_{1})=0 \implies (\tilde{D}^{(k)},\cdots,\tilde{D}^{(0)})(x_{0},x_{1})=0,
$$
so that the kernel of $l_t$ does not depend on $t$. Similarly, one can show that $(\tilde{D}_{t}^{(k)},\cdots,\tilde{D}_{t}^{(1)},\tilde{D}^{(0)})$ is fiberwise surjective for all $t$. Since $t=0$ is the map from Lemma~\ref{lem deformation1}, the map $(l_{t},p^{0}_{-\infty}\circ(\tilde{r}_{0},\tilde{r}_{1}))$ is fiberwise Fredholm for all $t$.

The second part was essentially proved in \cite[Section 4.11]{Manolescu2} using similar inductive argument.

\end{proof}


\subsubsection*{Step 4}
We now make the following identification:

\begin{lem}
The bundle map (over $\pic(X,Y_{2})$)
$$
(\operatorname{pj},\tilde{E}^{(k-1)}\cdots \tilde{E}^{(0)},\xi) \colon \tilde{U}_{X_{0},X_{1}}    \rightarrow 
       \tilde{U}_{X}\oplus
        (\mathop{\oplus}^{k-1}_{m=0}\tilde{V}^{k-1-m}_{Y_{2}})\oplus 
        \underline{\mathbb{R}}
$$
is an isomorphism. The map $\xi$ is given by $\xi (x_1 , x_2 , h) = \int_{X_{0}}f_{0} +\int_{X_{1}}f_{1}$, where $f_i$ is the 0-form component of $x_i$.
\end{lem}
\begin{proof}
This also follows from gluing result of Sobolev spaces \cite[Lemma 3]{Manolescu2}. The only difference here is that the 0-form component $\tilde{U}_{X}$ consists of functions which integrate to 0. From  the standard decomposition $\Omega_{}^{0}(X) = \Omega_{0}^{0}(X) \oplus \mathbb{R} $, we can see that the projection onto $\mathbb{R} $ is given by the map $\xi$.
\end{proof}

On the other hand, 
we decompose $\tilde{D}^{(0)} $ from the following decomposition of the Hilbert spaces:
\begin{equation}
V^{k}_{Y_{2}}  =  Coul(Y_{2})  \oplus H  \oplus  \mathbb{R}   \text{ with } 
H   =   L^{2}_{k}(i(d\Omega^{0}(Y_{2}) \oplus \Omega_{0}^{0}(Y_{2}))).
\end{equation}
We denote the corresponding components of $D^{(0)}$ (resp. $\tilde{D}^{(0)}$) by $D_{Y_{2}}$, $D_{H}$ and $D_{\mathbb{R}}$ (resp.
$\tilde{D}_{Y_{2}},\tilde{D}_{H}$ and $\tilde{D}_{\mathbb{R}}$).

We make an observation that  the SWC-triple (\ref{deformation 2}) in Proposition~\ref{step 2} arises from a composition
\begin{align*}
& \tilde{V}_{X_{0},X_{1}} \xrightarrow{} \tilde{U}_{X_{0},X_{1}}  \oplus Coul(Y_{2})\oplus H  \\ 
&\qquad \qquad \xrightarrow{}
\tilde{U}_{X}\oplus
(\mathop{\oplus}^{k-1}_{m=0}\tilde{V}^{k-1-m}_{Y_{2}})\oplus \underline{\mathbb{R}} \oplus Coul(Y_{2})\oplus H,
\end{align*}
where the first arrow is $(\tilde{l}_{X_{0},X_{1}} + \tilde{c}_{X_{0},X_{1}} , \tilde{D}_{Y_{2}},\tilde{D}_{H} )$ and the second arrow is the isomorphism
$(\operatorname{pj},\tilde{E}^{(k-1)}\cdots \tilde{E}^{(0)},\xi , id , id)$. The only thing we need to check is that $\tilde{D}_{\mathbb{R}}=\xi\circ \tilde{l}_{X_{0},X_{1}} $ on the 1-form component, which follows from the Green-Stokes formula
$$
\int_{Y_{2}}\mathbf{t}(*\hat{a}_{0})-\int_{Y_{2}}\mathbf{t}(*\hat{a}_{1})=\int_{X_{0}}d^*\hat{a}_{0}+\int_{X_{1}}d^*\hat{a}_{1}.
$$
Thus, we can conclude

\begin{lem}\label{lem SWCtripleDH} The SWC-triple (\ref{deformation 2}) can be identified with the triple
\begin{align}
((\tilde{l}_{X_{0},X_{1}},\tilde{D}_{Y_{2}},\tilde{D}_{H}),(\tilde{c}_{X_{0},X_{1}},0,0),(\tilde{r}_{0},\tilde{r}_{1})).
\end{align}
\end{lem}

\subsubsection*{Step 5}   

In this step, we focus on deforming the $\tilde{D}_{H}$-component which  corresponds to boundary conditions for gauge fixing. We sometimes omit spinors  from expressions in this step. 

For $\hat{a}_{j}\in i\Omega^{1}(X_{j})$, we have a Hodge decomposition $\mathbf{t}_{Y_{2}}(\hat{a}_{j})=a_{j}+b_{j}$ on $Y_2$
with $a_{j} \in  \ker d^{*}$  and $b_{j} \in \im d$. We also denote by $e_j := c_j -\tfrac{\int c_j d\text{vol}}{\text{vol}(Y_{2})} \in i\Omega^{0}_{0}(Y_{2})$, where ${\hat{a}_j}{|_{Y_2}} = \mathbf{t}_{Y_{2}}(\hat{a}_{j}) + c_j dt$. 
With this formulation, we see that $D_H (\hat{a}_0 , \hat{a}_1 ) = (b_0 - b_1 , e_0 - e_1) $.

Let us consider an isomorphism 
$$
\bar{d}:L^{2}_{k}(i\Omega_{0}^{0}(Y_{2}))\rightarrow L^{2}_{k}(id\Omega^{0}(Y_{2}))
$$
defined by $\bar{d}f:=\lambda^{-1}df$ for any $f\in i\Omega_{0}^{0}(Y_{2})$ with $d^{*}df=\lambda^{2}f$ with $\lambda>0$
using the spectral decomposition of $d^{*}d$. 
We let 
$$\bar{d}^{*}: L^{2}_{k}(id\Omega^{0}(Y_{2}))\rightarrow L^{2}_{k}(i\Omega_{0}^{0}(Y_{2}))$$
be its formal adjoint. Note that $\bar{d}^{*}$ can also be obtained directly by  $\bar{d}^{*}\alpha:=\lambda f$
for $\alpha = df$ satisfying $dd^{*}\alpha=\lambda^{2}\alpha$ with $\lambda>0$ and $\int_{Y_2} f = 0 $.
We then define a family of maps 
$$
D_{H,t} \colon V_{X_{0}}\times V_{X_{1}}\rightarrow H
$$
given by
$$
D_{H,t}(\hat{a}_{0},\hat{a}_{1}):=(b_{0}-b_{1},t\cdot \bar{d}^{*}(b_{0}+b_{1})+(1-t)\cdot (e_{0}-e_{1})).
$$

The main point here is to establish that the gauge fixing conditions 
$D_{H,t} =0$ are isomorphic and vary continuously. In particular, we will find a harmonic gauge transformation in the identity component to relate them. For a pair of coclosed 1-forms $(\hat{a}_{0},\hat{a}_{1}) \in \Omega^{1}(X_{0},Y_{0},\alpha^{0}\cup \beta) \times \Omega^{1}(X_{1},Y_{1},\alpha^{1})$ with $b_0 = b_1 $, finding such a transformation amounts to solving for a pair of functions $(f_0 , f_1 ) \in \Omega^0(X_0) \times \Omega^0 (X_1) $ such that
\begin{align*} 
& 2t\cdot \bar{d}^{*}d(f_{0}|_{Y_{2}})+ (1-t)(\partial_{\vec{n}}f_{0}|_{Y_{2}}-\partial_{\vec{n}}f_{1}|_{Y_{2}})  \\
&\qquad = 2t\cdot \bar{d}^{*}(b_{0})+(1-t) (e_{0}-e_{1})
\end{align*}                                                                                                             and also satisfies other gauge fixing conditions.
We have the following existence and uniqueness result.

\begin{lem}\label{laplace with mixed bundary condition}
Let $W\subset L^{2}_{k+3/2}(X_{0};\mathbb{R})\times L^{2}_{k+3/2}(X_{1};\mathbb{R})$ be the subspace 
containing all functions $(f_{0},f_{1})$ satisfying the following conditions:
\begin{enumerate}
\item $\Delta f_{i}=0$; \label{list laplace1}
\item $f_{i}(\hat{o})=0$; \label{list laplace2}
\item $f_{0}|_{Y_{2}}=f_{1}|_{Y_{2}}$; \label{list laplace3}
\item $f_{i}$ is a constant on each component of $Y_{i}$, $i=0,1$; \label{list laplace4}
\item $\partial_{\vec{n}}f_{i}$ integrates to zero on each component of $Y_{i}$, $i=0,1$. \label{list laplace5}
\end{enumerate}
Then the map $\rho_{t} \colon W\rightarrow L^{2}_{k}(\Omega_{0}^{0}(Y_{2}))$ defined by $$\rho_{t}(f_{0},f_{1})=2t\cdot \bar{d}^{*}d(f_{0}|_{Y_{2}})+(1-t)(\partial_{\vec{n}}f_{0}|_{Y_{2}}-\partial_{\vec{n}}f_{1}|_{Y_{2}})$$
is an isomorphism.
\end{lem}
\begin{proof}
We first show that $\rho_{t}$ is an isomorphism when $t=1$. For $\xi\in L^{2}_{k}(i\Omega_{0}^{0}(Y_{2}))$,
we want to find $f_{i}$ such that $f_{i}|_{Y_{2}}=\frac{\xi}{2}-\frac{\xi(\hat{o})}{2}$ and satisfies the other conditions.
The existence and uniqueness of such functions follow from the same argument as in the double Coulomb condition (cf. \cite[Proposition~2.2]{Khandhawit1}).


Since each $\rho_{t}$ corresponds to Laplace equation with mixed Dirichlet and Neumann boundary condition, it is Fredholm with index zero (from $t=1$). Thus, for $t<1$, we are left to show that $\rho_{t}$ is injective.  Suppose
$\rho_{t}(f_{0},f_{1})=0$. Then by Green's formula, we have 
\[\begin{split}
(1-t)(\int_{X_{0}}\langle df_{0},df_{0}\rangle+\int_{X_{1}}\langle df_{1},df_{1}\rangle)&=(1-t)\int_{Y_{2}}f_{0}(\partial_{\vec{n}}f_{0}-\partial_{\vec{n}}f_{1})\\&=-2t\int_{Y_{2}}f_{0}\cdot
(\bar{d}^{*}d(f_{0}|_{Y_{2}}))    
\end{split}
\]
The first expression is nonnegative but the expression 
$$\int_{Y_{2}}f_{0}
(\bar{d}^{*}d(f_{0}|_{Y_{2}})) = \int_{Y_2} (f_0)^2 - \frac{1}{vol Y_2}(\int_{Y_2} f_0)^2 $$ is also nonnegative
by Cauchy--Schwartz inequality.
Hence both $f_{0}$ and $f_{1}$ must be constant and are in fact identically zero because $f_{i}(\hat{o})=0$. 
\end{proof}

As $D_{H,t} $ is equivariant, we can form bundle maps $\tilde{D}_{H,t} $ and obtain a c-homotopy.

\begin{pro}\label{changing boundary condition gives c-homotopy}
For any $t \in [0,1]$, the triple $$((\tilde{l}_{X_{0},X_{1}},\tilde{D}_{Y_{2}},\tilde{D}_{H,t}),(\tilde{c}_{X_{0},X_{1}},0,0),(\tilde{r}_{0},\tilde{r}_{1}))$$
is an SWC-triple. Consequently, this provides a c-homotopy between the triples at $t=0$ and $t=1$.
\end{pro}
\begin{proof} The statement for $t=0$ follows from Lemma~\ref{lem SWCtripleDH}. For each element in the kernel of  
$(\tilde{l}_{X_{0},X_{1}},\tilde{D}_{Y_{2}},\tilde{D}_{H,t})$ there is a unique gauge transformation to an element in the kernel of $(\tilde{l}_{X_{0},X_{1}},\tilde{D}_{Y_{2}},\tilde{D}_{H,0})$ as a result of Lemma~\ref{laplace with mixed bundary condition}. This provides a linear bijection, so the kernel of $(\tilde{l}_{X_{0},X_{1}},\tilde{D}_{Y_{2}},\tilde{D}_{H,t})$ is also finite-dimensional.

The map $(\tilde{l}_{X_{0},X_{1}},\tilde{D}_{Y_{2}},\tilde{D}_{H,t})$ differs from the map $(\tilde{l}_{X_{0},X_{1}},\tilde{D}_{Y_{2}},\tilde{D}_{H,0})$ only at the $\Omega_{0}^{0}(Y_{2})$-component. By Lemma~\ref{laplace with mixed bundary
condition}, the map $\rho_t$ is surjective, so the map $(\tilde{l}_{X_{0},X_{1}},\tilde{D}_{Y_{2}},\tilde{D}_{H,t})$ is surjective on the $\Omega_{0}^{0}(Y_{2})$-component. This implies that the cokernels at each $t$ are in fact the same. 
Therefore, $(\tilde{l}_{X_{0},X_{1}},\tilde{D}_{Y_{2}},\tilde{D}_{H,t})$ are Fredholm.

Applying Lemma~\ref{laplace with mixed bundary condition} again,
one can see that there is a unique gauge transformation from a solution of 
$((\tilde{l}_{X_{0},X_{1}},\tilde{D}_{Y_{2}},\tilde{D}_{H})$, $(\tilde{c}_{X_{0},X_{1}},0,0)) $
to a solution of
$((\tilde{l}_{X_{0},X_{1}},\tilde{D}_{Y_{2}},\tilde{D}_{H,t})$, $(\tilde{c}_{X_{0},X_{1}},0,0))$
which depends continuously.
This provides a homeomorphism between them. Then the boundedness result follows from the case $t=0$ and compactness of
$[0,1]$.

\end{proof}

%
%

\subsubsection*{Step 6} Here, we will basically change the action of $G$ by identifying it with a group of harmonic gauge transformations with different boundary conditions. Recall from our setup  that $\tau (h) $ for $h \in H^{1}(X,Y_{2};\mathbb{R})$
is the unique harmonic 1-form on $X$ representing $h$ such that $\mathbf{t}_{Y_{2}}(\tau(h))$ is exact
and $\tau(h)\in
i\Omega_{CC}^{1}(X)$.  Note that for $t \in [0,1]$,  
\[
    D_{H, t}(\tau(h)|_{X_{0}},\tau(h)|_{X_{1}})=(0,2t\bar{d}^{*}(\mathbf{t}_{Y_{2}}(\tau(h)))).
\]
We put
$$ 
(\xi_{0,t}(h),\xi_{1,t}(h)) :=\rho^{-1}_{t}(2t\bar{d}^{*}(\mathbf{t}_{Y_{2}}(\tau(h)))).
$$
We then apply gauge transformation to $\tau(h) $ and define 
\begin{align*} 
&\tau_{t}=(\tau_{X_{0},t},\tau_{X_{1},t}) \colon H^{1}(X,Y_{2};\mathbb{R}) \rightarrow
\Omega^{1}_{h}(X_{0})\times \Omega^{1}_{h}(X_{1}) \\
& h \mapsto ( \tau(h)|_{X_{0}}-d\xi_{0,t}(h) , \tau(h)|_{X_{1}}-d\xi_{1,t}(h) ).
 \end{align*}
From our construction, we have $D_{H,t} (\tau_{t}(h))=0$ and $d\xi_{i,0} =0 $.  

%
%

We will consider  harmonic gauge transformations corresponding to boundary condition $D_{H,t}=0 $.
For $h\in G$, we define 
$$u_{t}(h) :=(u_{X_{0},t}(h),u_{X_{1},t}(h))$$
such that $u_{X_{i},t}(h)$ is the unique gauge transformation on $X_i$ satisfying  
$$u_{X_{i},t}(h)(\hat{o})=1,\ u_{X_{i},t}^{-1}du_{X_{i},t}=\tau_{X_{i},t}(h).$$
Notice that for $u_{X_i , 0} $ is the restriction of $u \in \mathcal{G}^{h,o}_{X,Y_{2}}$ and  $u_{X_i , t}(h) = e^{-\xi_{i,t}(h)} u_{X_{i},0}(h) $.

Consider a new action $\varphi_{t}$ of $G$ on the spaces $V_{X_{i}},U_{X_{i}}$, $H^{1}(X_{i},Y_{2};\mathbb{R})$, $Coul(Y_{i})$
and $H$ such that the action on spinors is given by the gauge transformations
$(u_{X_{0},t}(h),u_{X_{1},t}(h))$ instead of restriction of $u \in \mathcal{G}^{h,o}_{X,Y_{2}}$.
We also consider a map 
$$l^{t}_{X_{0},X_{1}},c^{t}_{X_{0},X_{1}} \colon V_{X_{0}}\times V_{X_{1}}\times H^{1}(X,Y_{2};\mathbb{R})\rightarrow U_{X_{0}}\times
U_{X_{1}}\times H^{1}(X,Y_{2};\mathbb{R})$$
by replacing the term $\tau(h)|_{X_{i}}$ in the definition (cf. (\ref{linear part of SW map}))  with $\tau_{X_{i},t}(h)$.

 It is not hard to check that the maps
$l^{t}_{X_{0},X_{1}},\ c^{t}_{X_{0},X_{1}},\ D_{Y_{2}}\times \operatorname{id}_{H^{1}(X,Y_{2};\mathbb{R})}$ and $D_{H,{t}}\times
\operatorname{id}_{H^{1}(X,Y_{2};\mathbb{R})}$ are all equivariant under the action $\varphi_{t}$. By taking quotients, we
obtain bundles
$$\tilde{V}^{t}_{X_{0},X_{1}}:=(V_{X_{0}}\times V_{X_{1}}\times H^{1}(X,Y_{2};\mathbb{R}))/(G,\varphi_{t});$$ $$\tilde{U}^{t}_{X_{0},X_{1}}:=(U_{X_{0}}\times
U_{X_{1}}\times H^{1}(X,Y_{2};\mathbb{R}))/(G,\varphi_{t})$$
and bundle maps $\tilde{l}^{t}_{X_{0},X_{1}},\ \tilde{c}^{t}_{X_{0},X_{1}} , \tilde{D}_{Y_{2},t},\tilde{D}_{H,t}$.
We can consider an obvious bundle isomorphism from $\tilde{V}^{}_{X_{0},X_{1}}$ (resp.  $\tilde{U}^{}_{X_{0},X_{1}}$) to $\tilde{V}^{t}_{X_{0},X_{1}}$  (resp.  $\tilde{U}^{}_{X_{0},X_{1}}$) by sending $ (a_i,\phi_i,h)$ to $(a_i,e^{\xi_{i,t}(h)}\phi_i,h) $. All of the above maps fit in a commutative diagram.
\begin{equation*}
\xymatrix{
\tilde{V}^{}_{X_{0},X_{1}} \ar[d]_{} \ar[r]^{} & \tilde{U}^{}_{X_{0},X_{1}} \ar[d]^{}\\
\tilde{V}^{t}_{X_{0},X_{1}} \ar[r]^{} & \tilde{U}^{t}_{X_{0},X_{1}}}
\end{equation*}
We can conclude:

\begin{lem} The triple
 $((\tilde{l}^{1}_{X_{0},X_{1}},\tilde{D}^{1}_{Y_{2}},\tilde{D}^{1}_{H}),(\tilde{c}^{1}_{X_{0},X_{1}},0,0),(\tilde{r}_{0},\tilde{r}_{1}))$
 is an SWC triple and is c-homotopic to the triple $$((\tilde{l}_{X_{0},X_{1}},\tilde{D}_{Y_{2}},\tilde{D}_{H,1}),(\tilde{c}_{X_{0},X_{1}},0,0),(\tilde{r}_{0},\tilde{r}_{1})).$$
\end{lem}

Let us take a closer look at the SWC triple $$((\tilde{l}^{1}_{X_{0},X_{1}},\tilde{D}^{1}_{Y_{2}},\tilde{D}^{1}_{H}),(\tilde{c}^{1}_{X_{0},X_{1}},0,0),(\tilde{r}_{0},\tilde{r}_{1})).$$
Observe that the boundary condition $b_0 - b_1 = 0$ and $\bar{d}^* (b_0 + b_1 ) = 0$ implies $b_0 = b_1 =0 $.
This allows us to recover the double Coulomb condition on $X_i$.

\begin{lem}\label{kernel equals double coulomb slice}
The operator
$$ (d^*_{X_0} , d^*_{X_1} , D_{H,1}) \colon V_{X_{1}}\times V_{X_{2}}\rightarrow L^{2}_{k-1/2}(i\Omega^{0}(X_{0})\oplus i\Omega^{0}(X_{1}))\oplus H$$
is surjective and its kernel can be written as
$$
L^{2}_{k+1/2}(i\Omega_{CC}^{1}(X_{0},\alpha^{0}\cup \beta)\oplus \Gamma(S^{+}_{X_{0}}))\times L^{2}_{k+1/2}(i\Omega_{CC}^{1}(X_{1},\alpha^{1})\oplus
\Gamma(S^{+}_{X_{1}})).
$$
\end{lem}
\begin{proof} We consider a pair of exact forms $(df_0 , df_1)$. 
Then, surjectivity reduces to finding a solution of Poisson equation with Dirichlet boundary condition on $X_0$ and $X_1$.

\end{proof}

Note that we can identify \begin{equation}L^{2}_{k+1/2}(i\Omega_{CC}^{1}(X_{0},\alpha^{0}\cup \beta)\oplus \Gamma(S^{+}_{X_{0}}))\times
H^{1}(X_{0},Y_{2};\mathbb{R})\cong Coul^{CC}(X_{0},\beta)\end{equation}
by sending $((\hat{a}_{0},\phi),h)$ to $(\hat{a}_{0}+\hat{a}_{h},\phi)$, where $\hat{a}_{h}$ is the element in $\mathcal{H}^{1}_{DC}(X_{0})$
corresponding to $h$ (cf. (\ref{eq H1DC}) from Section~\ref{sec 4mfd}). Under this identification, the natural projection
to $H^{1}(X_{0},Y_{2};\mathbb{R})$ becomes the map $\hat{p}_{\alpha , X_0} $ (cf. (\ref{eq defp_alpha})).
Similarly, we have an isomorphism 

 \begin{equation}L^{2}_{k+1/2}(i\Omega_{CC}^{1}(X_{1},\alpha^{1})\oplus \Gamma(S^{+}_{X_{0}}))\times H^{1}(X_{1},Y_{2};\mathbb{R})\cong
Coul^{CC}(X_{1}).\end{equation}
As a result, the action $\varphi^{1}$ provides an action on $Coul^{CC}(X_{0},\beta)\times Coul^{CC}(X_{1})$
via an identification 
$$G=H^{1}(X_{0},Y_{2})\times H^{1}(X_{1},Y_{2})\cong \mathcal{G}^{h,\hat{o}}_{X_{0},\partial X_{0}}\times\mathcal{G}^{h,\hat{o}}_{X_{1},\partial
X_{1}}.$$ 
This holds because $Y_0 $ and $Y_1$ are homology spheres.

%

As in Section~\ref{sec 4mfd}, we have Seiberg--Witten maps 
\begin{align*}
\widebar{SW}_{X_{0}} =\bar{L}_{X_{0}} + \bar{Q}_{X_{0}} \colon &Coul^{CC}(X_{0},\beta)/\mathcal{G}^{h,\hat{o}}_{X_{0},\partial X_{0}}\rightarrow\\& (L^{2}_{k-1/2}(i\Omega^{+}_{2}(X_{0})\oplus \Gamma(S^{-}_{X_{0}}))\times  \mathcal{H}^1_{DC}(X_0))/\mathcal{G}^{h,\hat{o}}_{X_{0},\partial
X_{0}}, \\
\widebar{SW}_{X_{1}} =\bar{L}_{X_{1}} + \bar{Q}_{X_{1}} \colon &Coul^{CC}(X_{1})/\mathcal{G}^{h,\hat{o}}_{X_{1},\partial
X_{1}}\rightarrow\\& (L^{2}_{k-1/2}(i\Omega^{+}_{2}(X_{1})\oplus \Gamma(S^{-}_{X_{1}}))\times  \mathcal{H}^1_{DC}(X_1))/\mathcal{G}^{h,\hat{o}}_{X_{1},\partial
X_{1}}.
\end{align*}
Since an element of $\mathcal{G}^{h,\hat{o}}_{X_{i},\partial X_{i}}$ takes value $1$ on $Y_{2}$, there are
well-defined restriction maps $r_2$ from $Coul^{CC}(X_{0},\beta)/\mathcal{G}^{h,\hat{o}}_{X_{0},\partial X_{0}}\text{ and }  Coul^{CC}(X_{1})/\mathcal{G}^{h,\hat{o}}_{X_{1},\partial
X_{1}}$ to $Coul(Y_{2}) $.
We then consider a map 
\begin{align*}
\bar{D}_{Y_{2}} \colon Coul^{CC}(X_{0},\beta)/\mathcal{G}^{h,\hat{o}}_{X_{0},\partial X_{0}}\times Coul^{CC}(X_{1})/\mathcal{G}^{h,\hat{o}}_{X_{1},\partial
X_{1}} &\rightarrow Coul(Y_{2}) \\
(x_{0},x_{1}) &\mapsto r_{2}(x_{0})-r_{2}(x_{1}).
\end{align*}

\begin{cor}
The triple $((\bar{L}_{X_{0}},\bar{L}_{X_{1}},\bar{D}_{Y_{2}}),(\bar{Q}_{X_{0}},\bar{Q}_{X_{1}},0),(\tilde{r}_{0},\tilde{r}_{1}))$
is an SWC triple stably c-homotopic to the triple $$((\tilde{l}^{1}_{X_{0},X_{1}},\tilde{D}^{1}_{Y_{2}},\tilde{D}^{1}_{H}),(\tilde{c}^{1}_{X_{0},X_{1}},0,0),(\tilde{r}_{0},\tilde{r}_{1})).$$
\end{cor}
\begin{proof} This follows by applying Lemma~\ref{moving map to  domain2} to the 
triple $$((\tilde{l}^{1}_{X_{0},X_{1}},\tilde{D}^{1}_{Y_{2}},\tilde{D}^{1}_{H}),(\tilde{c}^{1}_{X_{0},X_{1}},0,0),(\tilde{r}_{0},\tilde{r}_{1}))$$  with $g = (d^*_{X_0} , d^*_{X_1} , D_{H,1})$ as in Lemma~\ref{kernel equals double coulomb slice}.
\end{proof}
\subsubsection*{Step 7} This is the final step.
Recall from Section~\ref{sec gluingsetup} that we chose finite dimensional subspaces $U^i_n $ of 
$L^{2}_{k-1/2}(i\Omega^{+}_{2}(X_{i})\oplus \Gamma(S^{-}_{X_{i}}))
$
and eigen-spaces $V^i_n $ of $Coul(Y_i) $.
In the SWC construction of the triple $$((\bar{L}_{X_{0}},\bar{L}_{X_{1}},\bar{D}_{Y_{2}}),(\bar{Q}_{X_{0}},\bar{Q}_{X_{1}},0),(\tilde{r}_{0},\tilde{r}_{1})),$$ the subbundles involved are preimages of the map $(\bar{L}_{X_{0}},\bar{L}_{X_{1}},\bar{D}_{Y_{2}},p^{\mu_n}_{-\infty} \circ \tilde r_0 , p^{\mu_n}_{-\infty} \circ \tilde r_1 )$ rather than preimages of the product map $(\bar{L}_{X_{0}}, p^{\mu_n}_{-\infty}\circ \tilde r_0 , p^{\mu_n}_{-\infty} \circ \tilde r_2)  \times (\bar{L}_{X_{1}}, p^{\mu_n}_{-\infty} \circ \tilde r_1 , p^\infty_{-\mu_n} \circ \tilde r_2)$ in the construction of relative Bauer--Furuta invariants.
Note that there is a choice of trivialization but we do not emphasize this here.

Using the spectral decomposition, we see that $r_2 (x_0) - r_2(x_1) \in V^{\mu_n}_{-\mu_n} $ if and only if
\begin{align*}
p^{\infty}_{\mu_{n}}\circ r_{2}(x_{0}) &=  p^{\infty}_{\mu_{n}}\circ r_{2}(x_{1}),\\ 
p^{-\mu_{n}}_{-\infty}\circ r_{2}(x_{1}) &= p^{-\mu_{n}}_{-\infty}\circ r_{2}(x_{0}).
\end{align*}
We introduce a family of subbundles : for $t \in [0,1]$, 
\begin{align*}
& W_{X_{0},X_{1}}^{n,t}:=  \\
& \quad \{ (x_{0},x_{1})\in (Coul^{CC}(X_{0},\beta)/\mathcal{G}^{h,\hat{o}}_{X_{0},\partial X_{0}})\times
(Coul^{CC}(X_{1})/\mathcal{G}^{h,\hat{o}}_{X_{1},\partial X_{1}})\mid \\  
& \qquad p^{\mu_n}_{-\infty}\tilde{r}_{i}(x_{i})\in V^i_n,\
\bar L_{X_{i}}(x_{i})\in U^i_{n},  \\
& \qquad  p^{\infty}_{\mu_{n}}\circ r_{2}(x_{0})= t p^{\infty}_{\mu_{n}}\circ
r_{2}(x_{1}), \\ 
& \qquad p^{-\mu_{n}}_{-\infty}\circ r_{2}(x_{1})= t   p^{-\mu_{n}}_{-\infty}\circ r_{2}(x_{0}) \}.
\end{align*}



We have a boundedness result for this family.
\begin{lem}
 For any $R>0$, there exist $N,\epsilon_{0}$ with the following significance: For any $n>N, t \in [0,1]$,
$(x_{0},x_{1})\in B^{+}(W^{n,\theta}_{X_{0},X_{1}},R)$ and $\gamma_{i} \colon (-\infty,0]\rightarrow B(V^{\lambda_{n}}_{-\lambda_{n}}(Y_{i}),R)$
where  $i=0,1$   satisfying
\begin{itemize}
\item $\|p^{\mu_{n}}_{-\mu_{n}}(r_{2}(x_{0})-r_{2}(x_{1}))\|_{L^{2}_{k}}\leq \epsilon$,\item $\|p_{U^i_{n}}\circ \widebar{SW}_{X_i} (x_{i})\|_{L^{2}_{k-1/2}}\leq \epsilon$,
\item $\gamma_{i}$ is an approximated trajectory with $\gamma_{i}(0)=p^{\mu_{n}}_{-\mu_{n}}\circ \tilde{r}_{i}(x_{i}),$
\end{itemize}
one has $\|x_{i}\|_{F}\leq R_{3}+1$ and $\|\gamma_{i}(t)\|_{L^{2}_{k}}\leq R_{3}+1$, where $R_{3}$ is the constant in Proposition~\ref{bounded
for gluing}. 
\end{lem}

\begin{proof} The proof is essentially identical to Proposition~\ref{bounded for gluing, approximated} by using \cite[Lemma 1]{Manolescu2}
to control $\|p^{\infty}_{\mu_{n}}\circ r_{2}(x_{0})\|_{L^{2}_{k}}$ (resp. $\|p_{-\infty}^{-\mu_{n}}\circ r_{2}(x_{1})\|_{L^{2}_{k}}$)
in terms of $\|\bar{L}_{X_{0}}(x_{0})\|_{L^{2}_{k-1/2}}$ (resp. $\|\bar{L}_{X_{1}}(x_{1})\|_{L^{2}_{k-1/2}}$).
\end{proof}

As a result,
we obtain a family of maps as $t \in [0 , 1]$.
When $t=1$, this is the same as the SWC construction
for the original triple. When  $t=0$, we have 
$$
W^{n,0}_{X_{0},X_{1}}=W^0_{n, \beta}\times W^1_{n}
$$ 
and we then recover the homotopy class in Proposition~\ref{deformed pairing}. The proof of the gluing theorem is finished.


\end{document}